\documentclass[12]{amsart}
\usepackage[T1]{fontenc}
\usepackage[francais, english]{babel}
\usepackage{amsmath,amssymb,hyperref}
\usepackage{color,xcolor}
\usepackage[all]{xy}
\usepackage{geometry} 
\usepackage{array,multirow,makecell}
\usepackage{pst-all}
\usepackage{graphicx}

\newtheorem{thmA}{Theorem}
\newtheorem{thm}{Theorem}

\newtheorem{cor}[thm]{Corollary}
\newtheorem{prop}[thm]{Proposition}
\newtheorem{lem}[thm]{Lemma}
\newtheorem{definition}[thm]{Definition}
\newtheorem{remark}[thm]{Remark}
\newtheorem{example}[thm]{Example}

\newtheorem{theorem}{Theorem}
\newtheorem{corollary}[theorem]{Corollary}

\newcommand{\OM}[1]{\Omega^1_{#1}}

\DeclareMathOperator{\elm}{elm}

\DeclareMathOperator{\tr}{tr}

\DeclareMathOperator{\length}{length}

\DeclareMathOperator{\Res}{Res}

\def\O{\mathcal O}

\def\CON{\mathfrak{Con}}
\def\BUN{\mathfrak{Bun}}

\def\OM0{\mathbf{\Omega}_0}

\title[Irregular parabolic bundles]{Moduli space of irregular rank two parabolic bundles over the Riemann sphere 
and its compactification}
\author[A. Komyo, F. Loray and M.-H.\ Saito]{Arata Komyo \and Frank Loray \and Masa-Hiko Saito}
\address{Center for Mathematical and Data Sciences, Kobe Univ.,
1-1 Rokkodai-cho, Nada-ku, Kobe, 657-8501, Japan}
\email{akomyo@math.kobe-u.ac.jp}
\address{Univ Rennes, CNRS, IRMAR - UMR 6625, F-35000 Rennes, France}
\email{frank.loray@univ-rennes1.fr}
\address{Department of Mathematics, Graduate School of Science, Kobe University, 1-1 Rokkodai-cho, Nada-ku, Kobe, 657-8501, Japan and 
\newline
\quad Faculty of Business Administration, Kobe Gakuin University, 1-1-3, Minatojima, Chuo-ku, Kobe, 650-8586, Japan}
\email{mhsaito@ba.kobegakuin.ac.jp}
\thanks{The first author is supported by JSPS KAKENHI: 
Grant Numbers JP17H06127 and JP19K14506.
The second author is supported by CNRS,  ANR-16-CE40-0008 project ``Foliage'' and Henri Lebesgue Center, program ANR-11-LABX-0020-0.
The third author is supported by JSPS KAKENHI: 
Grant Number JP17H06127, JP22H00094.}
\date{\today}

\numberwithin{equation}{section}

\begin{document}

\maketitle

\begin{abstract}
In this paper, we study
rank 2 (quasi) parabolic bundles over the Riemann sphere with an effective divisor
and these moduli spaces.
First we consider a criterium when a parabolic bundle admits 
a unramified irregular singular parabolic connection.
Second, to give a good compactification of 
the moduli space of semistable parabolic bundles,
we introduce a generalization of parabolic bundles,
which is called refined parabolic bundles.
Third, we discuss a stability condition of refined parabolic bundles and
define elementary transformations of the refined parabolic bundles.
Finally, we describe the moduli spaces of refined parabolic bundles
when the dimensions of the moduli spaces are two.
These are related to geometry of some weak del Pezzo surfaces.
\end{abstract}

\section{Introduction}

In this paper, we investigate the moduli space of rank 2 
(quasi) parabolic bundles over the Riemann sphere with an effective divisor $D$
and extend some results obtained in \cite{LS}
 when the effective divisor $D$ is reduced.
Let $r$ and $\nu$ be positive integers and $C$ be a smooth projective curve over $\mathbb{C}$.
Fix points $t_1,t_2,\ldots,t_{\nu}$ on $C$,
and set $D = n_1 [t_1] + \cdots + n_{\nu} [t_{\nu}]$.
In general, 
a parabolic bundle $(E,\bold{l})$ over $(C,D)$
is a pair where
 $E$ is a rank $r$ vector bundle over $C$,
 $\bold{l}=\{l^{j}_i\}^{0\leq j \leq r-1}_{1 \leq i \leq \nu}$, 
 which are filtrations 
$E|_{n_i [t_i]} =l^0_i \supset l^1_i \supset \cdots \supset l^{r-1}_i \supset l^r_i =0$ 
by free $\mathcal{O}_{n_i [t_i]}$-modules 
such that $l^j_i/l^{j+1}_i \cong \mathcal{O}_{n_i [t_i]}$ for any $i$, $j$.

Parabolic (vector) bundles over an irreducible smooth complex projective curve with 
a reduced effective divisor 
were introduced in Mehta--Seshadri \cite{MS}.
To construct a good moduli space, 
a stability condition for the parabolic weights $\boldsymbol{\alpha}$
was introduced.
The moduli spaces of 
$\boldsymbol{\alpha}$-semistable parabolic bundles were constructed
in \cite{MS}, \cite{Bh1}, and \cite{BB}.
The effect on the moduli space of varying the parabolic weights was studied in 
\cite{BH}, \cite{Bauer}, \cite{Mukai}, and \cite{MY}.
This effect is interesting from the point of view of 
the birational geometry (in the sense of Mori's program).
A generalization of parabolic bundles was studied in \cite{Bh2} and \cite{Bh3}.  
Here the (generalized) parabolic bundles are defined over 
an irreducible smooth complex projective curve with 
an effective divisor, which is not necessary reduced.
A stability condition of such parabolic bundles was introduced and 
the moduli spaces of 
$\boldsymbol{\alpha}$-semistable parabolic bundles were constructed in \cite{Bh3}.
Finally a more general notion of parabolic bundles  
was studied by Yokogawa (\cite{Yoko1} and \cite{Yoko2}).

Usually, a parabolic bundle consists of a vector bundle $E$, filtrations 
$\bold{l}=\{l^j_i\}^{0\leq j \leq r-1}_{1 \leq i \leq \nu}$,
and parabolic weights $\boldsymbol{\alpha}$. 
But in this paper, we call a pair $(E,\bold{l})$ parabolic bundle
and call the filtration $\bold{l}$ a parabolic structure.
Parabolic structures are important 
for study of geometry of isomonodromic deformations 
(or Painlev\'e type equations). 
For example, parabolic connections were studied in \cite{IIS}, \cite{IIS2}, and \cite{Ina}. 
A parabolic connection is a pair of a connection 
and a parabolic structure, which is compatible with the connection. 
When the effective divisor is reduced,
a stability condition for the parabolic weights $\boldsymbol{\alpha}$
was introduced in \cite{IIS} and \cite{Ina}.
Now we fix generic $\boldsymbol{\alpha}$ 
so that $\boldsymbol{\alpha}$-stable = $\boldsymbol{\alpha}$-semistable.
The moduli spaces of $\boldsymbol{\alpha}$-stable parabolic connections
were constructed in \cite{IIS} and \cite{Ina}.
The cases where the underlying effective divisors are not necessary reduced 
were studied in \cite{IS}.
In the cases,
as in \cite{IS}, we call the parabolic connections
{\it unramified irregular singular parabolic connections}. 
A stability condition for the parabolic weights $\boldsymbol{\alpha}$
was also introduced and
the moduli spaces of $\boldsymbol{\alpha}$-stable 
unramified irregular singular parabolic connections
were constructed in \cite{IS}.
These moduli spaces play important roles to prove
the geometric Painlev\'e property of 
the isomonodromic deformations of these connections (\cite{IIS}, \cite{Ina}, and \cite{IS}).
To study geometry of the moduli spaces of parabolic connections
(or unramified irregular singular parabolic connections),
the moduli spaces of parabolic bundles are useful. 
Indeed, we have a correspondence from 
(unramified irregular singular) parabolic connections
to parabolic bundles by forgetting the connections.
Remark that the corresponding parabolic bundles are not necessary 
$\boldsymbol{\alpha}$-stable if
(unramified irregular singular) parabolic connections
are $\boldsymbol{\alpha}$-stable.
But, 
this correspondence induces maps from 
Zariski open subsets on the moduli spaces of connections
to the moduli spaces of $\boldsymbol{\alpha}$-stable parabolic bundles.
We may expect that these maps have nice properties
(for example, these are affine bundles and Lagrangian fibrations).

In this paper, we impose that the ranks of underlying vector bundles
are 2
and base curves are the projective line $\mathbb{P}^1$.
That is, 
our parabolic bundle $(E,\bold{l})$ 
is a pair where
 $E$ is a rank 2 vector bundle over $\mathbb{P}^1$,
 $\bold{l}=\{l_i\}_{1 \leq i \leq \nu}$, and
$l_{ i }\subset E|_{n_i [t_i]}$ is a free $\mathcal{O}_{n_i [t_i]}$-submodule of
length $n_i$ for each $i$.
So we will consider only special situations.
But 
isomonodromic deformations in this situation are interesting. 
Indeed, these isomonodromic deformations
correspond to 
\begin{itemize}
\item Painlev\'e equation $P_{VI}$ when $D=[t_1]+[t_2]+[t_3] +[t_4]$;
\item Garnier systems when $\deg(D) \geq 5$ and 
$D$ is reduced;
\item Painlev\'e equations $P_{V}, P_{IV}, P_{III(D_6^{(1)})}$ and $P_{II}$ 
when  $D=[t_1]+[t_2]+2[t_3]$,
 $D=[t_1]+3[t_2]$,
  $D=2[t_1]+2[t_2]$, and
   $D=4[t_1]$, respectively; and
\item some irregular Garnier systems
when $\deg(D) \geq 5$ and $D$ is not reduced. 
\end{itemize}
On the other hand, the special situation was studied 
in the point of view of the geometric Langlands problem
\cite{AL}, \cite{AF}, and \cite{DP}.

\subsection{When the effective divisor is reduced}

Our purpose is to extend some result on parabolic bundles 
in the cases where $D$ is reduced.
Now we recall some results on parabolic bundles with a reduced effective divisor 
$D=[t_1]+\cdots+[t_{\nu}]$,
which are concerned with the contents of the present paper.
Let $\Lambda=(\lambda_i^+,\lambda^-_i)_{1\leq i \leq \nu}$
be a tuple of complex numbers such that 
$\sum_{i=1}^{\nu} (\lambda_i^++\lambda^-_i) +d =0$.
Here we impose $\sum_{i=1}^{\nu} \lambda_i^{\epsilon_i} \not\in \mathbb{Z} $
for any choice $(\epsilon_i)_{1\leq i \leq \nu} \in \{ +,-\}$.
We say $\Lambda$ is generic if $\Lambda$ satisfies this condition.
We define a $\Lambda$-parabolic connection as
a parabolic connection such that 
the two residual eigenvalues of the connection are $\lambda_i^+$ and $\lambda_i^-$ for each $t_i$ 
and the parabolic structure satisfies the compatibility condition with the connection and $\lambda_i^+$
(for details, see the paragraph after the proof of Lemma \ref{Prop:GenericIrred}
with $n_i =1$ for any $i$).
We say a parabolic bundle is $\Lambda$-flat if it admits a 
$\Lambda$-parabolic connection.
We have a criterium of $\Lambda$-flatness given by \cite[Proposition 3]{AL}: 
\begin{equation}\label{eq:2021.12.30.21.35}
\mathrm{End} (E,\bold{l})=\mathbb{C}
\, \Longleftrightarrow \, 
\text{$(E,\bold{l})$ is $\Lambda$-flat}
\, \Longleftrightarrow \, 
\text{$(E,\bold{l})$ is undecomposable}.
\end{equation}
Let $\mathfrak{Bun}(D,d)$ be the
coarse moduli space of undecomposable parabolic bundles of degree $d$ over $(\mathbb{P}^1,D)$.
By the flatness criterium, the image of the moduli space of 
$\Lambda$-parabolic connections
under the forgetful map is just $\mathfrak{Bun}(D,d)$. 
However, as observed in \cite{AL}, \cite{AL2} and \cite{LS}, 
the coarse moduli space $\mathfrak{Bun}(D,d)$ is a non Hausdorff topological space, 
or a nonseparated scheme.
We introduce a stability condition with respect to weights 
$\boldsymbol{w}\in (w_i) \in [ 0,1]^{\nu}$ (as in \cite[Definition 2.2]{LS}).
The moduli space $\mathfrak{Bun}_{\boldsymbol{w}}(D,d)$ 
of $\boldsymbol{w}$-semistable parabolic bundles of degree $d$
is a normal irreducible projective variety.
The open subset of $\boldsymbol{w}$-stable bundles is smooth.
For generic weights $\boldsymbol{w}$, 
$\boldsymbol{w}$-semistable = $\boldsymbol{w}$-stable 
and we get a smooth projective variety.
If $\mathfrak{Bun}_{\boldsymbol{w}}(D,d)$ is not empty,
the dimension of $\mathfrak{Bun}_{\boldsymbol{w}}(D,d)$ is $n-3$.
On the other hand, we have a following fact given by \cite[Proposition 3.4]{LS}:
\begin{equation}\label{eq:2021.12.31.11.56}
\text{$(E,\bold{l})$ is undecomposable}
\, \Longleftrightarrow \, 
\text{$\boldsymbol{w}$-stable for a convenient choice of weights $\boldsymbol{w}$}.
\end{equation}
As a consequence, 
the coarse moduli space $\mathfrak{Bun}(D,d)$ of undecomposable parabolic bundles 
of degree $d$ is covered by 
those smooth projective varieties $\mathfrak{Bun}_{\boldsymbol{w}}(D,d)$ 
when $\boldsymbol{w}$ runs.
When $\deg(D)=4$ or $\deg(D)=5$, 
the dimension of $\mathfrak{Bun}_{\boldsymbol{w}}(D,d)$ is $1$ or $2$,
respectively. 
In these cases,
we can give explicit description of $\mathfrak{Bun}_{\boldsymbol{w}}(D,d)$ 
for some weights $\boldsymbol{w}$.
We fix a real number $w \in [0,1]$.
We impose that
weights $\boldsymbol{w}$ satisfy 
the condition $w_1=w_2=\cdots=w_{\nu}=w$.
We call weights which satisfy this condition {\it democratic weights}.
Now,
we assume that $\deg(D)=4$ and $d=-1$.
Then we have 
\begin{itemize}
\item $\mathfrak{Bun}_{\boldsymbol{w}}(D,-1)$ is empty when $0<w <\frac{1}{4},
\frac{3}{4}<w <1$;
\item $\mathfrak{Bun}_{\boldsymbol{w}}(D,-1)=\mathbb{P}^1$ when $\frac{1}{4}<w <\frac{1}{2}$;
\item $\mathfrak{Bun}_{\boldsymbol{w}}(D,-1)=\mathbb{P}^1$ when $\frac{1}{2}<w <\frac{3}{4}$.
\end{itemize}
The coarse moduli space
$\mathfrak{Bun}(D,d)$ is constructed by gluing 
$\mathfrak{Bun}_{\boldsymbol{w}}(D,-1)\cong \mathbb{P}^1$ ($\frac{1}{4}<w <\frac{1}{2}$)
and 
$\mathfrak{Bun}_{\boldsymbol{w}}(D,-1)\cong \mathbb{P}^1$ ($\frac{1}{2}<w <\frac{3}{4}$)
outside $t_1, \ldots , t_4$ (for instance, see \cite[Section 3.7]{LS}).
Next, we assume that $\deg(D)=5$ and $d=-1$.
Then we have the following description of the moduli spaces 
$\mathfrak{Bun}_{\boldsymbol{w}}(D,-1)$:
\begin{itemize}
\item $\mathfrak{Bun}_{\boldsymbol{w}}(D,-1)$ is empty when $0<w <\frac{1}{5}$;
\item $\mathfrak{Bun}_{\boldsymbol{w}}(D,-1)=\mathbb{P}^2$ when $\frac{1}{5}<w <\frac{1}{3}$;
\item $\mathfrak{Bun}_{\boldsymbol{w}}(D,-1)$ is a del Pezzo surface of degree 4 
when $\frac{1}{3}<w <\frac{3}{5}$;
\item $\mathfrak{Bun}_{\boldsymbol{w}}(D,-1)$ is a del Pezzo surface of degree 5 
when $\frac{3}{5}<w <1$.
\end{itemize}
(See \cite[Section 6.1]{LS} and \cite[Proposition 3.11]{DP}).
When $d$ is even, 
descriptions of the moduli spaces are given by \cite[Theorem 3.2]{DP}.
To cover the coarse moduli space $\mathfrak{Bun}(D,-1)$,
thinking only democratic weights is not enough.
By applying elementary transformations on 
$\mathfrak{Bun}_{\boldsymbol{w}}(D,-1)$ with democratic weights $\boldsymbol{w}$,
we can cover the coarse moduli space $\mathfrak{Bun}(D,-1)$ by 
the moduli spaces $\mathfrak{Bun}_{\boldsymbol{w}}(D,-1)$.
Remark that the stability of parabolic bundles changes by 
applying the elementary transformations.
If $w_1=w_2=\cdots=w_{\nu}=\frac{1}{2}$,
the stability of parabolic bundles is preserved by 
applying the elementary transformations.
So some elementary transformations induce 
automorphisms on the moduli space $\mathfrak{Bun}_{\boldsymbol{w}}(D,d)$ with $w=\frac{1}{2}$.
In \cite{AFKM}, a modular interpretation 
for the automorphism group of the moduli space $\mathfrak{Bun}_{\boldsymbol{w}}(D,d)$
with any $\deg(D)$ and with any $\deg(E)$ is given by this point of view.

\subsection{Results}

Now we describe results of the present paper.
The effective divisor $D$ is not necessary reduced.
First we will discuss a counterpart of \eqref{eq:2021.12.30.21.35}.
Let $\Lambda=(\lambda^+,\lambda^-)$ where $\lambda^{\pm}\in\Omega^1(D)/\Omega^1$.
As in the reduced divisor cases, we impose that 
$\Lambda$ satisfies some conditions.  
We define a $\Lambda$-unramified irregular singular parabolic connection
as an unramified irregular singular parabolic connection 
such that 
the diagonalization of the principal parts of 
the connection at $D$
is a diagonal matrix whose diagonal entries are $\lambda^+$ and $\lambda^-$
and 
the parabolic structure satisfies the compatibility condition with the connection and
$\lambda^+$ 
(for details, see the paragraph after the proof of Lemma \ref{Prop:GenericIrred}).
We say a parabolic bundle is $\Lambda$-{\it flat} if it admits a 
$\Lambda$-unramified irregular singular parabolic connection.
If $\deg(D)=4$ and $\deg(E)$ is odd, then 
$$
\text{ $\mathrm{End}(E,\bold{l})=\mathbb C$ }
\Longleftrightarrow
\text{ $(E,\bold{l})$ is $\Lambda$-flat}
$$
which is proved in \cite[Proposition 4.8]{AF}.
We can define undecomposability for parabolic bundles 
when the effective divisor $D$ is not necessary reduced.
If $(E,\bold{l})$ is $\Lambda$-flat, then $(E,\bold{l})$ is undecomposable.
But the converse proposition is not true even when $\deg(D)=4$ and $\deg(E)$ is odd.
So
in addition to undecomposability, we assume the {\it admissibleness} 
for parabolic bundles, which is defined in 
Definition \ref{eq:2021.12.31.11.48} below.
Using some material of  \cite[Section 4]{AF} and ideas of \cite[Proposition 3]{AL}, 
we can show that  
\begin{equation}\label{2022.2.23.17.26}
\text{$(E,\bold{l})$ is simple}
\, \Longrightarrow \, 
\text{$(E,\bold{l})$ is $\Lambda$-flat}
\, \Longrightarrow \, 
\text{$(E,\bold{l})$ is undecomposable and admissible}.
\end{equation}
This is a counterpart of \eqref{eq:2021.12.30.21.35}.
Here we say $(E,\bold{l})$ is {\it simple} if 
$\mathrm{End} (E,\bold{l})=\mathbb{C}$.
Counterexamples of the converse of the first implication of \eqref{2022.2.23.17.26}
are described in
Proposition \ref{prop:2021.11.30.22.07_0} and  
Proposition \ref{prop:2021.11.30.22.07} below.
Counterexamples of the converse of the second implication of \eqref{2022.2.23.17.26}
also are described in
Proposition \ref{prop:2021.11.30.22.07_0} and  
Proposition \ref{prop:2021.11.30.22.07} (see also Example \ref{2021.1.12.12.45}).
If $\deg(D)=4$ and $\deg(E)$ is odd, then we have
\begin{equation}\label{2022.2.23.17.32}
\text{$(E,\bold{l})$ is undecomposable and admissible}
\, \Longrightarrow \, 
\mathrm{End} (E,\bold{l})=\mathbb{C}.
\end{equation}
The implications \eqref{2022.2.23.17.26} and \eqref{2022.2.23.17.32}
 are described in Proposition \ref{Prop:FlatnessImplications} below.
These are the first results.

Second we will discuss a counterpart of \eqref{eq:2021.12.31.11.56}.
We can introduce a stability condition with respect to weights 
$\boldsymbol{w}\in (w_i) \in [ 0,1]^{\nu}$, which is defined in Section \ref{section:StabilityOfRefined} below.
Let $\mathfrak{Bun}_{\boldsymbol{w}}(D,d)$ be 
the moduli space of $\boldsymbol{w}$-semistable parabolic bundles.
However $\mathfrak{Bun}_{\boldsymbol{w}}(D,d)$
is not projective in general,
since we impose that each $l_i$ is free in 
the definition of parabolic bundles.
So, to construct a good compactification, we will define {\it refined parabolic bundles},
which are generalization of parabolic bundles.
Let $l_{ i , \bullet} = \{ l_{i,k}\}_{1 \leq k \leq n_i}$ be a filtration 
$E|_{n_i [t_i]} \supset l_{i,n_i} \supset l_{i,n_i-1} \supset \cdots \supset 
l_{i,1} \supset 0$
of $\mathcal{O}_{n_i [t_i]}$-modules such that the length of $l_{i,k}$ is $k$.
We call $l_{ i , \bullet} = \{ l_{i,k}\}_{1 \leq k \leq n_i}$ a refined parabolic structure at $t_i$.
If $l_{i,n_i}$ is free, then the notion of refined parabolic structures coincides with 
the notion of parabolic structures.
Refined parabolic bundles are vector bundles with refined parabolic structures.
For weights 
$$\boldsymbol{w} 
= (\boldsymbol{w}_1,\ldots , \boldsymbol{w}_{\nu}), \quad 
\boldsymbol{w}_i= (w_{i,n_i},\ldots ,w_{i,1}) \in [0,1]^{n_i} \quad (1\leq i\leq \nu),
$$
where
$0 \leq w_{i,n_i} \leq \cdots  \leq w_{i,1} \leq 1$,
we can define $\boldsymbol{w}$-stability and 
$\boldsymbol{w}$-semistability of a refined parabolic bundle
(Definition \ref{2022.3.15.21.27} below).

\begin{remark}
A compactification of 
the moduli space of $\boldsymbol{\alpha}$-stable $\Lambda$-parabolic connections 
is constructed
in \cite[Proposition 5.6]{IIS}. 
This compactification is 
the moduli space of $\boldsymbol{\alpha}$-stable 
$\Lambda$-$\phi$-parabolic connections.
When $D$ is reduced, $\deg(D)=4$ and $\mathrm{rank}(E)=2$,
the compactification with the boundary divisor is isomorphic to
the Okamoto--Painlev\'e pair of $P_{VI}$ (\cite{IIS2}).
Here the notion of Okamoto-Painlev\'e pairs 
was introduced in \cite{STT}.
In \cite{Miya},
a compactification of 
the moduli space of $\boldsymbol{\alpha}$-stable 
$\Lambda$-unramified irregular singular parabolic connections
was studied when 
$D$ is not reduced, $\deg(D)=4$ and $\mathrm{rank}(E)=2$.
However, 
the moduli space of $\boldsymbol{\alpha}$-stable 
$\Lambda$-$\phi$-unramified irregular singular parabolic connections
is projective but is not smooth.
So, in \cite{Miya}, Miyazaki introduced refined parabolic structures,
which are same as our refined parabolic structures.
By considering refined parabolic structures, 
we have smooth compactifications,
which are isomorphic to 
the Okamoto--Painlev\'e pairs of $P_V$, $P_{IV}$, $P_{III(D_6^{(1)})}$, and $P_{II}$. 
\end{remark}

\begin{remark}
The notion of the refined parabolic bundles is contained in 
the notion of the parabolic bundles 
in Yokogawa's papers \cite{Yoko1} and \cite{Yoko2}.
\end{remark}

We can define undecomposable refined parabolic bundles
and admissible refined parabolic bundles, naturally.

\begin{remark}
We have defined the $\Lambda$-flatness 
for (ordinary) parabolic bundles.
On the other hand, we do not define the $\Lambda$-flatness
for refined parabolic bundles,
since unramified irregular singular parabolic connections
have (ordinary) parabolic structures.
\end{remark}

To give a counterpart of \eqref{eq:2021.12.31.11.56},
the conditions undecomposableness and admissibleness
are not sufficient.
That is, there exist 
undecomposable and admissible refined parabolic bundles
which is not $\boldsymbol{w}$-stable
for any weights $\boldsymbol{w}$ (Example \ref{2021.1.12.12.45} below).

\begin{remark}
Moreover, there exist examples of $\Lambda$-flat parabolic bundles,
which are $\boldsymbol{w}$-unstable for any weights $\boldsymbol{w}$.
\end{remark}

So we will define {\it tame} refined parabolic bundles
in Definition \ref{eq:2021.12.31.13.46}.
We will check that 
\begin{equation*}\label{eq:2021.12.31.16.07}
\text{$(E,\bold{l})$ is undecomposable and tame}
\, \Longrightarrow \, 
\text{$(E,\bold{l})$ is admissible}
\end{equation*}
(described in Proposition \ref{prop:2021.12.31.13.50} below).
Then we have a counterpart of \eqref{eq:2021.12.31.11.56} as follows:

\begin{thmA}[{Theorem \ref{thm:2021.10.27.13.44} and Corollary \ref{cor:2022.2.24.0.20}}]\label{2022_9_29_12_23}
Let $(E,\bold{l})$ be a refined parabolic bundle of degree $d$ over $(\mathbb{P}^1,D)$.
We have the following equivalence relation:
\begin{equation*}\label{eq:2021.12.31.16.07_2}
\text{$(E,\bold{l})$ is undecomposable and tame}
\, \Longleftrightarrow \, 
\text{$\boldsymbol{w}$-stable for a convenient choice of weights $\boldsymbol{w}$.}
\end{equation*}
Let $(E,\bold{l})$ be a
parabolic bundle of degree $d$ over $(\mathbb{P}^1,D)$. 
Then
\begin{equation*}
\text{$(E,\bold{l})$ is simple}
\, \Longleftrightarrow \, 
\text{$\boldsymbol{w}$-stable for a convenient choice of weights $\boldsymbol{w}$.}
\end{equation*}
\end{thmA}

Third, we will discuss elementary transformations 
of refined parabolic bundles.
We will define elementary transformations of 
refined parabolic bundles in Section \ref{section:2021.12.31.15.32}
and we will show some properties of the elementary transformations
in Section \ref{section:2021.12.31.15.34}.
For example, we will discuss a relation between 
the elementary transformations and 
the stability condition (Proposition \ref{prop:2021.10.21.15.22}).
 
Fourth, we will describe 
the moduli space of $\boldsymbol{w}$-stable refined parabolic bundles
when $\deg(D)=5$, $\deg(E)=1$,
and $\boldsymbol{w}$ is generic democratic weights.
Let $\overline{\mathfrak{Bun}}_{\boldsymbol{w}}(D,d)$ be 
the moduli space of $\boldsymbol{w}$-stable refined parabolic bundles of degree $d$. 
Here $\boldsymbol{w}$-semistable $=$ $\boldsymbol{w}$-stable,
since $\boldsymbol{w}$ is generic.
In this case, the dimension of $\overline{\mathfrak{Bun}}_{\boldsymbol{w}}(D,d)$
is $2$.
Since $\deg(D)=5$, the effective divisor $D$ is one of the following 
effective divisors:
\begin{equation*}
\begin{aligned}
D_{2111}:=2[t_1] +[t_2]+[t_3] +[t_4], & &
D_{221}:=2[t_1] +2[t_2]+[t_3], &&
D_{311}:=3[t_1] +[t_2]+[t_3],\\
D_{32}:=3[t_1] +2[t_2],&&
D_{41}:=4[t_1] +[t_2], && 
D_{5}:=5[t_1].
\end{aligned}
\end{equation*}
For these cases, by Theorem A, we can make a list of refined parabolic bundles 
which are $\boldsymbol{w}$-stable for a convenient choice of weights $\boldsymbol{w}$.
This list is described in the appendix.
We will define types of the refined parabolic bundles by some numerical data.
This type is related to behavior of the stability index when weights are varied.
By Proposition \ref{prop:2021.7.28.14.36} below,
we have that $\overline{\mathfrak{Bun}}_{\boldsymbol{w}}(D,1)=\mathbb{P}^2$
when $\frac{1}{5}<w <\frac{1}{3}$.
By this fact and the relation between the types and behavior of the stability index,
we have the following theorem.

\begin{thmA}[{Section \ref{section:2021.12.31.15.25}}]\label{2022.2.2.23.11}
Assume that $\deg(D) = 5$, $\deg(E) =1$, and weight $\boldsymbol{w} = (\boldsymbol{w}_i)_i$
satisfies $w_{i,1}=w_{i,2}=\cdots =w_{i,n_i}=w$ for any $i$ $(0<w<1)$, 
that is, $\boldsymbol{w}$ are democratic weights.

\begin{itemize}

\item[(i)]
We have the following explicit descriptions of $\overline{\mathfrak{Bun}}_{\boldsymbol{w}}(D,1)$:
\begin{itemize}
\item $\overline{\mathfrak{Bun}}_{\boldsymbol{w}}(D,1)$ is empty 
when $0<w <\frac{1}{5}$;
\item $\overline{\mathfrak{Bun}}_{\boldsymbol{w}}(D,1)=\mathbb{P}^2$ 
when $\frac{1}{5}<w <\frac{1}{3}$;
\item $\overline{\mathfrak{Bun}}_{\boldsymbol{w}}(D,1)$ is 
a weak del Pezzo surface of degree 4 
when $\frac{1}{3}<w <\frac{3}{5}$;
\item $\overline{\mathfrak{Bun}}_{\boldsymbol{w}}(D,1)$ is 
a weak del Pezzo surface of degree 5 
when $\frac{3}{5}<w <1$.
\end{itemize}
Let $\Pi_{-2}(D)$ and $\Pi'_{-2}(D)$ be the effective divisors consisting of all $(-2)$-curves on 
$\overline{\mathfrak{Bun}}_{\boldsymbol{w}}(D,1)$
when $\frac{1}{3}<w <\frac{3}{5}$ 
and $\frac{3}{5}<w <1$,
respectively.
Then the configurations of $\Pi_{-2}(D)$ and $\Pi_{-2}'(D)$ are same,
and the configuration
is $A_1$, $2A_1$, $A_2$, $A_2+A_1$, $A_3$ and $A_4$
when $D=D_{2111}$, $D_{221}$, $D_{311}$, $D_{32}$, $D_{41}$, and $D_{5}$,
respectively.

\item[(ii)]
If $(E,\{ l_{i,\bullet }\}_{i} )$ be a refined parabolic bundle
corresponding to a point on $\Pi_{-2}(D)$ or $\Pi'_{-2}(D)$,
then there exists $i$ $(1\leq i\leq \nu)$ such that $l_{i,n_i}$ is not free,
that is, $l_{i,\bullet }$ is not a parabolic structure for this $i$.

\item[(iii)]
If $(E,\{ l_{i,\bullet }\}_{i} )$ be a refined parabolic bundle
corresponding 
to a point on $\overline{\mathfrak{Bun}}_{\boldsymbol{w}}(D,1)$ 
$(\frac{1}{5}<w <\frac{1}{3})$,
$\overline{\mathfrak{Bun}}_{\boldsymbol{w}}(D,1) \setminus \Pi_{-2}(D)$ 
$(\frac{1}{3}<w <\frac{3}{5})$, or 
$\overline{\mathfrak{Bun}}_{\boldsymbol{w}}(D,1)\setminus \Pi_{-2}'(D)$ 
$(\frac{3}{5}<w <1)$,
then $l_{i,n_i}$ is free for any $i$, and
$(E,\{ l_{i,n_i}\}_i )$ is a simple parabolic bundle.
In particular, $(E,\{ l_{i,n_i}\}_i )$ is a $\Lambda$-flat for generic $\Lambda$.
\end{itemize}
\end{thmA}

These weak del Pezzo surfaces are listed in \cite[Section 8.6.3]{Dolga}.
Descriptions of the corresponding isomonodromic deformations 
are given by Kimura \cite{Kimura} (see also \cite{KNS}).
The corresponding isomonodromic deformations
 are the following:
\begin{center}
  \begin{tabular}{|c|c|c|}
    \hline
 the divisor    & $\Pi_{-2}$ 
    &  the isomonodromic deformations\\
    \hline\hline 
    $D_{2111}$ &$A_1 $ & $H(1,1,1,2)$  \\
    \hline
        $D_{221}$ &$2A_1 $ & $H(1,2,2)$  \\
    \hline
        $D_{311}$ &$A_2 $ & $H(1,1,3)$  \\
    \hline
        $D_{32}$ &$A_2 +A_1 $ & $H(2,3)$  \\
    \hline
        $D_{41}$ &$A_3 $ & $H(1,4)$  \\
    \hline
        $D_{5}$ &$A_4 $ & $H(5)$  \\
    \hline
  \end{tabular}
\end{center}
Here the notations in the last column are in \cite[pp.39--40]{Kimura}.

Fifth, we will
recover the geometry of the weak del Pezzo surfaces
of degree 4 from the modular point of view.
When weights $\boldsymbol{w}$ are the democratic weights with $w=\frac{1}{2}$,
the moduli spaces of refined parabolic bundles are 
weak del Pezzo surfaces of degree 4.
We may reinterpret the negative curves on these surfaces as
the loci of special refined parabolic bundles.
Since weights $\boldsymbol{w}$ are the democratic weights with $w=\frac{1}{2}$,
we may show that the stability of refined parabolic bundles is preserved by 
applying the elementary transformations (Section \ref{section:2021.12.31.15.42}).
So some elementary transformations induce 
automorphisms on the moduli space 
$\overline{\mathfrak{Bun}}_{\boldsymbol{w}}(D,d)$ with $w=\frac{1}{2}$.
We will reconstruct automorphisms of the weak del Pezzo surfaces of degree 4
via these elementary transformations of refined parabolic bundles.

The organization of this paper is as follows.
In Section \ref{section:ParaBun},
we will define parabolic bundles over $(\mathbb{P}^1,D)$,
where $D$ is not necessary reduced.
Next, in Section \ref{section:2021.12.31.15.53}, we will
show the implications \eqref{2022.2.23.17.26} and
\eqref{2022.2.23.17.32}, 
and will give a flatness criterium.
In Section \ref{section:2023.2.24.12.47}, we will give examples of the flatness criterium.
In Section \ref{section:2021.12.31.15.59}, 
we will define refined parabolic bundles
and will discuss combinatorial structure of refined parabolic structures at a multiple point.
This structure is related to Young tableaus.
In Section \ref{section:StabilityOfRefined}, we will define a stability condition 
of refined parabolic bundles,
and we will show Theorem A.
In Section \ref{section:2021.12.31.15.32}, 
we will define elementary transformations of refined parabolic bundles.
In Section \ref{section:2021.12.31.15.34},
we will show some properties of the 
elementary transformations of refined parabolic bundles.
In Section \ref{section:2021.12.31.15.25},
we will describe the moduli spaces when $\deg(D)=5$, $\deg(E)=1$,
and weights $\boldsymbol{w}$ are democratic weights. 
In Section \ref{section:2021.12.31.15.25_1},
we will
recover the geometry of the weak del Pezzo surfaces
of degree 4 from the modular point of view.
In the appendix, 
we will give a list of special refined parabolic bundles,
which appear when $\deg(D)=5$ and $\deg(E)=1$.


\subsection*{Acknowledgments}
The authors would like to warmly thank Professor Roman Fedorov for illuminating
discussion helping us to 
understand the flatness criterium in Proposition \ref{Prop:FlatnessImplications}.
They also thank the referee for useful remarks.


\section{Parabolic bundles}\label{section:ParaBun}

Let $D$ be an effective divisor on $\mathbb P^1$, defined by 
$$D=\sum_{i=1}^\nu n_i[t_i],\ \ \ \deg(D)=\sum_{i=1}^\nu n_i=:n.$$
Take a generator $f_i$ of the maximal ideal of $\O_{\mathbb{P}^1,t_i}$.
Set 
$\O_{n_i [t_i]} := \O_{\mathbb{P}^1,t_i}/(f_i^{n_i})$.
For a vector bundle $E$ on $\mathbb{P}^1$,
we denote by $E|_{n_i[t_i]}$ the tensor product $E \otimes \O_{n_i [t_i]}$.
Let $I= \{ 1,2,\ldots,\nu \}$ be the set of indexes of $D$.
Take $\Lambda=(\lambda^+,\lambda^-)$, with $\lambda^{\pm}\in\Omega^1(D)/\Omega^1$.
Set $\lambda_i^{\pm}:=\lambda^{\pm}\vert_{n_i[t_i]} 
\in \Omega^1(D)/\Omega^1 \otimes \O_{n_i [t_i]}$
for each $i \in I$.
We assume (as in \cite[Section 2.1]{AF}):
\begin{itemize}
\item[(a)] the order of pole of $\lambda_i^+-\lambda_i^-$ is precisely $n_i$,
\item[(b)] $\sum_i \Res_{t_i} (\lambda_i^++\lambda_i^-)=d\in\mathbb Z$,
\item[(c)] $\sum_i \Res_{t_i}\lambda_i^{\pm}\not\in\mathbb Z$ whatever are the signs $\pm$,
\item[(d)] if $n_i=1$ then $ \Res_{t_i}\lambda_i^+-  \Res_{t_i}\lambda_i^-\not\in\mathbb Z$.
\end{itemize}

The relation in (b) is called Fuchs' relation
and the condition (d) is called non resonant condition.
Let $E$ be a rank 2 vector bundle on $\mathbb{P}^1$
and $\nabla \colon E \rightarrow E \otimes \Omega^1(D)$
be a meromorphic connection
with polar divisor $D$.

\begin{definition}\label{2022_9_29_12_06}
We say $(E,\nabla)$ is {\rm $\Lambda$-connection} 
if its principal part at each pole $n_i[t_i]$ is given by 
$$
d+\begin{pmatrix}\lambda_i^+&0\\ 0&\lambda_i^-\end{pmatrix}
\colon E|_{n_i[t_i]} \longrightarrow E|_{n_i[t_i]} \otimes \Omega^1(D)/\Omega^1
$$
up to holomorphic gauge transformation.
We call $\Lambda$ the {\rm formal} (or {\rm spectral}) {\rm data} of the meromorphic connection.
\end{definition}
If $n_i \geq 2$, the condition (a) implies that  a $\Lambda$-connection $(E,\nabla)$ has unramified irregular singular point at $t_i$.   
The degree of $E$ is $d$ by (b).
The following is \cite[Lemma 4.6]{AF}:

\begin{prop}\label{Prop:GenericIrred}
Let $\Lambda$ satisfying {\rm(a),(b),(c),(d)} as above.
Then, any $\Lambda$-connection $(E,\nabla)$ is irreducible.
\end{prop}

\begin{proof}If $L\subset E$ is a $\nabla$-invariant line bundle, then $\nabla\vert_{L}:L\to L\otimes\Omega^1(D)$ 
is a connection with principal parts $\lambda_i^{\pm}$ at each $t_i$. Therefore, Fuchs' relation for $\nabla\vert_{L}$
yields a contradiction with generic assumption (c) on residual eigenvalues.
\end{proof}

Let $(E,\nabla)$ be a $\Lambda$-connection.
By the condition of $\Lambda$,
there is a unique free $\O_{n_i [t_i]}$-submodule $l_i \subset E|_{n_i[t_i]}$
of length $n_i$ for each $i$
by the compatible with $\nabla$ in the following sense:
$\nabla (s)- \lambda_i^+s$ is regular at $t_i$
for any $i$ and every section $s$ of $E$ in a neighborhood of $t_i$
such that $s|_{n_i [t_i]} \in l_i$.
Let $\bold{l} = \{ l_i \}_{i \in I}$ be the set of the free $\O_{n_i [t_i]}$-submodule $l_i \subset E|_{n_i[t_i]}$
of length $n_i$
given by a $\Lambda$-connection for each $i$.
We call the triple $(E, \nabla , \bold{l})$ 
{\it unramified irregular singular $\Lambda$-parabolic connection} 
(see \cite[Definition 2.1]{IS}).
After fixing convenient weights $\boldsymbol{\alpha}$,
the moduli space of $\boldsymbol{\alpha}$-stable
unramified irregular singular $\Lambda$-parabolic connections of rank $2$ on $(\mathbb{P}^1,D)$
forms a smooth quasi-projective scheme of dimension $2(n-3)$ (see \cite[Section 2]{IS}).
In our case,
by Proposition \ref{Prop:GenericIrred}, 
unramified irregular singular $\Lambda$-parabolic connections
induced by $\Lambda$-connections are irreducible.
So these unramified irregular singular $\Lambda$-parabolic connections
satisfy the stability condition in \cite[Definition 2.2]{IS}.
So we identify $\boldsymbol{\alpha}$-stable
unramified irregular singular $\Lambda$-parabolic connections
with $\Lambda$-connections.
Let $\CON_{\Lambda}(D)$ be the moduli space of $\Lambda$-connections.
In this paper, we are concerned with forgetting the connections from 
unramified irregular singular $\Lambda$-parabolic connections.

\begin{definition}\label{2022.2.23.22.18}
We say $(E, \bold{l} )$ (where $\bold{l} = \{ l_{ i} \}_{ i \in I  }$) is
a {\rm parabolic bundle of rank $2$ and of degree $d$ 
over $(\mathbb{P}^1,D)$}
if 
$E$ is a rank 2 vector bundle on $\mathbb{P}^1$ with $\deg(E)=d$ 
and $l_{ i } \subset E|_{n_i [t_i]}$ is a free $\mathcal{O}_{n_i [t_i]}$-submodule of
length $n_i$ for each $i$.
We call $l_{ i}$ a {\rm parabolic structure at $t_i$}.
\end{definition}

\subsection{Flatness criterium}\label{section:2021.12.31.15.53}

We say a parabolic bundle $(E, \bold{l})$ is {\it $\Lambda$-flat} if 
there exists a connection on $E$ compatible with parabolic structure $\bold{l}$ 
and formal data $\Lambda$.
We say that a parabolic bundle
$(E, \{ l_{ i } \}_{ i \in I  } )$ is {\it decomposable} if
there exists a decomposition $E =L_1 \oplus L_2$ such that
$l_{i} = l^{(1)}_{i}$ or $l_{i} = l^{(2)}_{i}$ for any $i\in I$,
where we set $l^{(1)}_{i} := l_{i}\cap (L_1|_{n_i[t_i]})$
and $l^{(2)}_{i} := l_{i}\cap (L_2|_{n_i[t_i]})$.
We say that $(E, \{ l_{ i} \}_{ i \in I  } )$ is {\it undecomposable} if
$(E, \{ l_{ i } \}_{ i \in I } )$ is not decomposable.
If $D$ is a reduced divisor, 
then
$$
\text{$\mathrm{End}(E,\bold{l})=\mathbb C$ } \Longleftrightarrow 
\text{ $(E,\bold{l})$ is $\Lambda$-flat } \Longleftrightarrow 
\text{ $(E,\bold{l})$ is undecomposable}
$$
(see \cite[Proposition 3]{AL} and \cite[Proposition 3.1]{LS}).
The identification of these three conditions is not necessarily true for any curves.
So for reduced divisor cases,
the target of the forgetting map from $\CON_{\Lambda}(D)$ is 
the moduli space of undecomposable parabolic bundles.
Here $\CON_{\Lambda} (D)$
is the moduli space of $\Lambda$-connections.
For $n=4$ (that is, $D=t_1+\cdots+t_4$), 
the moduli space of rank 2 undecomposable parabolic bundles with degree $0$ (or $1$)
can be obtained by glueing two copies of $\mathbb{P}^1$ outside $t_1,\ldots ,t_4$
(see \cite[Section 2.3]{AF}).
For $n=5$ (that is, $D=t_1+\cdots+t_5$), 
Loray--Saito have provided a precise description of
the moduli space of rank 2 undecomposable parabolic bundles with degree $-1$.
This moduli space is closely related to the geometry of degree 4 del Pezzo surfaces.

Now we consider cases where $D$ is not necessary reduced. 
It is proved in \cite[Proposition 4.8]{AF} that, 
if $n=4$ (Painlev\'e case) and $\deg(E)$ is odd, then 
$$
\text{ $\mathrm{End}(E,\bold{l})=\mathbb C$ }
\Longleftrightarrow
\text{ $(E,\bold{l})$ is $\Lambda$-flat}.
$$
We will check the following claim: 
If $(E,\bold{l})$ is $\Lambda$-flat, then $(E,\bold{l})$ is undecomposable.
But the converse proposition is not true even when $\deg(D)=4$ and $\deg(E)$ is odd.
So
in addition to undecomposability, we assume the {\it admissibleness} defined below.

Remind that, on $\mathbb P^1$, any vector bundle is decomposable by Birkhoff:
$$E\simeq\O(d_1)\oplus\O(d_2)\hskip2cm\text{with}\hskip2cm d_1\le d_2\ \text{unique}.$$
The subbundle $\O(d_1)\hookrightarrow E$ is not unique, but if $d_1<d_2$, then $\O(d_2)\hookrightarrow E$ 
is unique. When $d_1=d_2$, the two subbundles of the decomposition can be arbitrarily choosen in 
a $1$-parameter family. In any case, the (possible) factor $L=\O(d_2)$ of the decomposition is characterized 
by $\deg(E)\le 2\deg (L)$.
\begin{definition}\label{eq:2021.12.31.11.48}
We will say that $(E,\bold{l})$ is {\rm admissible} if 
the following condition is satisfied:
\begin{equation}
\forall L\subset E,\ \  \deg(E)\le 2\deg (L)\ \ 
\Longrightarrow\ \ \sum_{i \in I} \length( l_i \cap L|_{n_i[t_i]})\le n+\deg(E)-2\deg(L)-2.
\end{equation} 
\end{definition}
Using some material of  \cite[Section 4]{AF} and ideas of \cite[Proposition 3]{AL}, we can prove:

\begin{prop}\label{Prop:FlatnessImplications} 
Let $\Lambda$ satisfying {\rm(a),(b),(c),(d)} as above,
and $(E,\bold{l})$ a parabolic bundle of degree $d$.
Then, we have implications {\rm (i) $\Rightarrow$ (ii) $\Rightarrow$ (iii)} where:
\begin{itemize}
\item[(i)] $\mathrm{End}(E,\bold{l})=\mathbb C$, 
\item[(ii)] $(E,\bold{l})$ is $\Lambda$-flat,
\item[(iii)] $(E,\bold{l})$ is undecomposable and admissible.
\end{itemize}
Moreover, for $n=4$ and $\deg(E)$ odd, these are equivalent.
\end{prop}

\begin{proof}
(ii) $\Rightarrow$ (iii). Let $\nabla$ be a connection on $E$ compatible with the parabolic structure $\bold{l}$ 
and formal data $\Lambda$.
Assume, by contradiction, that $(E,\bold{l})=(L_1,D_1)\oplus(L_2,D_2)$, meaning that the vector bundle splits as 
$E=L_1\oplus L_2$ for line bundles $L_i$, and parabolic directions are distributed on $L_i$ over $D_i$, $i=1,2$,
where $D=D_1+D_2$
with $D_1,D_2$ having disjoint support. Then, with respect to that decomposition, the connection $\nabla$ writes 
$$\begin{pmatrix}\nabla_1 & \theta_{1,2}\\ \theta_{2,1} & \nabla_2\end{pmatrix}$$
where $\theta_{i,j}\in\mathrm{Hom}(L_j,L_i\otimes\Omega^1(D))$ and $\nabla_i:L_i\to L_i\otimes\Omega^1(D)$
a connection. Then, the principal part at $t_i$ writes:
$$\text{either} \hskip1cm \begin{pmatrix}\lambda_i^+ & *\\ 0 & \lambda_i^-\end{pmatrix},
\hskip2cm \text{or} \hskip1cm \begin{pmatrix}\lambda_i^- & 0\\ * & \lambda_i^+\end{pmatrix}.$$
Therefore, Fuchs' relation for $L_1$ (or $L_2$) yields a contradiction with generic assumption (c) on residual eigenvalues.

On the other hand, let $L\subset E$ be an arbitrary proper subbundle of $E$.
We consider the composition map
$$L\hookrightarrow E\stackrel{\nabla}{\rightarrow} E\otimes\Omega^1(D)\twoheadrightarrow (E/L)\otimes\Omega^1(D).$$
This map $\varphi_L$ is $\O$-linear and it is not the zero map, otherwise $L$ would be $\nabla$-invariant, 
therefore contradicting Proposition \ref{Prop:GenericIrred}. We deduce that
$$m:=\deg(E)-2\deg(L)+n-2\ge0.$$
Moreover, $m$ is the number of zeroes of $\varphi_L$ (counted with multiplicities).
If $l_i \cap L|_{n_i[t_i]}\neq 0$, then $l_i \cap L|_{n_i[t_i]}$ contributes for zeroes of $\varphi_L$, 
and we therefore deduce
$$\sum_{i \in I}\length(l_i\cap L|_{n_i[t_i]})\le \deg(E)-2\deg(L)+n-2$$
which implies, when applied to the case $\deg(E)\le 2\deg(L)$, that  $(E,\bold{l})$ is admissible.

(i) $\Rightarrow$ (ii). Let $(E,\bold{l})$ be as in (i). We can locally trivialize $E\vert_{U_\alpha}\equiv\mathbb C^2$
with parabolic structure normalized as $\begin{pmatrix}1\\ 0\end{pmatrix}$, so that we can locally define 
a $\Lambda$-connection compatible with the parabolic structure on $U_\alpha$ by
$$\nabla_\alpha=d+\begin{pmatrix}\lambda_\alpha^+&0\\ 0&\lambda_\alpha^-\end{pmatrix}$$
where $\lambda_\alpha^{\pm}$ is a $1$-form on $U_\alpha$ with principal par $\lambda^\pm\vert_{U_\alpha}$.
On intersections $U_\alpha\cap U_\beta$, the difference $\nabla_\alpha-\nabla_\beta=\theta_{\alpha\beta}$
is a section of the sheaf of parabolic Higgs fields $\mathcal Higgs(E,\bold{l})$ which is the subsheaf of
$\mathcal End(E,\bold{l})\otimes\Omega^1(D)$ whose principal parts are nilpotent. In above trivialization,
we have 
$$\theta_{\alpha\beta}=\begin{pmatrix}0&*\\ 0&0\end{pmatrix}+\text{holomorphic}.$$
This defines an element of $\{\theta_{\alpha\beta}\}\in H^1(\mathbb P^1,\mathcal Higgs(E,\bold{l}))$
and the existence of a global $\Lambda$-connection $\nabla$ compatible with the parabolic structure
is equivalent to the vanishing of this cocycle: $\nabla_{U_\alpha}=\nabla_\alpha-\theta_\alpha$ with 
$\theta_\alpha-\theta_\beta=\theta_{\alpha\beta}$.

In order to prove that $\{\theta_{\alpha\beta}\}\equiv0$, we are going to use Serre duality, i.e. there is a perfect pairing
$$\begin{matrix}
H^1(\mathbb P^1,\mathcal Higgs(E,\bold{l}))\times H^0(\mathbb P^1,\mathcal End(E,\bold{l}))
&\to& H^1(\mathbb P^1,\Omega^1)&\stackrel{\sim}{\to}&\mathbb C\\
\left( \{\theta_{\alpha\beta}\},A\right)&\mapsto&\left\{\omega_{\alpha\beta}:=\tr(\theta_{\alpha\beta}\cdot A)\right\}&&
\end{matrix}$$
where $\mathcal End(E,\bold{l})$ is the sheaf of parabolic endomorphisms. 
If $\{\omega_{\alpha\beta}\}\equiv0$ for all $A$, then we can conclude that
$\{\theta_{\alpha\beta}\}\equiv0$ and therefore that there exists a 
global unramified irregular singular $\Lambda$-parabolic connection
on $(E,\bold{l})$. By assumption (i), each $A$ takes the form $c\cdot \mathrm{Id}_E$, and by linearity, 
it is therefore enough to test the vanishing of 
$\{\omega_{\alpha\beta}\}$ for $A=\mathrm{Id}_E$.

Fix any global meromorphic connection $\nabla^0$ on $E$; for instance, one can pull-back the trivial 
connection $d$ from any birational trivialization $E\dashrightarrow \mathbb C^2$. It obviously satisfies 
Fuchs' relation $\sum\Res\nabla^0=-\deg(E)$. Now consider the following meromorphic resolution 
of the cocycle 
$$\theta_{\alpha\beta}=(\nabla_\alpha-\nabla^0)-(\nabla_\beta-\nabla^0);$$
therefore, $\{\omega_{\alpha\beta}\}$ can also be viewed as the meromorphic \v{C}ech coboundary 
$$\{\omega_{\alpha\beta}\}=\partial\left\{\omega_\alpha:=\tr\left((\nabla_\alpha-\nabla^0)\cdot A\right)\right\}.$$
Then we have: 
$$\{\omega_{\alpha\beta}\}\sim0\hskip1cm \Longleftrightarrow\hskip1cm
\sum\Res(\omega_\alpha)=0$$
(note that residue does not depend on the choice of $\alpha$ as residue of $\nabla_\alpha-\nabla_\beta$ is nilpotent).
Indeed, the vanishing of residue is equivalent to the existence of a global meromorphic $1$-form $\omega$
such that $\omega-\omega_\alpha$ is holomorphic, i.e. such that $\{\omega_{\alpha\beta}\}$ is a holomorphic 
\v{C}ech coboundary:
$$\{\omega_{\alpha\beta}\}=\partial\{\omega_\alpha-\omega\}$$
meaning that $\{\omega_{\alpha\beta}\}\sim0$ in $H^1(\mathbb P^1,\Omega^1)$.
If $A=\mathrm{Id}_E$, then we have by Fuchs' relation for $\nabla^0$ and assumption (b) for $\Lambda$, we have
$$\sum\Res(\omega_{\alpha})=\sum\Res(\tr(\nabla_\alpha))-\sum\Res(\tr(\nabla^0))=0.$$

(iii) $\Rightarrow$ (i) in the case $n=4$ and $\deg(E)$ odd. Let $A$ be an endomorphism of the parabolic bundle $(E,\bold{l})$. 
Since eigenvalues of $A\vert_{E_x}$ are holomorphic functions of $x\in X$, they are constant $c_1,c_2\in\mathbb C$.
If $c_1\not=c_2$, then the corresponding 
eigenspaces define a decomposition $E=L_1\oplus L_2$ and parabolic structure, since it is $A$-invariant, 
must be distributed in $L_1$ and $L_2$; this implies that $(E,\bold{l})$ is decomposable, contradiction.
We therefore have $c_1=c_2=:c$, and any endomorphism of the parabolic bundle takes the form 
$A=c\cdot \mathrm{Id}_E+N$ with $N$ nilpotent. Clearly, $c\cdot \mathrm{Id}_E\in\mathrm{End}(E,\bold{l})$,
and (i) is equivalent to say that there is no non-zero nilpotent endomorphism in $\mathrm{End}(E,\bold{l})$.

Assume, by contradiction, that $N$ is a non-zero nilpotent element of $\mathrm{End}(E,\bold{l})$.
Let $L_0=\ker(N)=\mathrm{image}(N)$. Since $N$ is not trivial, it induces a non trivial morphism $E/L_0\to L_0$. 
This implies that $\deg(E)-2\deg(L_0)\le0$. In particular, we have a decomposition $E=L\oplus L_0$
with $\deg(L)\le\deg(L_0)$ and, through that decomposition, the nilpotent endomorphism $N$ writes:
$$N=\begin{pmatrix}0&0\\f&0\end{pmatrix}$$
where $f=f(x)$ is a polynomial of degree $2\deg(L_0)-\deg(E)$. 
Locally at $n_i[t_i]$, the parabolic structure is given, through the splitting and local coordinate $z=x-t_i$, by  
$$\begin{pmatrix}z^{m_i}g(z)\\1\end{pmatrix}$$
where $m_i= \sum_{i\in I}\mathrm{length} ( l_i \cap L_0\vert_{n_i[t_i]})$, 
and $g(0)\not=0$. The action of $N$ on the parabolic writes
$$\begin{pmatrix}0&0\\f&0\end{pmatrix}\cdot\begin{pmatrix}z^{m_i}g\\ 1\end{pmatrix}=\begin{pmatrix}0\\z^{m_i}fg\end{pmatrix}$$
so that $N$ preserves the parabolic structure if, and only if, $f$ (and therefore $N$) vanish at $t_i$ at order $n_i-2m_i$.
Indeed, we have
\begin{equation}\label{2022.2.23.20.43}
\det\begin{pmatrix}z^{m_i}g & 0\\ 1& z^{m_i}fg\end{pmatrix}=z^{2m_i}fg^2=0\ \mod\ z^{n_i}\ \ \ \text{iff}\ \ \ 
f=z^{n_i-2m_i}\tilde f
\end{equation}
and the induced action of $N$ is by multiplication
$$\begin{pmatrix}0&0\\f&0\end{pmatrix}\cdot\begin{pmatrix}z^{m_i}g\\ 1\end{pmatrix}=z^{m_i}fg\cdot \begin{pmatrix}z^{m_i}g\\ 1\end{pmatrix}\ \mod\ z^{n_i}.$$
We promptly deduce that 
\begin{equation}\label{2022.2.23.23.50}
n-2
\sum_{i \in I} \length( l_i \cap L_0|_{n_i[t_i]})\le\deg(f)\le 2\deg(L_0)-\deg(E).
\end{equation}
On the other hand, from (iii), $(E,\bold{l})$ is admissible and we have that 
$$
\sum_{i \in I} \length( l_i \cap L_0|_{n_i[t_i]})\le n+\deg(E)-2\deg(L_0)-2.$$
We deduce 
$$\frac{n+\deg(E)-2\deg(L_0)}{2}\le
\sum_{i \in I} \length( l_i \cap L_0|_{n_i[t_i]})\le n+\deg(E)-2\deg(L_0)-2.$$
This implies, by comparing left and right hand-sides, that
$$n\ge\underbrace{2\deg(L_0)-\deg(E)}_{\ge 0}+ 4$$
which is impossible when $n=4$ and $\deg(E)$ odd.
Therefore, $N=0$ and $\mathrm{End}(E,\bold{l})=\mathbb C$.
\end{proof}

The undecomposability and admissibleness are 
independent of $\Lambda$. 
So, when $D$ is reduced or when $n=4$ and $\deg (E)$ is odd,
the $\Lambda$-flatness is independent of $\Lambda$.
But, in general, the $\Lambda$-flatness  
depends on $\Lambda$.
Indeed, we have the following proposition.

\begin{prop}[{\cite[Corollary 4.4]{AF}}]\label{Prop:FlatnessCriterium}
Let $\Lambda$ satisfying {\rm(a),(b),(c),(d)} as above,
and $(E,\bold{l})$ a parabolic bundle of degree $d$. 
Assume $(E,\bold{l})$ is undecomposable. 
Then $(E,\bold{l})$ is $\Lambda$-flat if, and only if, for any nilpotent endomorphism
$N$ of $(E,\bold{l})$, we have
\begin{equation}\label{eq:FlatnessCriterium}
\Res(N\vert\bold{l}\cdot (\lambda^+-\lambda^-))=0
\end{equation}
where $N\vert\bold{l}$ denotes the action induced by $N$ on the parabolic structure,
and $\Res$ is the sum of residues.
\end{prop}

\begin{proof}We rephrase and continue the part [(iii) $\Rightarrow$ (i)] of the proof of Proposition \ref{Prop:FlatnessImplications}.
The obstruction to the existence of a connection is defined by an element 
$\{\nabla_{\alpha}-\nabla_\beta\}\in H^1(\mathbb P^1,\mathcal Higgs(E,\bold{l}))$. By Serre duality, it is zero 
if, and only if $\Res(\tr\left(\nabla_\alpha-\nabla^0)\cdot A\right)=0$ for all endomorphism $A$ of $(E,\bold{l})$;
here $\nabla^0$ is a global meromorphic connection that we have fixed {\it a priori}.
Since $(E,\bold{l})$ is undecomposable, any endomorphism takes the form $A=c\cdot \mathrm{Id}_E+N$ with $N$ nilpotent.
We already know that the condition is satisfied for $A=\mathrm{Id}_E$ because of Fuchs' relation and assumption (b) for $\Lambda$;
by linearity, it is enough to consider $A=N$ nilpotent. A non zero nilpotent endomorphism takes the form
$$N=\begin{pmatrix}0&0\\f&0\end{pmatrix}$$
in a decomposition $E=L\oplus L_0$ with $\deg(L)\le\deg(L_0)$.
In fact, we can choose $\nabla^0$ such that it has one single pole with residue 
$$\begin{pmatrix}k_1 & 0\\ 0 & k_2\end{pmatrix}\frac{dz}{z}$$
so that the contribution of $\nabla^0$ in $\Res(\tr\left(\nabla_\alpha-\nabla^0)\cdot A\right)$ is zero.
Let us compute the contribution of $\nabla_\alpha$. Let the parabolic structure be defined by 
$$\begin{pmatrix}g\\ 1\end{pmatrix}$$
(e.g. in local coordinate near a pole). Since it is $N$-invariant, we have that $fg^2$ restricts to zero on $D$ (or $n_i[t_i]$), that we note $fg^2\sim 0$, and
$$\begin{pmatrix}0&0\\f&0\end{pmatrix}\cdot\begin{pmatrix}g\\ 1\end{pmatrix}=\begin{pmatrix}0\\ fg\end{pmatrix}\sim fg \begin{pmatrix}g\\ 1\end{pmatrix}$$
so that the induced action by (nilpotent) multiplication $N\vert_D$ is the multiplication by $fg$. 
For instance, if $g$ vanishes at $t_i$ with multiplicity $m_i$,
then $f$ vanishes with multiplicity $n_i-2m_i$, and $N\vert_D=fg$ vanishes with multiplicity $n_i-m_i$.
By gauge transformation
$$ \begin{pmatrix} 1 & -g  \\ 0 & 1  \end{pmatrix} \begin{pmatrix} 0 & 0  \\ f & 0  \end{pmatrix} \begin{pmatrix} 1 & g  \\ 0 & 1  \end{pmatrix}\sim \begin{pmatrix} -fg & 0 \\ f & fg \end{pmatrix},$$
we can assume
$$ \bold{l}=\begin{pmatrix} 0 \\ 1 \end{pmatrix}\hskip1cm
\text{and therefore}\hskip1cm
\nabla_\alpha=\begin{pmatrix} \lambda^- & 0 \\ * & \lambda^+ \end{pmatrix}.$$
Then, we finally get 
$$(\nabla_\alpha-\nabla^0)\cdot N=\begin{pmatrix} -fg\lambda^- & 0 \\ * & fg\lambda^+ \end{pmatrix}$$
so that
$$\Res(\tr\left(\nabla_\alpha-\nabla^0)\cdot N\right)=\Res(fg(\lambda^+-\lambda^-))=\Res(N\vert\bold{l}\cdot (\lambda^+-\lambda^-))$$
must be zero for any nilpotent endomorphism $N$ of  $(E,\bold{l})$.
\end{proof}

\subsection{Examples}\label{section:2023.2.24.12.47}

We will give counterexamples of the implication [(ii) $\Rightarrow$ (i)]
and the implication [(iii) $\Rightarrow$ (ii)],
which are converses
of the implications in Proposition \ref{Prop:FlatnessImplications}.
To give counterexamples, we will use the criterium of Proposition \ref{Prop:FlatnessCriterium}.
First, we consider an example of this criterium.

\begin{example}\label{ex:parab221}
Assume $n=5$ with divisor $2[0]+2[1]+[\infty]$, $E=\O\oplus\O(1)$ and 
the parabolic structure $\bold{l}$
intersect the destabilizing line bundle $\O(1)$ only at $[0]+[1]$ (without multiplicity). 
In other words, the parabolic structure is given by
$$\begin{pmatrix} c_0 x\\ 1\end{pmatrix}\ \mod\ x^2, 
\ \ \ \begin{pmatrix} c_1 (x-1)\\ 1\end{pmatrix}\ \mod\ (x-1)^2,
\ \ \ \text{and}\ \ \ \begin{pmatrix} 1 \\ 0 \end{pmatrix}\ \text{at}\ \infty.$$
Then $(E,\bold{l})$ is undecomposable provided that $(c_0,c_1)\not=(0,0)$. 
On the other hand, there is (up to homothecy) one non trivial nilpotent endomorphism
$$N=\begin{pmatrix} 0& 0\\ 1 & 0
\end{pmatrix}$$
that preserves the parabolic structure. It induces the multiplication by $c_0x$ and $c_1(x-1)$ 
on parabolic directions over $x=0$ and $x=1$ respectively.
Up to bundle isomorphism, we see that only the projective variable $[c_0:c_1]\in\mathbb P^1$ makes sense.
Criterium of Proposition \ref{Prop:FlatnessCriterium} writes
$$\Res(c_0x(\lambda_0^+-\lambda_0^-)+c_1(x-1)(\lambda_1^+-\lambda_1^-))=0.$$
Recall that dominant terms of $\lambda^+-\lambda^-$ are non zero at $x=0$ and $x=1$, given say by $a_0\frac{dx}{x^2}$
and $a_1\frac{dx}{(x-1)^2}$ respectively. Flatness condition writes
$$c_0a_0+c_1a_1=0$$
which determines the parabolic bundle uniquely.  
\end{example}

Now we consider the case where $n=\deg(D)=4$ and $\deg(E)$ is even.
We give counterexamples of the implication [(ii) $\Rightarrow$ (i)]
and the implication [(iii) $\Rightarrow$ (ii)].

\begin{prop}\label{prop:2021.11.30.22.07_0}
Let $n=\deg(D)=4$,  $\Lambda$ 
as above, $(E,\bold{l})$ a parabolic bundle, 
$\deg(E)$ any even number. 
\begin{itemize}
\item Assume that $(E,\bold{l})$ is undecomposable and admissible.
If $(E,\bold{l})$ is not simple,
then $E=L\otimes(\O\oplus\O)$ and
\begin{itemize}
\item[(i)] $D=2[t_1]+2[t_2]$ and parabolic structure writes
$$\begin{pmatrix} c_1 (x-t_1)\\ 1\end{pmatrix}\ \ \ \text{and}\ \ \ \begin{pmatrix} c_2 (x-t_2)\\ 1\end{pmatrix}$$
where $[c_1:c_2]\in\mathbb P^1$;
\item[(ii)] $D=4[t_1]$ and parabolic structure writes
$$\begin{pmatrix} c_1 (x-t)^2+c_2(x-t)^3\\ 1\end{pmatrix}$$
where $[c_1:c_2]\in\mathbb P^1$.
\end{itemize}

\item The only $\Lambda$-flat parabolic bundles
that are not simple are the following ones: $E=L\otimes(\O\oplus\O)$ and
\begin{itemize}
\item $D=2[t_1]+2[t_2]$ and parabolic structure writes as in (i)
where $[c_1:c_2]\in\mathbb P^1$ is a point determined by $\Lambda$;
\item $D=4[t_1]$ and parabolic structure writes as in (ii)
where $[c_1:c_2]\in\mathbb P^1$ is a point determined by $\Lambda$.
\end{itemize}
\end{itemize}
\end{prop}

\begin{proof}
Setting $k=\max_L(2\deg(L)-\deg(E))$, 
$L_0$ be maximizing and $m=\sum_{i \in I} \length( l_i \cap L_0|_{n_i[t_i]})$. 
Then, end of proof of Proposition \ref{Prop:FlatnessImplications} yields
$$\frac{n-k}{2}\le m\le n-k-2    \hskip2cm\text{so that in particular}\hskip1cm
k+4\le n.$$
It follows, when $n=4$ that $k=0$ and $m=2$. Since $k=0$, then $E=L\otimes(\O\oplus\O)$. 
On the other hand, assuming $(E,\bold{l})$ is not simple, we know that there is a non trivial nilpotent 
endomorphism which writes
$$N=\begin{pmatrix} 0&0\\1&0\end{pmatrix}$$
and $L_0$ is the invariant subspace. We deduce that all parabolics intersect $L_0$,
and at each pole $t_i$, we must have $n_i-2 m_i\le0$; otherwise, parabolic structure would not be invariant.
Together with above inequalities, we get $n_i=2m_i$ at each $t_i$, 
whence the only two possibilities of the statement.
In each case, writing down the criterium of Proposition \ref{Prop:FlatnessCriterium}, as in Example \ref{ex:parab221},
we get the restriction on $(c_1:c_2)$.
\end{proof}

Now we consider the case where $n=\deg(D)=5$ and $\deg(E)$ is odd.

\begin{lem}\label{2022.2.2.22.32}
Let $n=\deg(D)=5$ and
$(E,\bold{l})$ a parabolic bundle with odd degree.
Assume that $(E,\bold{l})$ is undecomposable and admissible.
If $(E,\bold{l})$ is not simple,
then $E=L\otimes(\O\oplus\O(1))$ and
\begin{itemize}
\item[(i)] $D=2[t_1]+2[t_2]+[t_3]$ and parabolic structure writes
$$\begin{pmatrix} c_1 (x-t_1)\\ 1\end{pmatrix},\ \ \ \begin{pmatrix} c_2 (x-t_2)\\ 1\end{pmatrix}\ \ \ \text{and}\ \ \ \begin{pmatrix} 1\\0\end{pmatrix}$$
where $[c_1:c_2]\in\mathbb P^1$ satisfies $c_1c_2 \neq 0$;
\item[(ii)] $D=3[t_1]+2[t_2]$ and parabolic structure writes
$$\begin{pmatrix} c_1(x-t_1)+c_2(x-t_1)^2\\ 1\end{pmatrix}\ \ \ \text{and}\ \ \ \begin{pmatrix} c_3 (x-t_2)\\ 1\end{pmatrix}$$
where $[c_1:c_2:c_3]\in\mathbb P^2$ satisfies $c_1c_3 \neq 0$;
\item[(iii)] $D=4[t_1]+[t_2]$ and parabolic structure writes
$$\begin{pmatrix} c_1 (x-t_1)^2+c_2(x-t_1)^3\\ 1\end{pmatrix}\ \ \ \text{and}\ \ \ \begin{pmatrix} 1\\ 0\end{pmatrix}$$
where $[c_1:c_2]\in\mathbb P^1$ satisfies $c_1 \neq 0$;
\item[(iv)] $D=5[t_1]$ and parabolic structure writes
$$\begin{pmatrix} c_1 (x-t_1)^2+c_2(x-t_1)^3+c_3(x-t_1)^4\\ 1\end{pmatrix}$$
where $[c_1:c_2:c_3]\in\mathbb P^2$ satisfies $c_1 \neq 0$.
\end{itemize}
In particular, when $D=2t_1+t_2+t_3+t_4$ or $D=3t_1+t_2+t_3$, 
we have that
\begin{equation*}
\mathrm{End} (E,\bold{l})=\mathbb{C}
\, \Longleftrightarrow \, 
\text{$(E,\bold{l})$ is $\Lambda$-flat}
\, \Longleftrightarrow \, 
\text{$(E,\bold{l})$ is undecomposable and admissible}.
\end{equation*}
\end{lem}

\begin{proof}
Setting $k=\max_L(2\deg(L)-\deg(E))$, 
$L_0$ be maximizing and $m=\sum_{i \in I} \length( l_i \cap L_0|_{n_i[t_i]})$. 
Since $(E,\bold{l})$ is not simple,
we have $k+4 \leq n$ as in the proof of the previous proposition.
By this inequality,
we have $k=0$ or $k=1$.
Since $E$ has odd degree, $k\neq 0$.
So we have $k=1$.
This means that $E=L\otimes(\O\oplus\O(1))$.
Since $(E,\bold{l})$ is not simple, we know that there is a non trivial nilpotent 
endomorphism which writes
$$N=\begin{pmatrix} 0&0\\f (x)&0\end{pmatrix}$$
and $L_0$ is the invariant subspace.
Here $f (x)$ is a polynomial in $x$ with $\deg(f) \leq 1$.
For one of indexes in $I$, 
$m_i$ satisfies a condition $n_i-2 m_i\leq 1$, 
and, for other indexes, $m_i$ satisfies a condition $n_i-2 m_i\leq 0$.
In particular $  \sum m_i \geq 2$.
Since $(E,\bold{l})$ is admissible,
we have $ \sum m_i \leq 2$.
Then for one of indexes in $I$, $n_i-2 m_i= 1$, 
and, for other indexes, $n_i-2 m_i= 0$.
There exist the only four possibilities of the statement.
\end{proof}

Now we give counterexamples when $n=\deg(D)=5$ and $\deg(E)$ is odd.

\begin{prop}\label{prop:2021.11.30.22.07}
Let $n=\deg(D)=5$, $\Lambda$ 
as above,
and $(E,\bold{l})$ a parabolic bundle with odd degree. 
Then, the only $\Lambda$-flat parabolic bundles
that are not simple are the following ones: $E=L\otimes(\O\oplus\O(1))$ and
\begin{itemize}
\item $D=2[t_1]+2[t_2]+[t_3]$ and parabolic structure writes
as in (i) of Lemma \ref{2022.2.2.22.32}
where $[c_1:c_2]$ is a point determined by $\Lambda$;
\item $D=3[t_1]+2[t_2]$ and parabolic structure writes
as in (ii) of Lemma \ref{2022.2.2.22.32}
where $[c_1:c_2:c_3]$ 
 lies along some line determined by $\Lambda$;
\item $D=4[t_1]+[t_2]$ and parabolic structure writes
as in (iii) of Lemma \ref{2022.2.2.22.32}
where $[c_1:c_2]$ is a point determined by $\Lambda$;
\item $D=5[t_1]$ and parabolic structure writes
as in (iv) of Lemma \ref{2022.2.2.22.32}
where $[c_1:c_2:c_3]$ 
 lies along some line determined by $\Lambda$.
\end{itemize}
\end{prop}

\begin{proof}
By Proposition \ref{Prop:FlatnessImplications}, we have that 
$(E,\bold{l})$ is undecomposable and admissible. 
Then,
the parabolic structure write as in Lemma \ref{2022.2.2.22.32}.
In each case, writing down the criterium of Proposition \ref{Prop:FlatnessCriterium}, 
as in Example \ref{ex:parab221}.
\end{proof}

\section{Refined parabolic bundles}\label{section:2021.12.31.15.59}

In the previous section, we have defined parabolic bundles of rank 2 and of degree $d$ 
over $(\mathbb{P}^1,D)$.
In this section, we will consider a generalization of parabolic bundles,
which will be called {\it refined parabolic bundles}.

\subsection{Cases where $D$ is a reduced effective divisor}\label{sect:2021.12.2.15.19}

Here we assume that $D$ is a reduced effective divisor.
In \cite{LS}, the moduli space of undecomposable parabolic bundles of rank 2 and of degree $d$
was studied.
Since now we assume that $D$ is reduced,
this moduli space coincides with
the moduli space of $\Lambda$-flat parabolic bundles of rank 2 and of degree $d$.
We denote by $\BUN_{\Lambda}(D,d)$
the moduli space of $\Lambda$-flat parabolic bundles with degree $d$.
This moduli space is the image of the forgetting map:
$$
\begin{aligned}
\CON_{\Lambda} (D) &\longrightarrow \BUN_{\Lambda} (D,d) \\
(E,\nabla)&\longmapsto (E,\bold{l}).
\end{aligned}
$$
Here $\CON_{\Lambda} (D)$
is the moduli space of $\Lambda$-connections.
This forgetting map is important for understanding the moduli space $\CON_{\Lambda} (D)$.
But the moduli space $\BUN_{\Lambda} (D,d)$ is a non-separated scheme. 
To get a good moduli space, we have to impose a stability condition with respect to 
$\boldsymbol{w} = (w_i) \in [0,1]^n$.
We denoted by $\BUN_{\boldsymbol{w}}(D,d)$ 
the moduli space of
$\boldsymbol{w}$-semistable parabolic bundles of rank 2 and of degree $d$.
The moduli space $\BUN_{\boldsymbol{w}}(D,d)$
is a normal irreducible projective variety.
The open subset of $\boldsymbol{w}$-stable parabolic bundles is smooth.
For generic weights, we have that $\boldsymbol{w}$-semistable = $\boldsymbol{w}$-stable.
So $\BUN_{\boldsymbol{w}}(D,d)$ is a smooth projective variety.
An important fact is that 
$(E, \{ l_{ i} \}_{i \in I  } )$ 
is stable for a convenient choice of weights $\boldsymbol{w}$
if, and only if, it is undecomposable
(see \cite[Proposition 3.4]{LS}).
So the moduli space $\BUN_{\Lambda} (D,d)$ is covered by 
$\BUN_{\boldsymbol{w}}(D,d)$ for some $\boldsymbol{w}$.

\subsection{Definition of refined parabolic bundles}

We try to extend the story in Section \ref{sect:2021.12.2.15.19} 
for cases where $D$ are not necessary reduced.
When $D$ is not a reduced divisor,
there are some problems:  
(1) undecomposable parabolic bundles are not necessary $\Lambda$-flat and 
are not necessary $\boldsymbol{w}$-stable for some weights $\boldsymbol{w}$; 
(2) $\BUN_{\boldsymbol{w}}(D,d)$ is not necessary projective 
for generic weights $\boldsymbol{w}$. 
(Remark that we impose that $l_i$ is free). 
Here the stability condition for parabolic bundles when 
$D$ is not necessary reduced is defined in Section \ref{section:StabilityOfRefined} below.

For the problem (1),
in Section \ref{sect:2021.12.2.15.32} and Section \ref{2022.3.1.11.33} below,
we will discuss on a necessary and sufficient condition
for the condition that 
there exist weights $\boldsymbol{w}$ such that 
a parabolic bundle is $\boldsymbol{w}$-stable.
The present section is concerned with the problem (2),
that is, constructing a smooth compactification of $\BUN_{\boldsymbol{w}}(D,d)$.
To take a good compactification, 
we consider a generalization of parabolic bundles as follows.

\begin{definition}
We say $(E, \{ l_{ i , \bullet} \}_{  i \in I  } )$ is
a {\rm refined parabolic bundle} of rank $2$ and of degree $d$ if 
$E$ is a rank 2 vector bundle on $\mathbb{P}^1$ with $\deg(E)=d$ 
and $l_{ i , \bullet} = \{ l_{i,k}\}_{1 \leq k \leq n_i}$ is a filtration 
$$
E|_{n_i [t_i]} \supset l_{i,n_i} \supset l_{i,n_i-1} \supset \cdots \supset 
l_{i,1} \supset 0 
$$
of $\mathcal{O}_{n_i [t_i]}$-modules where the length of $l_{i,k}$ is $k$.
We call this filtration a {\rm refined parabolic structure}.
\end{definition}

Let $(E, \{ l_{ i} \}_{ i \in I  } )$ be
a parabolic bundle of rank $2$ and of degree $d$.
For $(E, \{ l_{ i} \}_{ i \in I  } )$, 
we can define a refined parabolic bundle of rank $2$ and of degree $d$
as follows:
We define a filtration $l_{ i,\bullet }$ as 
$$
l_{ i,\bullet } \colon E|_{n_i [t_i]} \supset
l_i \supset f_i \cdot l_i \supset f_i^2\cdot l_i \supset \cdots \supset
f_i^{n_i-1}\cdot l_{i} \supset 0
$$
for each $i$.
Here $f_i$ is a generator of the maximal ideal of $\mathcal{O}_{\mathbb{P}^1,t_i}$.
Then we have a refined parabolic bundle $(E, \{ l_{ i , \bullet} \}_{  i \in I  } )$.
Conversely if a refined parabolic structure 
$E|_{n_i [t_i]} \supset l_{i,n_i} \supset l_{i,n_i-1} \supset \cdots \supset 
l_{i,1} \supset 0 $ where $l_{i,n_i}$ is free, 
then the refined parabolic structure coincides with 
the refined parabolic structure induced by the parabolic structure $l_{i,n_i}$.
So refined parabolic bundles are 
generalization of parabolic bundles.
In Section \ref{section:StabilityOfRefined} below, 
we will define a stability condition for refined parabolic bundles. 
First of all, in present section, we discuss 
refined parabolic structures at a multiple point.
We will describe a combinatorial structure on 
a refined parabolic structure at a multiple point.

\subsection{Refined parabolic structure at a multiple point}

Let $E$ be a rank $2$ vector bundle on $\mathbb{P}^1$.
We take a point $t$ on $\mathbb{P}^1$
and a positive integer $n$ such that $n>1$.
First, we will define the {\it type} of 
an $\O_{n[t]}$-submodule of $E|_{n[t]}$.
We consider a filtration of sheaves
$$
E \supset E(-[t]) \supset E(-2[t]) \supset \cdots \supset  E(-(n-1)[t]) \supset E(-n[t]) .
$$
Note that $E|_{n[t]} = E/E(-n[t])$.
Set $V_j := E(-(n-j+1)[t])/E(-n[t])$ 
for $j=1,2,\ldots,n+1$.
Then we have a filtration 
$$
E|_{n[t]}=V_{n+1}
 \supset V_{n} \supset V_{n-1} \supset \cdots \supset V_2 \supset V_1= 0
$$
of $\O_{n[t]}$-modules.

\begin{definition}
We fix a positive integer $k$ where $1\leq k \leq n$.
Let $l$ be an $\O_{n[t]}$-submodule of $E|_{n[t]}$ with $\mathrm{length}(l)=k$.
We define a tuple of integers 
$\lambda :=(\lambda_{n},\lambda_{n-1},\ldots,\lambda_{1})$ by
$$
\lambda_{j}:=  \mathrm{length} 
\left((V_{j+1}\cap {l} ) /  (V_{j} \cap {l} )  \right)
$$
for $j=1,2,\ldots,n$.
We call the tuple
$\lambda$ the {\rm type of the $\O_{n[t]}$-submodule $l$
with $\mathrm{length}(l)=k$}.
\end{definition}

We have the equality $\sum_{j=1}^{n} \lambda_j=k$.
The inequalities $\lambda_j\geq 0$ and
$$
\lambda_j \leq 
\mathrm{length} 
(V_{j+1} / V_{j}  ) =2
$$
follow for $j=1,2,\ldots,n$.
Let $f$ be a generator of the maximal ideal of $\mathcal{O}_{\mathbb{P}^1,t}$.
The map from $V_{j+1}$ to $V_{j}$ defined by multiply by $f$ 
induces an injective map
$$
\xymatrix@C=55pt{
(V_{j+1}\cap {l}) / (V_{j} \cap {l})
 \ar[r]^-{\text{multiply by $f$}}  
 & (V_{j}\cap {l}) /  (V_{j-1} \cap {l})
}
$$
for $j=2,3,\ldots n$.
Then we have the following inequality 
$$
\lambda_{j}=  \mathrm{length} 
((V_{j+1}\cap {l}) / (V_{j} \cap {l}) ) 
\leq \mathrm{length} 
( (V_{j}\cap {l}) /  (V_{j-1} \cap {l}))
=\lambda_{j-1}  
$$
for $j=2,3,\ldots n$.
For a tuple of integers $\lambda =(\lambda_{n},\lambda_{n-1},\ldots,\lambda_{1})$,
we set
\begin{equation}\label{eq:2021.8.9.16.43}
\begin{aligned}
a_1(\lambda ) :=&\ \# \{ j \mid j\in \{ 1,2,\ldots, n \} ,\lambda_j =1 \} \quad \text{and} \\
a_2(\lambda ) :=&\  \# \{ j \mid j\in \{ 1,2,\ldots, n \} ,\lambda_j =2 \}.
\end{aligned}
\end{equation}

\begin{prop}
If a tuple of integers $\lambda = (\lambda_{n},\lambda_{n-1},\ldots,\lambda_{1})$ 
satisfies the condition
\begin{equation}\label{eq:typeofRPS}
0\leq \lambda_{n} \leq \lambda_{n-1} \leq \cdots \leq \lambda_{1} \leq 2
\quad  \text{and} \quad
\sum_{j=1}^{n} \lambda_j=k,
\end{equation}
then there exists an $\O_{n[t]}$-submodule of $E|_{n[t]}$
with $\mathrm{length}(l)=k$ whose type is $\lambda$.
\end{prop}

\begin{proof}
We define $\boldsymbol{v}$ and $\boldsymbol{v}'$ as 
$$
\boldsymbol{v} = 
f^{n-a_1(\lambda)-a_2(\lambda)}
\begin{pmatrix}
v_1 \\ v_2
\end{pmatrix} \in \O_{n[t]} \oplus \O_{n[t]}
 \quad 
\boldsymbol{v}' = 
f^{n-a_2(\lambda)}
\begin{pmatrix}
v'_1 \\ v'_2
\end{pmatrix} \in \O_{n[t]} \oplus \O_{n[t]}
$$
so that the image of $\boldsymbol{v}$ under the quotient map
 $$ 
 (f^{n-a_1(\lambda)-a_2(\lambda)})^{\oplus 2}
\longrightarrow 
 ((f^{n-a_1(\lambda)-a_2(\lambda)})
/(f^{n-a_1(\lambda)-a_2(\lambda)+1}))^{\oplus 2}
\cong \mathbb{C}^2
$$
and the image of 
$\boldsymbol{v}'$ under the quotient map
$$
(f^{n-a_2(\lambda)})^{\oplus 2}
\longrightarrow 
((f^{n-a_2(\lambda)})/(f^{n-a_2(\lambda)+1}))^{\oplus 2}
\cong \mathbb{C}^2
$$ 
are linearly independent.
Let $l$ be the $\O_{n[t]}$-submodule of $\O_{n[t]} \oplus \O_{n[t]}$
generated by $\boldsymbol{v}$ and $\boldsymbol{v}'$.
The elements 
$f^{a_1(\lambda)} \boldsymbol{v}$ and $\boldsymbol{v}'$
are linearly independent in 
$((f^{n-a_2(\lambda)})^{\oplus 2}\cap l )/((f^{n-a_2(\lambda)+1})^{\oplus 2}\cap l) $.
In other words, 
$$
\mathrm{length} \big( ((f^{n-a_2(\lambda)})^{\oplus 2}\cap l )/((f^{n-a_2(\lambda)+1})^{\oplus 2}\cap l)
 \big) =2.
$$
If we take an isomorphism $E|_{n[t]} \cong \O_{n[t]} \oplus \O_{n[t]}$,
then $l$ induces a refined parabolic structure on $E|_{n[t]}$ 
via this isomorphism.
The type of this refined parabolic structure is $\lambda$.
\end{proof}

Let $\lambda = (\lambda_{n},\lambda_{n-1},\ldots,\lambda_{1})$
be the type of an $\O_{n[t]}$-submodule of $E|_{n[t]}$.
The type $\lambda$ is a partition of the positive integer $k$:
$$
k= \overbrace{0 + \cdots +0
+\underbrace{ 1 + \cdots +1}_{a_1(\lambda)} 
+\underbrace{2+\cdots+2}_{a_2(\lambda)}}^n.
$$
So there exists the corresponding Young diagram to the type $\lambda$.
The number of columns of the Young diagram
is $1$ or $2$.

\begin{lem}\label{lem:2021.9.13.22.13}
Let $l$ be an $\O_{n[t]}$-submodule of $E|_{n[t]}$
with $\mathrm{length}(l)=k$.
Let $\lambda = (\lambda_{n},\lambda_{n-1},\ldots,\lambda_{1})$
be the type of $l$.
If the number of columns of the corresponding Young diagram is $1$, 
then the $\O_{n[t]}$-submodule
 $l$ is generated by one element of $E|_{n[t]}$.
If the number of columns of the corresponding Young diagram
is $2$, 
then the $\O_{n[t]}$-submodule
 $l$ is generated by two elements of $E|_{n[t]}$.
\end{lem}

\begin{proof}
Let $a_1(\lambda )$ and $a_2(\lambda )$ be 
the integers defined in \eqref{eq:2021.8.9.16.43}.
We take $\boldsymbol{v}$ such that 
the image of $\boldsymbol{v}$ under the map 
$ V_{a_1(\lambda)+a_2(\lambda)+1} \rightarrow
 V_{a_1(\lambda)+a_2(\lambda)+1}/V_{a_1(\lambda)+a_2(\lambda)}$ is nonzero.

Now we assume that 
the number of columns of the corresponding Young diagram is $1$.
That is, $a_2(\lambda)=0$.
We consider an element $\boldsymbol{v}' \in l$.
There exists an integer $n_{\boldsymbol{v}'}$ such that 
$\boldsymbol{v}' \in V_{n_{\boldsymbol{v}'}+1}$
and the image of $\boldsymbol{v}'$ under the map 
$ V_{n_{\boldsymbol{v}'}+1} \rightarrow
 V_{n_{\boldsymbol{v}'}+1}/V_{n_{\boldsymbol{v}'}}$ is nonzero.
We have $a_1(\lambda) \geq n_{\boldsymbol{v}'}$.
The element $f^{a_1(\lambda)-n_{\boldsymbol{v}'}} \boldsymbol{v}$ is 
in $V_{n_{\boldsymbol{v}'}+1}$.
Since $\lambda_{n_{\boldsymbol{v}'}}=1$,
the images of
 $f^{a_1(\lambda)-n_{\boldsymbol{v}'}} \boldsymbol{v}$ and $\boldsymbol{v}'$
 under the map 
 $ V_{n_{\boldsymbol{v}'}+1} \rightarrow V_{n_{\boldsymbol{v}'}+1}/V_{n_{\boldsymbol{v}'}}$
are linearly dependent. 
Then there exist complex numbers $\alpha_1$ and $\alpha_2$ such that
$\alpha_1 f^{a_1(\lambda)-n_{\boldsymbol{v}'}} \boldsymbol{v} 
+ \alpha_2 \boldsymbol{v}' \in V_{n_{\boldsymbol{v}'}}$.
Moreover,
the images of
 $f^{a_1(\lambda)-n_{\boldsymbol{v}'}+1} \boldsymbol{v}$ and 
 $\alpha_1 f^{a_1(\lambda)-n_{\boldsymbol{v}'}} \boldsymbol{v} 
+ \alpha_2 \boldsymbol{v}'$
 under the map 
 $ V_{n_{\boldsymbol{v}'}} \rightarrow V_{n_{\boldsymbol{v}'}}/V_{n_{\boldsymbol{v}'}-1}$
are linearly dependent,
since $\lambda_{n_{\boldsymbol{v}'}-1}=1$.
By repeating this argument, 
we may write 
the element $\boldsymbol{v}' \in l$ 
as a linear combination of $\boldsymbol{v},f\boldsymbol{v}, f^2 \boldsymbol{v}$, and so on.
Then $l$ is generated by $\boldsymbol{v}$.

Next we assume that 
the number of columns of the corresponding Young diagram is $2$.
That is, $a_2(\lambda)>0$.
There exists an element 
$\boldsymbol{w}$ such that 
the image of $\boldsymbol{w}$ under the map 
$ V_{a_2(\lambda)+1} \rightarrow V_{a_2(\lambda)+1}/V_{a_2(\lambda)}$ is nonzero
and the images of
 $f^{a_1(\lambda)} \boldsymbol{v}$ and $\boldsymbol{w}$
 under the map 
 $ V_{a_2(\lambda)+1} \rightarrow V_{a_2(\lambda)+1}/V_{n_2}$
are linearly independent.
We may show that $l$ is generated by $\boldsymbol{v}$ and $\boldsymbol{w}$
by using the argument as in the case $a_2(\lambda)=0$.
\end{proof}

Let $\boldsymbol{v}$ be an element of $E|_{n[t]}$.
We denote by $\langle \boldsymbol{v} \rangle_{\O_{n[t]}}$ 
the $\O_{n[t]}$-submodule of $E|_{n[t]}$ generated by $\boldsymbol{v}$.
If we take another generator $\boldsymbol{v}'$ of $\langle \boldsymbol{v} \rangle_{\O_{n[t]}}$,
then there is an element $F \in \O_{n[t]}$ 
such that $\boldsymbol{v} =F \cdot \boldsymbol{v}'$.
We assume that $l$ is generated by elements $\boldsymbol{v}_1, \boldsymbol{v}_2$ 
of $E|_{n[t]}$,
where $\length(\langle \boldsymbol{v}_1\rangle_{\O_{n[t]}} )
\geq \length(\langle \boldsymbol{v}_2 \rangle_{\O_{n[t]}} )$.
If we take another generator $\boldsymbol{v}'_1$ and $\boldsymbol{v}'_2$ of $l$,
where $\length(\langle \boldsymbol{v}_1'\rangle_{\O_{n[t]}} )
\geq \length(\langle \boldsymbol{v}_2' \rangle_{\O_{n[t]}} )$,
then there is an element $F_1,F_2,G_1,G_2 \in \O_{n[t]}$ 
such that $\boldsymbol{v}_1 =F_1 \cdot \boldsymbol{v}'_1+ F_2 \cdot  \boldsymbol{v}'_2$
and  $ \boldsymbol{v}_2 = G_1 \cdot f^d \cdot  \boldsymbol{v}'_1 + G_2 \cdot \boldsymbol{v}'_2$,
where $d= \length(\langle \boldsymbol{v}_1'\rangle )- \length(\langle \boldsymbol{v}_2' \rangle )$.

Let
$l_{\bullet} = \{ l_{k}\}_{1 \leq k \leq n}$ be a filtration 
$$
E|_{n [t]} \supset l_{n} \supset l_{n-1} \supset \cdots \supset 
l_{1} \supset 0 
$$
of $\mathcal{O}_{n [t]}$-submodules of $E|_{n [t]}$ where the length of $l_{k}$ is $k$.
Let $\lambda^{(k)} = (\lambda^{(k)}_{n},\lambda^{(k)}_{n-1},\ldots,\lambda^{(k)}_{1})$
be the type of the $\mathcal{O}_{n [t]}$-module $l_k$ 
for $k=1,2,\ldots,n$.
Then we have a sequence of Young diagrams 
$$
\lambda^{(n)}  \supset \lambda^{(n-1)} \supset \cdots  \supset \lambda^{(1)} ,
$$
where the skew diagram $\lambda^{(k)}/\lambda^{(k-1)}$ consists of one box.

\begin{definition}\label{def:StandTab}
Let
$l_{\bullet} = \{ l_{k}\}_{1 \leq k \leq n}$ be a filtration 
$$
E|_{n [t]} \supset l_{n} \supset l_{n-1} \supset \cdots \supset 
l_{1} \supset 0  \quad (\text{$\mathrm{length} (l_{k})=k$ for $k=1,2,\ldots,n$})
$$
of $\mathcal{O}_{n [t]}$-modules.
We denote by $T_{l_{\bullet}}$ the corresponding standard Young tableau to the 
sequence of Young diagrams as above.
We call $T_{l_{\bullet}}$ the {\rm standard tableau of $l_{\bullet}$}.
\end{definition}

\subsection{Parameter space of refined parabolic structures at a multiple point}

Next, we will consider families of refined parabolic structures at a multiple point
and their parameter spaces.
Here, when we consider families of refined parabolic structures,
we fix the first submodule $l_n$ of refined parabolic structures.
We will construct parameter spaces of families of refined parabolic structures.
We will see that there exists a correspondence between 
their irreducible components and 
standard Young tableaus whose shapes are the type of $l_n$.

Let 
$\lambda = (\lambda_{n},\lambda_{n-1},\ldots,\lambda_{1})$ 
be a Young diagram which satisfies the condition \eqref{eq:typeofRPS}.
Let $a_1(\lambda )$ and $a_2(\lambda )$ be 
the integers defined in \eqref{eq:2021.8.9.16.43}.
We may describe the Young diagram $\lambda$ 
as $( 1^{a_1 (\lambda)} , 2^{a_2 (\lambda)} )$.
Let $l$ be an $\mathcal{O}_{n [t]}$-submodule of $E|_{n[t]}$ with type $\lambda$.
First we consider a parameter space of $\mathcal{O}_{n [t]}$-submodules $l'$ of $l$
with $\mathrm{length} (l/l')=1$.
We consider the case $a_2(\lambda ) = 0$. 
In this case, $l$ is generated by an element $\boldsymbol{v} \in E|_{n[t]}$ 
(see Lemma \ref{lem:2021.9.13.22.13}).  
An $\mathcal{O}_{n [t]}$-submodule $l'$ of $l$
with $\mathrm{length} (l/l')=1$ is unique:
$l' = f \cdot l$.
Here $f$ is a generator of the maximal ideal of $\mathcal{O}_{\mathbb{P}^1,t}$.
For the case $a_2(\lambda ) \neq 0$,
we have the following lemma.

\begin{lem}\label{lem:FamMod}
Assume that $a_2(\lambda ) \neq 0$.
By taking generators of $l$,
we may construct a family of $\mathcal{O}_{n [t]}$-submodules $l'$ of $l$
with $\mathrm{length} (l/l')=1$
parametrized by $\mathbb{P}^1$
such that, 
\begin{itemize}
\item when $a_1(\lambda )\neq 0$:
\begin{itemize}
\item the type of the submodule $l'$ corresponding to $[0:1]$ 
is $(1^{a_1(\lambda )-1}, 2^{a_2(\lambda )})$,
and 
\item the type of the submodule $l'$ corresponding to $[\alpha_1:\alpha_2]$,
where $\alpha_1\neq 0$,
is $(1^{a_1(\lambda )+1}, 2^{a_2(\lambda )-1})$;
\end{itemize}
\item when $a_1(\lambda ) =0$:
\begin{itemize}
\item the type of the submodule $l'$ corresponding to any $[\alpha_1:\alpha_2] \in \mathbb{P}^1$
is $(1^{a_1(\lambda )+1}, 2^{a_2(\lambda )-1})$.
\end{itemize}
\end{itemize}
This family gives a bijection
$$
\mathbb{P}^1 \longrightarrow
\{l'  \mid \text{$l'$ is an $\mathcal{O}_{n [t]}$-submodule of $l$
such that $\mathrm{length} (l/l')=1$} \} .
$$
\end{lem}

\begin{proof}
We take generators $\boldsymbol{v}_1$ and $\boldsymbol{v}_2$ of $l$ 
such that 
$\boldsymbol{v}_1$ is nonzero in  
$V_{a_1(\lambda)+a_2(\lambda)+1}/V_{a_1(\lambda)+a_2(\lambda)}$.
Let $[\alpha_1:\alpha_2]$ be a point of $\mathbb{P}^1$. 
We set 
$$
\boldsymbol{v}_0' := \alpha_1 \boldsymbol{v}_1
+\alpha_2 \boldsymbol{v}_2,\quad 
\boldsymbol{v}_1' = f \boldsymbol{v}_1, \quad
 \text{and} \quad  
 \boldsymbol{v}_2' := f  \boldsymbol{v}_2.
$$
Let $l'$ be an $\mathcal{O}_{n [t]}$-submodule of $l$
generated by $\boldsymbol{v}_0',\boldsymbol{v}_1'$, and $\boldsymbol{v}_2'$.
We may check that $\mathrm{length} (l/l')=1$.
Then we have an $\mathcal{O}_{n [t]}$-submodule of $l$ 
corresponding to the point $[\alpha_1:\alpha_2] \in \mathbb{P}^1$.
When $a_1(\lambda) \neq 0$, the type of the submodule $l'$ corresponding to $[0:1]$ 
is $(1^{a_1(\lambda)-1}, 2^{a_2(\lambda)})$,
and the type of the submodule $l'$ corresponding to $[1:\alpha_2/\alpha_1]$
is $(1^{a_1(\lambda)+1}, 2^{a_2(\lambda)-1})$.
When $a_1 (\lambda)= 0$, the type of the submodule $l'$ corresponding to $[\alpha_1:\alpha_2]$
is $(1^{a_1(\lambda)+1}, 2^{a_2(\lambda)-1})$.
So we have a map from $\mathbb{P}^1$ to 
the set of $\mathcal{O}_{n [t]}$-submodules $l'$ of $l$
with $\mathrm{length} (l/l')=1$.

We will construct the inverse map as follows.
Let $l'$ be an 
$\mathcal{O}_{n [t]}$-submodule of $l$
and $\mathrm{length} (l/l')=1$.
Let $\lambda'$ be the type of $l'$.
First we assume that $a_2(\lambda)>a_2(\lambda')$.
We have $a_1(\lambda')=a_1(\lambda)+1$ and $a_2(\lambda')=a_2(\lambda)-1$.
In this case, 
we may take an element $\boldsymbol{w}$ of $l'$ 
such that 
$\boldsymbol{w}$ is nonzero in  
$V_{a_1(\lambda)+a_2(\lambda)+1}/V_{a_1(\lambda)+a_2(\lambda)}$.
Since elements of $l'$ are also elements of $l$, 
there exist complex numbers $\alpha_1$ and $\alpha_2$ 
such that $\boldsymbol{w} - \alpha_1 \boldsymbol{v}_1- \alpha_2 \boldsymbol{v}_2 
\in \langle f\boldsymbol{v}_1, f\boldsymbol{v}_2\rangle_{\O_{n[t]}}$.
Since $(\alpha_1,\alpha_2) \neq (0,0)$,
we have the ratio $[\alpha_1:\alpha_2] \in \mathbb{P}^1$ by $\boldsymbol{w}$.
But we may check that the ratio $[\alpha_1:\alpha_2] \in \mathbb{P}^1$ is unique 
for the choice of such an element $\boldsymbol{w}$.
Then the ratio $[\alpha_1:\alpha_2] \in \mathbb{P}^1$ corresponds to $l'$.
Second we assume that $a_2(\lambda)=a_2(\lambda')$. 
We have $a_1(\lambda')=a_1(\lambda)-1$.
In particular, $a_1(\lambda) \neq 0$.
In this case, 
we may take an element $\boldsymbol{w}$ of $l'$ 
such that 
$\boldsymbol{w}$ is nonzero in  
$V_{a_1(\lambda)+a_2(\lambda)}/V_{a_1(\lambda)+a_2(\lambda)-1}$.
We may check that 
$\boldsymbol{w}
\in \langle f\boldsymbol{v}_1, \boldsymbol{v}_2\rangle_{\O_{n[t]}}$.
Since $\mathrm{length} (l/l')=1$,
there exists a nonzero complex number $\alpha$   
such that $\boldsymbol{w} -  \alpha \boldsymbol{v}_2 
\in \langle f\boldsymbol{v}_1, f\boldsymbol{v}_2\rangle_{\O_{n[t]}}$.
So we have the ratio $[0:\alpha] \in \mathbb{P}^1$ by $\boldsymbol{w}$.
But we may check that the ratio $[0:\alpha] \in \mathbb{P}^1$ is unique 
for the choice of such an element $\boldsymbol{w}$.
Then the ratio $[0:\alpha] \in \mathbb{P}^1$ corresponds to $l'$.
By these arguments, we have the inverse map.
\end{proof}

Let $\lambda^{(n)} = (1^{a_1}, 2^{a_2})$ be a Young diagram
and $l_n$ be an $\mathcal{O}_{n [t]}$-submodule of $E|_{n[t]}$ with type $\lambda^{(n)}$.
Before we consider families and parameter spaces of refined parabolic structures,
we discuss families and parameter spaces of filtrations 
$l_{n} \supset l_{n-1}  \supset \cdots \supset l_{n-m+1}$ $(2\leq  m\leq n)$
with $\length(l_{n-k+1}/l_{n-k})=1$ ($k=1,2,\ldots,m-1$).
If $m=n$, then this filtration is a refined parabolic structure.
We consider examples with $m=3$.
We will consider the following four sequences of Young diagrams:
\begin{align}
\nonumber &T \colon \lambda^{(n)}=(1^{a_1}, 2^{a_2}) \supset (1^{a_1-1}, 2^{a_2})
\supset (1^{a_1-2}, 2^{a_2}) \\ 
\nonumber&T' \colon \lambda^{(n)}=(1^{a_1}, 2^{a_2}) \supset (1^{a_1+1}, 2^{a_2-1})
\supset (1^{a_1}, 2^{a_2-1}) \\
\nonumber&T'' \colon \lambda^{(n)}=(1^{a_1}, 2^{a_2}) \supset (1^{a_1-1}, 2^{a_2})
\supset (1^{a_1}, 2^{a_2-1})\\
\nonumber&T''' \colon \lambda^{(n)}=(1^{a_1}, 2^{a_2}) \supset (1^{a_1+1}, 2^{a_2-1})
\supset (1^{a_1+2}, 2^{a_2-2})  .
\end{align}
Here we assume that 
\begin{itemize}
\item $a_1\geq 2$ when we consider the sequence $T$,
\item $a_2\geq 1$ when we consider the sequences $T'$, 
\item $a_1\geq 1$ and $a_2\geq 1$ when we consider the sequences $T''$, and 
\item $a_2\geq 2$ when we consider the sequence $T'''$.
\end{itemize}
We will construct families of 
filtrations with fixed $l_n$ corresponding to these sequences of Young diagrams.
Here for a sequence of Young diagrams,
the ``corresponding'' means that 
filtrations parametrized by a Zariski open subset of the parameter space 
have this sequence of Young diagrams.
We take generators $\boldsymbol{v}_1$ and $\boldsymbol{v}_2$ of $l_n$ such that 
$\mathrm{length} \left( \langle \boldsymbol{v}_1  \rangle_{\mathcal{O}_{n [t]}} \right) =a_1+a_2$
and
$\mathrm{length} \left( \langle \boldsymbol{v}_2  \rangle_{\mathcal{O}_{n [t]}} \right)=a_2$.

\begin{example}
First we will construct a family corresponding to $T$.
This is the easiest case.
Filtrations with $T$ are unique.
This filtration is 
$$
l_n= \langle  \boldsymbol{v}_1, \ 
 \boldsymbol{v}_2 \rangle_{\mathcal{O}_{n [t]}}
\supset 
 \langle f \boldsymbol{v}_1, \ 
 \boldsymbol{v}_2 \rangle_{\mathcal{O}_{n [t]}}
\supset 
 \langle f^2 \boldsymbol{v}_1, \ 
 \boldsymbol{v}_2 \rangle_{\mathcal{O}_{n [t]}}.
$$
So the parameter space corresponding to $T$ is a point.
\end{example}

\begin{example}
Second we will construct a family corresponding to $T'$ as follows.
We define a family of 
filtrations with fixed $l_n$:
$$
\begin{aligned}
l_{n-1} &= \langle \alpha_1\boldsymbol{v}_1 +
\alpha_2 \boldsymbol{v}_2 ,\ 
f \boldsymbol{v}_1 , \ 
f \boldsymbol{v}_2 \rangle_{\mathcal{O}_{n [t]}} 
\quad  \\
l_{n-2} &= \langle \alpha_1f \boldsymbol{v}_1 +
\alpha_2 f \boldsymbol{v}_2 ,\ 
f \boldsymbol{v}_1 , \ 
f \boldsymbol{v}_2 \rangle_{\mathcal{O}_{n [t]}} 
= \langle f \boldsymbol{v}_1, \ 
 f \boldsymbol{v}_2 \rangle_{\mathcal{O}_{n [t]}},
\end{aligned}
$$
which is parametrized by $[\alpha_1 : \alpha_2] \in \mathbb{P}^1$.
We denote by $C_{l_n,T'}$ this parameter space $\mathbb{P}^{1}$.
If $\alpha_1 \neq 0$, 
the sequences of types of the parametrized filtrations are $T'$.
So filtrations parametrized by the affine line $\mathbb{A}^{1}$
have the sequence $T'$.
If $\alpha_1=0$, 
the sequence of types of the parametrized filtration is $T''$.
So we have the parameter space $C_{l_n,T'}$ corresponding to $T'$.
\end{example}

\begin{example}\label{2022.1.15.10.24}
Third we will construct a family corresponding to $T''$ as follows.
We define a family of 
filtrations with fixed $l_n$:
$$
l_{n-1}' = \langle f \boldsymbol{v}_1 , \ 
 \boldsymbol{v}_2 \rangle_{\mathcal{O}_{n [t]}} 
\quad  \text{and} \quad 
l_{n-2}' = \langle \alpha'_1f \boldsymbol{v}_1 +
\alpha'_2  \boldsymbol{v}_2 , \ 
f^2 \boldsymbol{v}_1 ,\ 
f \boldsymbol{v}_2 \rangle_{\mathcal{O}_{n [t]}} ,
$$
which is parametrized by $[\alpha'_1 : \alpha'_2] \in \mathbb{P}^1$.
So we have the parameter space $C_{l_n,T''}$ corresponding to $T''$.
This parameter space $C_{l_n,T''}$ is $\mathbb{P}^{1}$.

We may check that 
$[0:1] \in C_{l_n,T'}$ and $[1:0] \in C_{l_n,T''}$ parametrize the same filtration:
$$
l_{n-1}' = \langle f \boldsymbol{v}_1 , \ 
 \boldsymbol{v}_2 \rangle_{\mathcal{O}_{n [t]}} 
\quad  \text{and} \quad
l_{n-2}' = \langle f \boldsymbol{v}_1 ,  \ 
 f \boldsymbol{v}_2 \rangle_{\mathcal{O}_{n [t]}} .
$$
So the parameter spaces $C_{l_n,T'}$ and $C_{l_n,T''}$ are the projective lines 
and the intersection $C_{l_n,T'}\, \cap \, C_{l_n,T''}$ is a point.
\end{example}

\begin{prop}
Let $\mathbb{F}_2 \rightarrow \mathbb{P}^1$ be the Hirzebruch surface
of degree two. 
We denote by $s_0$ and $p_{\infty}$ the $0$-section and the fiber 
of $\mathbb{F}_2 \rightarrow \mathbb{P}^1$ at $\infty$, respectively.
We can construct a family of filtrations with fixed $l_n$ 
corresponding to $T'''$.
More precisely,
by this family, 
$s_0$ is identified with $C_{l_n,T'}$,
$p_{\infty}$ is identified with $C_{l_n,T''}$,
and $\mathbb{F}_2 \setminus (s_0 \cup p_{\infty})$ is identified with 
$$
\left\{l_n \supset l_{n-1} \supset l_{n-2} 
\ \middle|\  
\begin{array}{l}
\text{$l_{n-k}$ is an $\mathcal{O}_{n [t]}$-submodule of $l_{n-k+1}$} \\
\text{such that $\length(l_{n-k+1}/l_{n-k})=1$ $(k=1,2)$}\\
\text{and the sequence of Young diagram is $T'''$} 
\end{array}
\right\} .
$$
\end{prop}

\begin{proof}
Set $U_0:= \{ [1:x_{1}]  \mid x_1 \in \mathbb{C}\}\subset \mathbb{P}^1$ 
and $U_{\infty}:= \{ [y_{1}:1] \mid y_1 \in\mathbb{C} \}\subset \mathbb{P}^1$. 
We consider $\mathbb{F}_2$ as the gluing of $U_0 \times \mathbb{P}^1$
and $U_{\infty} \times \mathbb{P}^1$ as follows:
$$
\begin{aligned}
(U_0 \times \mathbb{P}^1)|_{U_0 \cap U_{\infty}}
&\longrightarrow (U_\infty \times \mathbb{P}^1)|_{U_0 \cap U_{\infty}} \\
([1:x_1], [\alpha_{1,3}:\alpha_{2,3}])  &\longmapsto
([y_1:1], [\beta_{1,3}: \beta_{2,3}])
\end{aligned}
$$
where $y_1=1/x_1$, $\beta_{1,3}=\alpha_{1,3}$,
and $\beta_{2,3}=-x_1^{-2} \alpha_{2,3}$.
We will construct a family of 
filtrations 
$l_{n} \supset l_{n-1} \supset l_{n-2}$ 
with $\length(l_{n-k+1}/l_{n-k})=1$ ($k=1,2$) 
parametrized by $\mathbb{F}_2$ as follows.
We take generators $\boldsymbol{v}_1$ and $\boldsymbol{v}_2$ of $l_n$ such that 
$\mathrm{length} \left( \langle \boldsymbol{v}_1  \rangle_{\mathcal{O}_{n [t]}} \right) =a_1+a_2$
and
$\mathrm{length} \left( \langle \boldsymbol{v}_2  \rangle_{\mathcal{O}_{n [t]}} \right)=a_2$.
For $([1:x_1], [\alpha_{1,3}:\alpha_{2,3}])  \in U_0 \times \mathbb{P}^1$,
we define $l_{n-1}$ and $l_{n-2}$ as
$$
\begin{aligned}
l_{n-1} &= \left\langle \boldsymbol{v}_1+x_1 \boldsymbol{v}_2 ,\  f \boldsymbol{v}_2 \right\rangle   \\
l_{n-2} &= \left\langle \alpha_{1,3}(\boldsymbol{v}_1+x_1 \boldsymbol{v}_2)  
+\alpha_{2,3} f \boldsymbol{v}_2 , \ 
f \boldsymbol{v}_1+x_1 f \boldsymbol{v}_2, \ 
f^2 \boldsymbol{v}_2 \right\rangle .
\end{aligned}
$$
On the other hand, for 
$([y_1:1], [\beta_{1,3}: \beta_{2,3}]) \in U_\infty \times \mathbb{P}^1$
we define $l'_{n-1}$ and $l'_{n-2}$ as
$$
\begin{aligned}
l'_{n-1} &= \left\langle y_1 \boldsymbol{v}_1+ \boldsymbol{v}_2 ,\  f \boldsymbol{v}_1 \right\rangle   \\
l'_{n-2} &= \left\langle \beta_{1,3}(y_1\boldsymbol{v}_1+ \boldsymbol{v}_2)  
+\beta_{2,3} f \boldsymbol{v}_1 , \ 
 y_1f \boldsymbol{v}_1+ f \boldsymbol{v}_2, \ 
f^2 \boldsymbol{v}_1 \right\rangle .
\end{aligned}
$$
On $U_0 \cap U_{\infty}$, we put $y_1=1/x_1$, 
$\beta_{1,3}=\alpha_{1,3}$,
and $\beta_{2,3}=-x_1^{-2} \alpha_{2,3}$.
Then we have that $l_{n-1}|_{U_0 \cap U_{\infty}}=l_{n-1}'|_{U_0 \cap U_{\infty}}$ 
and $l_{n-2}|_{U_0 \cap U_{\infty}}=l_{n-2}'|_{U_0 \cap U_{\infty}}$.
So we obtain a family of filtrations parametrized by $\mathbb{F}_2$.
Points on the Zariski open subset $\alpha_{1,3} \neq 0$ in $U_0 \times \mathbb{P}^1$
parametrize filtrations whose sequences of types are $T'''$.
So we say that the parameter space $\mathbb{F}_2$
corresponds to $T'''$.
The curve $C_{l_n,T'}$ defined above 
is embedded into $\mathbb{F}_2$ as the $0$-section.
Here the 
$0$-section is defined by $\alpha_{1,3} =0$ over $U_0$ and 
$\beta_{1,3} =0$ over $U_\infty$.
The curve $C_{l_n,T''}$ defined above 
is embedded into $\mathbb{F}_2$ as
the fiber of the point $y_1=0$ under $\mathbb{F}_2 \rightarrow \mathbb{P}^1$.
\end{proof}

We observe that the dimension of the parameter space
corresponding to a sequence of Young diagrams
 $\lambda^{(n)} \supset
\lambda^{(n-1)} \supset \lambda^{(n-2)}$ is
$$
\# \left\{ k \ \middle| \ 
 n-1 \leq k \leq n , \  a_2(\lambda^{(k-1)})=a_2(\lambda^{(k)})-1 \right\}.
$$

Next we consider an example with $m=4$.
That is, we will construct
a family of 
filtrations 
$l_{n} \supset l_{n-1} \supset l_{n-2} \supset l_{n-3}$ 
with $\length(l_{n-k+1}/l_{n-k})=1$ ($k=1,2,3$).

\begin{prop}\label{prop:2021.11.18.22.25}
We can construct a $\mathbb{P}^1$-bundle over $\mathbb{F}_2$
such that this bundle has a $0$-section and
we can construct a family corresponding to 
$$
T^4 \colon \lambda^{(n)}=(1^{a_1}, 2^{a_2}) \supset (1^{a_1+1}, 2^{a_2-1})
\supset (1^{a_1+2}, 2^{a_2-2}) \supset (1^{a_1+3}, 2^{a_2-3})
$$
over this $\mathbb{P}^1$-bundle.
Here we assume that $a_2 \geq 3$.
\end{prop}

\begin{proof}
Let $l_n \supset \tilde l_{n-1} \supset  \tilde l_{n-2}$ be the family of filtrations 
parametrized by $\mathbb{F}_2$ defined in the case corresponding to $T'''$.
We set $U^{1}:= U_0\times U_0$, $U^{2}:=U_0\times U_{\infty}$,
$U^{3}:= U_{\infty}\times U_0$, $U^{4}:= U_{\infty}\times U_{\infty} $.
These products give an affine open covering of $\mathbb{F}_2$.
We consider the restriction $\tilde l_{n-2}|_{U^{i}}$ for $1\leq i \leq 4$.
Since we assume that $a_2 \geq 3$,
we may take two elements
 $\tilde{\boldsymbol{v}}^{i}_1$
and $\tilde{\boldsymbol{v}}^{i}_2$ of $\tilde l_{n-2}|_{U^{i}}$ such that 
 $\tilde l_{n-2}|_{U^{i}}=\langle \tilde{\boldsymbol{v}}^{i}_1 ,
 \tilde{\boldsymbol{v}}^{i}_2\rangle_{\O_{n[t]}}$.
 For each $U^{i}$, 
 we define a family of filtrations 
 $l_n \supset  \tilde l_{n-1}^{i} \supset \tilde l_{n-2}^{i} \supset  \tilde l_{n-3}^{i} $
(so $m=4$) parametrized by $U^{i} \times \mathbb{P}^1$ as follows.
We set $ \tilde l_{n-1}^{i}:=\tilde l_{n-1}|_{U^{i}}$,
$ \tilde l_{n-2}^{i}:=\tilde l_{n-2}|_{U^{i}}$, and
\begin{equation}\label{eq:2021.8.27.15.02}
\tilde l_{n-3}^{i} := 
\left\langle \alpha^{i}_{1,4} \tilde{\boldsymbol{v}}^{i}_1
+\alpha^{i}_{2,4}  \tilde{\boldsymbol{v}}^{i}_2,\ 
f\tilde{\boldsymbol{v}}^{i}_1 ,\ 
f \tilde{\boldsymbol{v}}^{i}_2 \right\rangle_{\O_{n[t]}}.
\end{equation}
where $[\alpha^{i}_{1,4}:\alpha^{i}_{2,4}]  \in  \mathbb{P}^1$.
Then we glue $\{ U^{i} \times \mathbb{P}^1\}_{1\leq i \leq 4}$ so that 
these families patch together.
Then we have a $\mathbb{P}^1$-bundle over $\mathbb{F}_2$.
We denote by $ \mathbb{F}^{(4)}_2 \rightarrow \mathbb{F}_2$ 
this $\mathbb{P}^1$-bundle.
By the gluing the families,
we have a family of filtrations $ l_n \supset \tilde l_{n-1} \supset \tilde l_{n-2} \supset  \tilde l_{n-3}$
parametrized by $ \mathbb{F}^{(4)}_2$.
Points on a Zariski open subset of $ \mathbb{F}^{(4)}_2$
parametrize filtrations whose sequences of types are $T^4$.

We may check that $ \mathbb{F}^{(4)}_2 \rightarrow \mathbb{F}_2$ has the $0$-section.
We may assume that 
$\mathrm{length} ( \langle \tilde{\boldsymbol{v}}^{i}_1  \rangle_{\mathcal{O}_{n [t]}} ) 
\geq \mathrm{length} ( \langle \tilde{\boldsymbol{v}}^{i}_2  \rangle_{\mathcal{O}_{n [t]}} ) $ for each $i$.
Then this $0$-section is defined by $\alpha^{i}_{1,4} =0$ in \eqref{eq:2021.8.27.15.02} for each $i$ 
($1\leq i \leq 4$).
That is, by the patching of the families, the local sections $\alpha^{i}_{1,4} =0$ over $U^{i}$
patch together.
\end{proof}

To describe parameter spaces of refined parabolic structures 
$l_{n} \supset l_{n-1} \supset \cdots \supset l_{1} \supset 0$,
we consider sequences of maps, which are $\mathbb{P}^1$-bundles or identity maps as follows.
For $T^4$,
we define a sequence of maps:
$$
\mathbb{F}_{l_n,T^4}^{(4)} 
\xrightarrow{\ p_{3} \ }\
\mathbb{F}_{l_n,T^4}^{(3)} 
\xrightarrow{\ p_{2} \ }\mathbb{F}_{l_n,T^4}^{(2)} 
\xrightarrow{\ p_{1} \ }\mathbb{F}_{l_n,T^4}^{(1)} = \{ {\rm pt} \}
$$
so that $\mathbb{F}_{l_n,T^4}^{(2)}$ is the projective line,
$p_{2}$ is the Hirzebruch surface
of degree two,
and $p_{3}$ is the $\mathbb{P}^1$-bundle in Proposition \ref{prop:2021.11.18.22.25}.
We identify $\mathbb{F}_{l_n,T^4}^{(4)}$
with the parameter space of this family corresponding to $T^4$.
Let
$$
T^5 \colon
 \lambda=(1^{a_1}, 2^{a_2}) \supset (1^{a_1+1}, 2^{a_2-1})
\supset (1^{a_1+2}, 2^{a_2-2}) \supset (1^{a_1+1}, 2^{a_2-2}).
$$
For $T^5$,
we define 
$$
\mathbb{F}_{l_n,T^5}^{(4)} 
\xrightarrow{\ p_{3} \ }\
\mathbb{F}_{l_n,T^5}^{(3)} 
\xrightarrow{\ p_{2} \ }\mathbb{F}_{l_n,T^5}^{(2)} 
\xrightarrow{\ p_{1} \ }\mathbb{F}_{l_n,T^5}^{(1)} = \{ {\rm pt} \}
$$
so that $\mathbb{F}_{l_n,T^5}^{(2)}$ is the projective line,
$p_{2}$ is the Hirzebruch surface
of degree two, $\mathbb{F}_{l_n,T^5}^{(3)}=\mathbb{F}_{l_n,T^5}^{(4)}$, 
and $p_{3}$ is the identity map.
We can construct a family corresponding to $T^5$
over the $0$-section of the $\mathbb{P}^1$-bundle in Proposition \ref{prop:2021.11.18.22.25}.
So we may identify $\mathbb{F}_{l_n,T^5}^{(4)}$
with the parameter space of this family corresponding to $T^5$.

Now we consider families and parameter spaces of refined parabolic structures.

\begin{prop}\label{2022.1.15.10.25}
We set $\lambda^{(n)}=(1^{a_1},2^{a_2})$.
We fix
an $\O_{n[t]}$-submodule $l_{n}$ of $E|_{n[t]}$
with $\mathrm{length}(l_n)=n$ whose type is 
$\lambda^{(n)}$.
Let $\mathcal{T}_{l_n}$ be the set of
standard Young tableaus whose shapes are
the Young diagram $\lambda^{(n)}$.
For $T=(\lambda^{(n)},\lambda^{(n-1)},\ldots,
\lambda^{(1)}) \in \mathcal{T}_{l_n}$, we can construct 
a sequence of maps:
$$
\mathbb{F}_{l_n,T}^{(n)} \xrightarrow{\ p_{n-1} \ }
\mathbb{F}_{l_n,T}^{(n-1)} \xrightarrow{\ p_{n-2} \ } \cdots 
\xrightarrow{\ p_{2} \ }\mathbb{F}_{l_n,T}^{(2)} 
\xrightarrow{\ p_{1} \ }\mathbb{F}_{l_n,T}^{(1)} = \{ {\rm pt} \}
$$
such that 
\begin{itemize}
\item $\mathbb{F}_{l_n,T}^{(k+1)}  = \mathbb{F}_{l_n,T}^{(k)} $
and $p_k$ is the identity when $a_2(\lambda^{(n-k+1)}) = a_2(\lambda^{(n-k)} )$

\item $p_k$ is a $\mathbb{P}^1$-bundle when $a_2(\lambda^{(n-k+1)}) = a_2(\lambda^{(n-k)} ) +1$

\item there exists a family of refined parabolic structures 
$\tilde l_{\bullet,T} \colon 
E|_{n [t]} \supset l_{n} \supset \tilde l_{n-1,T} \supset \cdots \supset 
\tilde l_{1,T} \supset 0$
with the fixed $l_{n}$ parametrized by $\mathbb{F}_{l_n,T}^{(n)}$ such that 
on a Zariski open subset $\mathbb{C}^{a_2}\subset \mathbb{F}_{l_n,T}^{(n)}$
the restriction $\tilde l_{\bullet,T}|_{\mathbb{C}^{a_2}}$ gives a bijection
$$
\mathbb{C}^{a_2} \longrightarrow
\left\{
 l_{\bullet} \colon 
 l_{n} \supset l_{n-1} \supset \cdots \supset 
l_{1} \supset 0
 \ \middle| \ 
\begin{array}{l} 
 \text{$l_{\bullet}$ is a refined parabolic structure on $n[t]$} \\
 \text{whose standard tableau is $T$}
\end{array}
 \right\} .
$$

\end{itemize}
\end{prop}

\begin{proof}
We will construct a sequence of maps
inductively.
We assume that we have a 
sequence of maps
$$
\mathbb{F}_{l_n,T}^{(k)} \xrightarrow{\ p_{k-1} \ }
\mathbb{F}_{l_n,T}^{(k-1)} \xrightarrow{\ p_{k-2} \ } \cdots 
\xrightarrow{\ p_{2} \ }\mathbb{F}_{l_n,T}^{(2)} 
\xrightarrow{\ p_{1} \ }\mathbb{F}_{l_n,T}^{(1)} = \{ {\rm pt} \}
$$
and such a family $ l_n \supset \tilde l_{n-1,T} \supset \cdots \supset  \tilde l_{n-k+1,T}$
 of filtrations 
parametrized by $\mathbb{F}_{l_n,T}^{(k)}$.

First we consider the case $a_2(\lambda^{(n-k+1)})  \neq 0$.
We take an affine open covering $\{ U^i\}_{i}$ of $\mathbb{F}_{l_n,T}^{(k)}$
so that
we may take two elements
 $\tilde{\boldsymbol{v}}^{i}_1$
and $\tilde{\boldsymbol{v}}^{i}_2$ of $\tilde l_{n-k+1}|_{U^{i}}$ such that 
 $\tilde l_{n-k+1,T}|_{U^{i}}=\langle \tilde{\boldsymbol{v}}^{i}_1 ,
 \tilde{\boldsymbol{v}}^{i}_2\rangle_{\O_{n[t]}}$
 with $\mathrm{length} \langle \tilde{\boldsymbol{v}}^{i}_1 \rangle_{\O_{n[t]}}
\geq \mathrm{length} \langle 
 \tilde{\boldsymbol{v}}^{i}_2\rangle_{\O_{n[t]}}$.
 For each $U^{i}$, 
we set 
$$
\tilde l_{n-k}^{i} := 
\left\langle \alpha^{i}_{1,k+1} \tilde{\boldsymbol{v}}^{i}_1
+\alpha^{i}_{2,k+1}  \tilde{\boldsymbol{v}}^{i}_2,\ 
f\tilde{\boldsymbol{v}}^{i}_1 ,\ 
f \tilde{\boldsymbol{v}}^{i}_2 \right\rangle_{\O_{n[t]}}.
$$
where $[\alpha^{i}_{1,k+1}:\alpha^{i}_{2,k+1}]  \in  \mathbb{P}^1$.
We glue $\{ U^{i} \times \mathbb{P}^1\}_{ i }$ so that 
these families patch together.
Then we have a $\mathbb{P}^1$-bundle over $\mathbb{F}_{l_n,T}^{(k)}$.
If $a_2(\lambda^{(n-k+1)}) = a_2(\lambda^{(n-k)} ) $,
then we define $\mathbb{F}_{l_n,T}^{(k+1)}$ by the $0$-section of this $\mathbb{P}^1$-bundle
and define $\tilde l_{n-k,T}$ by the gluing of
$\langle  \tilde{\boldsymbol{v}}^{i}_2,\ 
f\tilde{\boldsymbol{v}}^{i}_1  \rangle_{\O_{n[t]}}$.
If $a_2(\lambda^{(n-k+1)}) = a_2(\lambda^{(n-k)} ) +1$,
then we define $\mathbb{F}_{l_n,T}^{(k+1)}$ by this $\mathbb{P}^1$-bundle
and define $\tilde l_{n-k,T}$ by the gluing of $\tilde l_{n-k}^{i}$.

Second we consider the case $a_2(\lambda^{(n-k+1)})= 0$.
In particular $a_2(\lambda^{(n-k+1)}) = a_2(\lambda^{(n-k)} ) $.
We take an affine open covering $\{ U^i\}_{i}$ of $\mathbb{F}_{l_n,T}^{(k)}$
so that
we may take an element
 $\tilde{\boldsymbol{v}}^{i}$ such that 
 $\tilde l_{n-k+1,T}|_{U^{i}}=\langle \tilde{\boldsymbol{v}}^{i}\rangle_{\O_{n[t]}}$.
 For each $U^{i}$, 
we set 
$\tilde l_{n-k}^{i} := 
\left\langle f\tilde{\boldsymbol{v}}^{i}_1 \right\rangle_{\O_{n[t]}}$.
We define $\mathbb{F}_{l_n,T}^{(k+1)}$ by $\mathbb{F}_{l_n,T}^{(k+1)}=\mathbb{F}_{l_n,T}^{(k)}$
and define $\tilde l_{n-k,T}$ by the gluing of $\tilde l_{n-k}^{i}$.
\end{proof}

So the parameter space of refined parabolic structures with fixed $l_n$
is 
$\bigcup_{T \in \mathcal{T}_{l_n}} \mathbb{F}_{l_n,T}^{(n)}$.
Remark that $\mathbb{F}_{l_n,T}^{(n)}$ is irreducible 
and the dimension of $\mathbb{F}_{l_n,T}^{(n)}$ is $a_2$ for each $T$.
If $a_2=0$ or $a_2=1$, 
we can describe $\bigcup_{T \in \mathcal{T}_{l_n}} \mathbb{F}_{l_n,T}^{(n)}$ easily.
First if $a_2=0$, that is, $\lambda^{(n)} = (1^n,2^0)$,
then $\# \mathcal{T}_{l_n}=1$ and 
$\mathbb{F}_{l_n,T}^{(n)}$ is a point.
So $\bigcup_{T \in \mathcal{T}_{l_n}} \mathbb{F}_{l_n,T}^{(n)}$ is a point.
This point corresponds to the filtration
$$
l_n \supset f \cdot l_n \supset f^2\cdot l_n \supset \cdots \supset
f^{n-1}\cdot l_{n} \supset 0.
$$
Next we consider the case $a_2=1$,
that is,
$\lambda^{(n)} := (0,1,\ldots,1,2)=(1^{n-2},2^1)$.
We define a standard tableau $T_k$ so that 
the corresponding sequence of Young diagrams is
\begin{equation}\label{eq:2021.8.15.18.46}
\begin{aligned}
\lambda^{(n)}=(1^{n-2},2^1) \supset (1^{n-3},2^1) \supset \cdots
\supset (1^{n-k-1},2^1) \\
\qquad \supset (1^{n-k},2^0) \supset (1^{n-k-1},2^0) \supset \cdots 
\supset (1^1,2^0) 
\end{aligned}
\end{equation}
for $k=1,2,\ldots,n-1$.
By Proposition \ref{2022.1.15.10.25} and the argument in Examlpe \ref{2022.1.15.10.24},
we obtain the following statement:
\begin{cor}\label{2022.1.15.10.27}
The irreducible component $\mathbb{F}^{(n)}_{l_n,T_k}$ corresponding to
$T_{k}$ is the projective line $ \mathbb{P}^1$.
The parameter space of refined parabolic structures with 
fixed $l_n$ is
$$
 \mathbb{F}^{(n)}_{l_n,T_1}
 \, \cup \, \mathbb{F}^{(n)}_{l_n,T_2} 
 \, \cup\, \cdots 
 \, \cup \, \mathbb{F}^{(n)}_{l_n,T_{n-1}},
$$
which is a chain of $(n-1)$ projective lines.
The intersection $\mathbb{F}^{(n)}_{l_n,T_{k_1}} 
\, \cap \, \mathbb{F}^{(n)}_{l_n,T_{k_2}}$
 is a point if $|k_1-k_2|=1$.
When $|k_1-k_2|\neq 1$, this intersection is empty.
\end{cor}

\section{Stability of refined parabolic bundles}\label{section:StabilityOfRefined}

In the previous section, we have defined refined parabolic bundles.  
We would like to consider a moduli space of refined parabolic bundles.
In order to obtain a good moduli space, we have to introduce
a stability condition of refined parabolic bundles.
For this purpose, first, we define the stability index as in \cite[Section 2]{LS}.
This stability condition is the same as the stability condition as in \cite{Yoko1}.
In particular, the notion of parabolic bundles is contained in the notion of refined parabolic bundles.
So by the definition of the stability condition for refined parabolic bundles,
we have a stability condition for parabolic bundles.

Let $(E, \{ l_{ i , \bullet} \}_{  i \in I } )$ be
a refined parabolic bundle. 
We fix weights 
$$
\boldsymbol{w} 
= (\boldsymbol{w}_1,\ldots , \boldsymbol{w}_{\nu}), \quad 
\boldsymbol{w}_i= (w_{i,n_i},\ldots ,w_{i,1}) \in [0,1]^{n_i}
$$
such that 
\begin{equation}\label{2021.11.25.14.09}
0 \leq w_{i,n_i} \leq \cdots  \leq w_{i,1} \leq 1
\end{equation}
 for any $i \in I$.

\begin{definition}
Let $L$ be a line subbundle of $E$.
We define the $\boldsymbol{w}$-stability index of $L$ to the real number
$$
\mathrm{Stab}_{\boldsymbol{w} } (L) := 
\deg (E) - 2 \deg(L) 
+  \sum_{i\in I} \sum_{k=1}^{n_i}  w_{i,k} \left(1 -2  \,
\mathrm{length}   
((l_{i,k}\cap L|_{n_i[t_i]})/(l_{i,k-1}\cap L|_{n_i[t_i]}))  \right)
$$
where $l_{i,0}=0$ for any $i$.
\end{definition}

If we set
$\tilde I :=   \{ (i,k) \mid i \in I ,\  k=1,2,\ldots,n_i \}$ and 
$$
\epsilon_{i,k}(L) :=
\begin{cases}
-1 & \text{when } \mathrm{length}   
((l_{i,k}\cap L|_{n_i[t_i]})/(l_{i,k-1}\cap L|_{n_i[t_i]})) \neq 0 \\
1 & \text{when } \mathrm{length}   
((l_{i,k}\cap L|_{n_i[t_i]})/(l_{i,k-1}\cap L|_{n_i[t_i]})) = 0
\end{cases},
$$
then we have the following equality:
$$
\mathrm{Stab}_{\boldsymbol{w} } (L) 
= \deg (E) - 2 \deg(L) 
+\sum_{(i,k) \in \tilde I  } \epsilon_{i,k}(L) \cdot w_{i,k}   .
$$

\begin{definition}\label{2022.3.15.21.27}
A refined parabolic bundle $(E,\boldsymbol{l})$ is {\rm $\boldsymbol{w}$-stable} 
(resp. {\rm $\boldsymbol{w}$-semistable}) if
for any subbundle $L \subset E$,
the following inequality holds:
$$
\mathrm{Stab}_{\boldsymbol{w}}(L)>0 \qquad (\text{resp. $\geq 0$}). 
$$
\end{definition}

Now we discuss on the stability condition by using the parabolic 
degree. 
First, we recall a definition of the parabolic degree.
We fix parabolic weights $\boldsymbol{\alpha} 
= (\boldsymbol{\alpha}_1,\ldots , \boldsymbol{\alpha}_{\nu}) $, where
each $ \boldsymbol{\alpha}_i$ is a tuple of real numbers
$\boldsymbol{\alpha}_i= ( \alpha_{i,n_i+1},\alpha_{i,n_i},\ldots, \alpha_{i,1})$
with the condition 
$0< \alpha_{i,n_i+1}<\alpha_{i,n_i} < \cdots < \alpha_{i,1} <1$.
We define an $\boldsymbol{\alpha}$-\textit{parabolic degree} of a line subbbundle $L$ of $E$ 
as
$$
\text{para-deg}_{\boldsymbol{\alpha} }(L) = \deg (L) 
+ \sum_{i \in I} \sum_{k=1}^{n_i+1} 
\alpha_{i,k} \length ((l_{i,k}\cap L|_{n_i[t_i]})/(l_{i,k-1}\cap L|_{n_i[t_i]})) ,
$$
where $ l_{i,n_i+1} :=E|_{n_i [t_i]}$.
We can define the stability of refined parabolic bundles by 
using the parabolic degree:
A refined parabolic bundle $(E,\boldsymbol{l})$ is stable if
for any subbundle $L \subset E$,
the following inequality holds:
$$
\frac{\text{para-deg}_{\boldsymbol{\alpha} }(E)}{\mathrm{rank} (E)}>
\frac{\text{para-deg}_{\boldsymbol{\alpha} }(L)}{\mathrm{rank} (L)}.
$$
Now we will discuss a relation between 
the stability defined by the parabolic degree
and 
the stability defined by the stability index.
We compute $\text{para-deg}_{\boldsymbol{\alpha} }(E)/\mathrm{rank} (E)-
\text{para-deg}_{\boldsymbol{\alpha} }(L)/\mathrm{rank} (L)$ as follows.
$$
\begin{aligned}
\frac{\text{para-deg}_{\boldsymbol{\alpha} }(E)}{\mathrm{rank} (E)}-
\frac{\text{para-deg}_{\boldsymbol{\alpha} }(L)}{\mathrm{rank} (L)}
&=
\frac{1}{2} \Big(
\deg (E) -2\deg (L) 
+ \sum_{i \in I} \sum_{k=1}^{n_i+1} 
\alpha_{i,k} \length ((l_{i,k})/(l_{i,k-1}))  \\
&\quad  -2 \sum_{i \in I} \sum_{k=1}^{n_i+1} 
\alpha_{i,k} \length ((l_{i,k}\cap L|_{n_i[t_i]})/(l_{i,k-1}\cap L|_{n_i[t_i]}))
\Big) \\
&=\frac{1}{2} \Big(
\deg (E) -2\deg (L) 
+ \sum_{i \in I} 
\left( \alpha_{i,n_i+1} n_{i}
+\sum_{k=1}^{n_i} 
\alpha_{i,k}  \right) \\
&\quad -2 \sum_{i\in I} 
\Big( \alpha_{i,n_i+1} \left(n_i - \sum_{k=1}^{n_i} \length ((l_{i,k}\cap L|_{n_i[t_i]})/(l_{i,k-1}\cap 
L|_{n_i[t_i]})) \right) \\
&\quad \quad+\sum_{k=1}^{n_i} 
\alpha_{i,k} \length ((l_{i,k}\cap L|_{n_i[t_i]})/(l_{i,k-1}\cap L|_{n_i[t_i]})) \Big)
\Big)\\
&=\frac{1}{2} \Big(
\deg (E) -2\deg (L) 
+ \sum_{i \in I} 
\left(  \sum_{k=1}^{n_i} 
(\alpha_{i,k} - \alpha_{i,n_i+1}) \right) \\
& \quad -2 \sum_{i \in I} 
\left(\sum_{k=1}^{n_i} 
(\alpha_{i,k} -\alpha_{i,n_i+1}) \length ((l_{i,k}\cap L|_{n_i[t_i]})/(l_{i,k-1}\cap L|_{n_i[t_i]})) \right)
\Big).
\end{aligned}
$$
So if we set 
$w_{i,k} := \alpha_{i,k} -\alpha_{i,m_i+1}$,
then we have the following equality 
$$
\frac{\text{para-deg}_{\boldsymbol{\alpha} }(E)}{\mathrm{rank} (E)}-
\frac{\text{para-deg}_{\boldsymbol{\alpha} }(L)}{\mathrm{rank} (L)}
=\frac{\mathrm{Stab}_{\boldsymbol{w}}(L)}{2}.
$$
Then the stability defined by the stability index is the same as
the stability defined by the parabolic degree.

\subsection{Tame and undecomposable refined parabolic bundles}\label{sect:2021.12.2.15.32}

In Section \ref{section:ParaBun},
we have defined undecomposable parabolic bundles and 
admissible parabolic bundles.
Here we define undecomposable refined parabolic bundles and 
admissible refined parabolic bundles.

\begin{definition}
Let $(E, \{ l_{ i , \bullet} \}_{ i \in I  } )$ be 
a refined parabolic bundle
of rank $2$ and of degree $d$.
We say that $(E, \{ l_{ i , \bullet} \}_{ i \in I  } )$ is {\rm decomposable} if
there exists a decomposition $E =L_1 \oplus L_2$,
where $L_1$ and $L_2$ are non trivial, and 
$l_{i,k} = l^{(1)}_{i,k} \oplus l^{(2)}_{i,k}$ for any $i=1,\ldots,\nu$ and $k=1,\ldots,n_i$,
where we set $l^{(1)}_{i,k} := l_{i,k}\cap {L_1|_{n_i[t_i]}}$
and $l^{(2)}_{i,k} := l_{i,k}\cap {L_2|_{n_i[t_i]}}$.
We say that $(E, \{ l_{ i , \bullet} \}_{ i \in I  } )$ is {\rm undecomposable} if
$(E, \{ l_{ i , \bullet} \}_{ i \in I } )$ is not decomposable.
\end{definition}

\begin{definition}
Let $(E, \{ l_{ i , \bullet} \}_{i \in I  } )$ be 
a refined parabolic bundle
of rank $2$ and of degree $d$.
We say $(E, \{ l_{ i , \bullet} \}_{ i \in I  } )$ is 
{\rm admissible}
if this refined parabolic bundle satisfies the following condition:
$$
\sum_{i \in I} \sum_{k=1}^{n_i} \mathrm{length}   
((l_{i,k}\cap L|_{n_i[t_i]})/(l_{i,k-1}\cap L|_{n_i[t_i]})) \leq n+\deg (E) -2 \deg(L)-2 
$$
for any line subbundles $L$ such that $\deg(E) \leq 2\deg(L)$.
\end{definition}

If $D$ is a reduced effective divisor,
then we have the fact that 
$(E, \{ l_{ i , \bullet} \}_{i \in I  } )$ 
is stable for a convenient choice of weights $\boldsymbol{w}$
if, and only if, it is undecomposable
(see \cite[Proposition 3.4]{LS}).
On the other hand, when $D$ is not necessarily reduced,
there exist examples of (refined) parabolic bundles $(E, \{ l_{ i , \bullet} \}_{i \in I  } )$ 
such that it is undecomposable and 
it is not stable for any weights $\boldsymbol{w}$.

\begin{example}\label{2021.1.12.12.45}
Let $D$ and $(E,\bold{l})$ be one of the four listed parabolic bundles 
in Lemma \ref{2022.2.2.22.32}.
Then $(E,\bold{l})$ is undecomposable 
and admissible.
But,
by the inequalities \eqref{2021.11.25.14.09},
the parabolic bundle
$(E,\bold{l})$ is not stable for any weights $\boldsymbol{w}$ for any cases 
in Lemma \ref{2022.2.2.22.32}.
In particular, there exist
parabolic bundles $(E, \{ l_{ i } \}_{i \in I  } )$ 
such that it is $\Lambda$-flat and 
that it is not stable for any weights $\boldsymbol{w}$.
(Here $\Lambda$ is defined as in Proposition \ref{prop:2021.11.30.22.07}).
\end{example}

Now we will find a necessary and sufficient condition
of a refined parabolic bundle $(E, \{ l_{ i , \bullet} \}_{i \in I  } )$
for the condition that there exist weights $\boldsymbol{w}$ such that 
$(E, \{ l_{ i , \bullet} \}_{i \in I  } )$ is $\boldsymbol{w}$-stable.
For this purpose, we will define {\it tame} refined parabolic bundles.
For a line subbundle $L \subset E$ and $i \in I$,
we define an integer $N_i (L)$ as 
\begin{equation}\label{eq:2021.12.28.10.22}
N_i (L) = \max_{k' \in \{1,2,\ldots, n_i \}} \left\{ \sum_{k=1}^{k'} \epsilon_{i,k}(L) \right\}.
\end{equation}
We set 
$I^+_{L} = \{ i \in I \mid N_i (L)>0\}$.
Moreover, 
we set
$$
k_{i,\text{max}}(L) :=
\max
\left\{ k' \in \{1,2,\ldots, n_i \} \ \middle| \ 
N_i(L) = \sum_{k=1}^{k'} \epsilon_{i,k}(L)
\right\},
$$
$$
K_{i,\text{max}}(L) :=
\{1,2,\ldots, k_{i,\text{max}}(L)
\}, \quad \text{and} \quad 
K^c_{i,\text{max}}(L) :=
\{1,2,\ldots, n_i\} \setminus K_{i,\text{max}}(L).
$$

\begin{definition}\label{eq:2021.12.31.13.46}
Let $(E, \{ l_{ i , \bullet} \}_{i \in I  } )$ be 
a refined parabolic bundle
of rank $2$ and of degree $d$.
We say $(E, \{ l_{ i , \bullet} \}_{ i \in I  } )$ is 
{\rm tame}
if this refined parabolic bundle satisfies the following condition:
\begin{itemize}
\item $I^+_{L}$ is not empty, and
\item $\displaystyle -\deg (E) +2 \deg(L) +1 \leq \sum_{i \in I^+_{L}} N_i(L) $
\end{itemize}
for any line subbundles $L$ such that $\deg(E) \leq 2\deg(L)$.
\end{definition}

If $E$ has a decomposition $E \cong \O(d_1) \oplus \O(d_2)$ where $d_1<d_2$,
then there exists a unique 
line subbundle $L$ such that $\deg(E) \leq 2\deg(L)$.
This line subbundle is the destabilizing bundle $\O(d_2)$.
When $d_1=d_2$, 
there also exists a 
line subbundle $L$ such that $\deg(E) \leq 2\deg(L)$.
Remark that this line bundle is not unique in this case.
Now, we will show the implication \eqref{eq:2021.12.31.16.07}, that is,
\begin{equation*}
\text{$(E, \{ l_{ i , \bullet} \}_{ i \in I  })$ is undecomposable and tame}
\, \Longrightarrow \, 
\text{$(E,\{ l_{ i , \bullet} \}_{ i \in I  })$ is admissible}.
\end{equation*}
Moreover we will show that, when $(E,\bold{l})$ is a parabolic bundle,
\begin{equation*}
\text{$(E,\bold{l})$ is simple}
\, \Longleftrightarrow \, 
\text{$(E,\bold{l})$ is undecomposable and tame}.
\end{equation*}

\begin{prop}\label{prop:2021.12.31.13.50}
If a refined parabolic bundle $(E, \{ l_{ i , \bullet} \}_{i \in I  } )$ 
is tame and undecomposable, then it is admissible.
\end{prop}

\begin{proof}
For each $i \in I^+_{L}$, we have an inequality 
$$
\begin{aligned}
N_i(L) &\leq  \sum_{k \in\{k\mid \epsilon_{i,k} (L)=1\}} \epsilon_{i,k} (L) 
=n_i - 
\sum_{k=1}^{n_i} \mathrm{length}   
((l_{i,k}\cap L|_{n_i[t_i]})/(l_{i,k-1}\cap L|_{n_i[t_i]})). \\
\end{aligned}
$$
Clearly $n_i - 
\sum_{k=1}^{n_i} \mathrm{length}   
((l_{i,k}\cap L|_{n_i[t_i]})/(l_{i,k-1}\cap L|_{n_i[t_i]})) \geq 0$.
Then we have that for any line subbundles $L$ such that $\deg(E) \leq 2\deg(L)$,
\begin{equation*}
\begin{aligned}
 -\deg (E) +2 \deg(L) +1 &\leq \sum_{i \in I^+_{L}} N_i(L)  \\
&\leq \sum_{i \in I^+_{L}} 
\left( n_i - \sum_{k=1}^{n_i} \mathrm{length}   
((l_{i,k}\cap L|_{n_i[t_i]})/(l_{i,k-1}\cap L|_{n_i[t_i]}))
\right) \\
&\leq \sum_{i \in I} 
\left( n_i - \sum_{k=1}^{n_i} \mathrm{length}   
((l_{i,k}\cap L|_{n_i[t_i]})/(l_{i,k-1}\cap L|_{n_i[t_i]}))
\right)
\end{aligned}
\end{equation*}
We set $E= \O(d_1) \oplus \O(d_2)$ (with $d_1 \leq d_2$) and $L=\O(d_2)$.
We assume that 
\begin{align}
&-\deg (E) +2 \deg(L) +1 = \sum_{i \in I^+_{L}} N_i(L), \label{2021.11.30.12.27} \\
&N_i(L) =  \sum_{k \in\{k\mid \epsilon_{i,k} (L)=1\}} \epsilon_{i,k} (L)  \qquad (\text{for any $i \in I^+_{L}$}),
\label{2021.11.30.12.27_2}
\end{align}
and 
\begin{equation}\label{2021.11.30.12.27_3}
n_i - 
\sum_{k=1}^{n_i} \mathrm{length}   
((l_{i,k}\cap L|_{n_i[t_i]})/(l_{i,k-1}\cap L|_{n_i[t_i]})) = 0 \qquad (\text{for any $i \not\in I^+_{L}$}).
\end{equation}
By the equalities \eqref{2021.11.30.12.27_2} and \eqref{2021.11.30.12.27_3}, we have
\begin{equation}\label{2021.11.30.12.27_4}
\epsilon_{i,k} (L) = 
\begin{cases}
 1 & \text{for any $i \in I^+_{L}$ and any $k \in K_{i,\text{max}} (L)$} \\
-1 & \text{for any $i \in I^+_{L}$ and any $k \in K^c_{i,\text{max}} (L)$} \\
 -1 & \text{for any $i \not\in I^+_{L}$ and any $k \in \{ 1,2,\ldots ,n_i \}$}
\end{cases}.
\end{equation}
By the equality \eqref{2021.11.30.12.27}, we have the 
following equality: 
\begin{equation}\label{2021.11.30.12.27_5}
d_2 - d_1 +1= \sum_{i \in I^+_{L}}\# K_{i,\text{max}} (L) =
\sum_{i \in I^+_{L}}k_{i,\text{max}} (L).
\end{equation}
If $d_1=d_2$, then $\sum_{i \in I^+_{L}}k_{i,\text{max}} (L)=1$.
So $\# I^+_{L} =1$ and $K_{i_0,\text{max}} (L) =\{1\}$ for the $i_0 \in I^+_{L}$.
We have that $\epsilon_{i,k} (L) = 1$ when $(i,k)=(i_0,1)$ (with $i_0 \in I^+_{L}$) and 
$\epsilon_{i,k} (L) = -1$ otherwise. 
This implies that $(E, \{ l_{ i , \bullet} \}_{i \in I  } )$ is decomposable.
It is a contradiction. So $d_1\neq d_2$.
Then the vector bundle $\O(d_1)\oplus \O(d_2)$ has 
the following automorphism:
$$
\begin{pmatrix}
1&0\\ 
G&c
\end{pmatrix},
$$
where $c \in \mathbb{C}^*$ and $G=G(x)$ be a polynomial of degree
less than or equal to $d_2-d_1$.
By \eqref{2021.11.30.12.27_4} and \eqref{2021.11.30.12.27_5},
there exists an automorphism of $\O(d_1)\oplus \O(d_2)$
such that 
$l_{i,k} = l^{(1)}_{i,k} \oplus l^{(2)}_{i,k}$ for any $i=1,\ldots,\nu$ and $k=1,\ldots,n_i$.
Here we set $l^{(1)}_{i,k} := l_{i,k}\cap {\O(d_1)|_{n_i[t_i]}}$
and $l^{(2)}_{i,k} := l_{i,k}\cap {\O(d_2)|_{n_i[t_i]}}$.
Since $(E, \{ l_{ i , \bullet} \}_{i \in I  } )$ is undecomposable,
it is a contradiction.
So we have
$$
-\deg (E) +2 \deg(L) +1 
< \sum_{i \in I} 
\left( n_i - \sum_{k=1}^{n_i} \mathrm{length}   
((l_{i,k}\cap L|_{n_i[t_i]})/(l_{i,k-1}\cap L|_{n_i[t_i]}))
\right).
$$
This inequality implies that $(E, \{ l_{ i , \bullet} \}_{i \in I  } )$ is admissible.
\end{proof}

\begin{prop}\label{prop:2022.2.23.20.13}
Let $(E, \{ l_{ i} \}_{i \in I  } )$ be a parabolic bundle 
(Definition \ref{2022.2.23.22.18})
and $(E, \{ l_{ i , \bullet} \}_{i \in I  } )$ be the corresponding 
refined parabolic bundle.
The parabolic bundle $(E, \{ l_{ i} \}_{i \in I  } )$ is simple if, and only if, 
$(E, \{ l_{ i , \bullet} \}_{i \in I  } )$ is undecomposable and tame.
\end{prop}

\begin{proof}
First we will show that 
$(E, \{ l_{ i , \bullet} \}_{i \in I  } )$ is undecomposable and tame
when $(E, \{ l_{ i} \}_{i \in I  } )$ is simple.
If $(E, \{ l_{ i} \}_{i \in I  } )$ is simple,
$(E, \{ l_{ i , \bullet} \}_{i \in I  } )$ is undecomposable
by Proposition \ref{Prop:FlatnessImplications}.
So we will show that 
$(E, \{ l_{ i} \}_{i \in I  } )$ is not simple 
when $(E, \{ l_{ i , \bullet} \}_{i \in I  } )$ is not tame.
We will show that there is a non-zero nilpotent element of $\mathrm{End}(E, \{ l_{ i} \}_{i \in I  } )$.
Since $(E, \{ l_{ i , \bullet} \}_{i \in I  } )$ is not tame, 
there exists a line subbundle $L_0$ with $\deg (E) \leq \deg (L_0)$ 
such that $I^+_{L_0}$ is empty or 
$-\deg (E) +2\deg (L_0) \geq \sum_{ i \in I^+_{L_0}} N_i (L_0) $.
We decompose $E =L \oplus L_0$ with $\deg(L)\leq \deg(L_0)$.
We take a nilpotent element $N \in \mathrm{End}(L \oplus L_0)$ as follows:
$$
N
=\begin{pmatrix}
0 & 0 \\ 
f & 0 
\end{pmatrix}
$$
where $f=f(x)$ is a polynomial where $\deg (f) \leq 2 \deg (L_0) -\deg(E)$.
We set $m_i = \length( l_i \cap \O(L_0)|_{n_i[t_i]})$
and $I^+=\{ i \in I \mid n_i - 2m_i \geq 0 \}$.
Since $l_i$ is free,
we may check that 
$$
\epsilon_{i,k} (L_0)
= \begin{cases}
 1 & ( m_i < k \leq  n_i ) \\
 -1 & ( 1 \leq k \leq  m_i )
\end{cases} 
$$
for any $i \in I$.
So we have that 
$N_i(L_0) = n_i - 2m_i$ and $I^+ = I^+_{L_0}$.
First we assume that $I^+_{L_0}$ is empty.
Then, by using the equalities \eqref{2022.2.23.20.43}, 
we may check that there exists a non-zero polynomial $f$
such that $N$ preserves the parabolic structure 
$\{ l_{ i} \}_{i \in I  }$.
So if $I^+_{L_0}$ is empty, then $(E, \{ l_{ i} \}_{i \in I  } )$ is not simple.
Second we assume that $I^+\neq \emptyset$ and
$-\deg (E) +2\deg (L_0) \geq \sum_{ i \in I^+_{L_0}} N_i (L_0) $. 
By this inequality, we have that 
\begin{equation}\label{2022.2.23.20.59}
\sum_{i \in I^+} (n_i-2 m_i)\le  2\deg (L_0) -\deg (E) .
\end{equation}
Then, by using the equalities \eqref{2022.2.23.20.43}, 
we may check that there exists a non-zero polynomial $f$
such that $N$ preserves the parabolic structure 
$\{ l_{ i} \}_{i \in I  }$.
Finally we obtain that 
$(E, \{ l_{ i} \}_{i \in I  } )$ is not simple 
when $(E, \{ l_{ i , \bullet} \}_{i \in I  } )$ is not tame.

Next we will show that 
$(E, \{ l_{ i} \}_{i \in I  } )$ is simple
when $(E, \{ l_{ i , \bullet} \}_{i \in I  } )$ is undecomposable and tame.
Since $(E, \{ l_{ i , \bullet} \}_{i \in I  } )$ is undecomposable, 
any endomorphism $A$ of the $(E, \{ l_{ i} \}_{i \in I  } )$ takes the form 
$A= c \cdot \mathrm{Id}_E + N$ with $c\in \mathbb{C}$ and $N$ nilpotent.
We assume that there exists a non-zero nilpotent element $N$ of 
$\mathrm{End}(E, \{ l_{ i} \}_{i \in I  } )$.
By using the argument as in the verification of the inequality \eqref{2022.2.23.23.50},
we may check that there exists a line subbundle $L_0$ with $\deg(E)\leq 2\deg(L_0)$
such that $\sum_{i \in I^+}(n_i-2 m_i)\le  2\deg (L_0) -\deg (E)$.
Here, $m_i$ and $I^+$ are defined as above for this $L_0$.
By the argument as above, 
we have that $N_i(L_0)=n_i -2m_i$ and $I^+_{L_0}=I^+$.
So we obtain that $I^+_{L_0}$ is empty or 
 the inequality $-\deg (E) +2\deg (L_0) \geq \sum_{ i \in I^+_{L_0}} N_i (L_0) $ holds.
Since $(E, \{ l_{ i , \bullet} \}_{i \in I  } )$ is tame, it is a contradiction.
So $(E, \{ l_{ i} \}_{i \in I  } )$ is simple. 
\end{proof}

\subsection{Proof of Theorem A}\label{2022.3.1.11.33}
Now, we start to show Theorem A 
(Theorem \ref{thm:2021.10.27.13.44} and Corollary \ref{cor:2022.2.24.0.20} below).

\begin{prop}\label{prop:2021.12.3.12.50_1}
If a refined parabolic bundle $(E, \{ l_{ i , \bullet} \}_{i \in I  } )$ 
is stable for a convenient choice of weights $\boldsymbol{w}$,
then it is tame and undecomposable.
\end{prop}

\begin{proof}
Assume that 
$(E, \{ l_{ i , \bullet} \}_{i \in I  } )$ is decomposable.
There exists a decomposition $E \cong L_1 \oplus L_2$ such that 
$\mathrm{Stab}_{\boldsymbol{w}} (L_1) + \mathrm{Stab}_{\boldsymbol{w}} (L_2) =0$
for any weights $\boldsymbol{w}$.
Then $(E, \{ l_{ i , \bullet} \}_{i \in I  } )$ is not stable 
for any weights $\boldsymbol{w}$.
Assume that 
$(E, \{ l_{ i , \bullet} \}_{i \in I  } )$ is not tame.
There exists a line subbundle $L \subset E$ with $\deg(E) \leq 2\deg(L)$ 
such that 
\begin{itemize}
\item $I^+_{L}$ is empty, or
\item $I^+_{L}$ is not empty and 
$ -\deg (E) +2 \deg(L)
\geq  \sum_{i \in I^+_{L}} N_i(L)$.
\end{itemize}
If $I^+_{L}$ is empty, then we have the inequality
$ \sum_{k=1}^{n_i} 
 \epsilon_{i,k} (L)\cdot w_{i,k} \leq 0$
for any $i \in I$,
since we have the inequalities \eqref{2021.11.25.14.09}.
So we have 
$$
\begin{aligned}
\mathrm{Stab}_{\boldsymbol{w}}(L ) &=
 \deg (E) -2 \deg(L)  +\sum_{i \in I} \sum_{k=1}^{n_i} 
 \epsilon_{i,k} (L)\cdot w_{i,k} \\
&\leq \deg (E) -2 \deg(L)  \leq 0.
\end{aligned} 
$$
Then $(E, \{ l_{ i , \bullet} \}_{i \in I  } )$ is not stable 
for any weights $\boldsymbol{w}$.
Now we assume that $I^+_{L}$ is not empty.
By the inequalities \eqref{2021.11.25.14.09}
and the definitions of $K_{i,\text{max}}(L)$ and $K^c_{i,\text{max}}(L)$,
we have the following inequalities for each $i \in I^+_{L}$:
$$
\begin{aligned}
\sum_{k \in K^c_{i,\text{max}}(L)} \epsilon_{i,k} (L)\cdot w_{i,k}\leq 0
\quad \text{and} \quad 
   \sum_{k \in K_{i,\text{max}}(L)} \epsilon_{i,k} (L)\cdot w_{i,k} \leq N_i(L).
\end{aligned}
$$
If $i \not\in I^+_{L}$, then $\sum_{k = 1}^{n_i} \epsilon_{i,k} (L)\cdot w_{i,k}\leq 0$.
By these inequalities, we have
$$
\begin{aligned}
\mathrm{Stab}_{\boldsymbol{w}}(L ) &=
 \deg (E) -2 \deg(L)  +\sum_{i \in I \setminus I^+_{L}} \sum_{k=1}^{n_i} 
 \epsilon_{i,k} (L)\cdot w_{i,k}\\
 &\quad + \sum_{i \in I^+_{L}} \sum_{k \in K^c_{i,\text{max}}(L)} \epsilon_{i,k} (L)\cdot w_{i,k}
+\sum_{i \in I^+_{L}}   \sum_{k \in K_{i,\text{max}}(L)} \epsilon_{i,k} (L)\cdot w_{i,k}  \\
&\leq \deg (E) -2 \deg(L)  +
\sum_{i \in I^+_{L}} N_i(L)  \leq 0.
\end{aligned} 
$$
Then $(E, \{ l_{ i , \bullet} \}_{i \in I  } )$ is not stable 
for any weights $\boldsymbol{w}$.
\end{proof}

\begin{prop}\label{prop:2021.12.3.12.50_2}
Assume that $E=\O(d_1)\oplus \O(d_2)$ with $d_1<d_2$.
If a refined parabolic bundle $(E, \{ l_{ i , \bullet} \}_{i \in I  } )$ 
is tame and undecomposable,
then it is stable for a convenient choice of weights $\boldsymbol{w}$.
\end{prop}

\begin{proof}
We assume that $(E, \{ l_{ i , \bullet} \}_{i \in I  } )$ is tame and undecomposable.
Since $d_1<d_2$, there exists a unique morphism $\varphi_{d_2} \colon \O(d_2) \hookrightarrow E$.
We set 
$$
N^c_i(\O(d_2))
:=
\begin{cases}
\sum_{k \in K^c_{i,\text{max}}(\O(d_2))} \epsilon_{i,k} (\O(d_2)) 
&\text{for $i\in I^+_{\O(d_2)}$} \\
\sum_{k =1}^{n_i} \epsilon_{i,k} (\O(d_2)) 
&\text{for $i \not\in I^+_{\O(d_2)}$}
\end{cases}.
$$
For the unique morphism $\varphi_{d_2} \colon \O(d_2) \hookrightarrow E$,
we define weights $\boldsymbol{w}=(\boldsymbol{w}_1,\ldots , \boldsymbol{w}_n)$ as follows:
$$
w_{i,k} = 
\begin{cases}
w & \text{for $i \in I^+_{\O(d_2)}$ and $k \in K_{i,\text{max}} (\O(d_2))$}  \\
w' &\text{for $i \in I^+_{\O(d_2)}$ and $k \in K^c_{i,\text{max}} (\O(d_2))$} \\
w' &\text{for $i \not\in I^+_{\O(d_2)}$ and $k \in \{1,2,\ldots,n_i\}$}
\end{cases}.
$$
Here $w$ and $w'$ are real numbers such that 
\begin{equation}\label{2021.11.29.20.44}
\begin{cases}
  w  > 
-\frac{\sum_{i \in I} N^c_i(\O(d_2))}{\sum_{i \in I^+_{\O(d_2)}} N_i(\O(d_2))}   \cdot  w' 
+ \frac{d_2 - d_1}{\sum_{i \in I^+_{\O(d_2)}} N_i(\O(d_2)} \\
  w  < 
-\frac{\sum_{i \in I} N^c_i(\O(d_2))-2}{\sum_{i \in I^+_{\O(d_2)}} N_i(\O(d_2))} \cdot  w' 
+ \frac{d_2 - d_1}{\sum_{i \in I^+_{\O(d_2)}} N_i(\O(d_2))}  
\end{cases} \quad \text{and} \quad
\begin{cases}
0<w' <w <1\\
w+w' <1
\end{cases}.
 \end{equation}
Remark that, since $(E, \{ l_{ i , \bullet} \}_{i \in I  } )$ is tame and $d_1<d_2$,
we have the following inequalities:
$$
0<\frac{d_2 - d_1}{\sum_{i \in I^+_{\O(d_2)}} N_i(\O(d_2))}<1
\quad \text{and} \quad 
-\frac{\sum_{i \in I} N^c_i(\O(d_2))}{\sum_{i \in I^+_{\O(d_2)}} N_i(\O(d_2))}<
-\frac{\sum_{i \in I} N^c_i(\O(d_2))-2}{\sum_{i \in I^+_{\O(d_2)}} N_i(\O(d_2))}.
$$
So there exist such real numbers $w$ and $w'$.

We will check that $(E, \{ l_{ i , \bullet} \}_{i \in I  } )$ is $\boldsymbol{w}$-stable
for such a weights $\boldsymbol{w}$.
For the subbundle $\mathcal{O}(d_2) \subset E$, we have the following inequality:
$$
\mathrm{Stab}_{\boldsymbol{w}}(\mathcal{O}(d_2) ) = d_1 -d_2  
+\left( \sum_{i \in I} N^c_i(\O(d_2)) 
\right) w'
+\left( \sum_{i \in I^+_{\O(d_2)}} N_i(\O(d_2)) 
\right)  w >0
$$
by the inequalities \eqref{2021.11.29.20.44}.
Next we calculate the stability index $\mathrm{Stab}_{\boldsymbol{w}}(\mathcal{O}(d_1) )$
for any subbundles $\mathcal{O}(d_1) \subset E$:
$$
\begin{aligned}
\mathrm{Stab}_{\boldsymbol{w}}(\mathcal{O}(d_1) ) &\geq 
d_2 -d_1 
-\left( \sum_{i \in I} N^c_i(\O(d_2)) 
\right) w'
-\left( \sum_{i \in I^+_{\O(d_2)}} N_i(\O(d_2)) 
\right)w +2 w'  >0
\end{aligned}
$$
by the inequalities \eqref{2021.11.29.20.44}.
Here in the first inequality, we used the assumption that $(E, \{ l_{ i , \bullet} \}_{i \in I  } )$
is undecomposable and $w'<w$.
Let $L$ be a subbundle $L \subset E$ such that $\deg(L) \leq d_1-1$.
We will 
calculate the stability index $\mathrm{Stab}_{\boldsymbol{w}}(L )$.
For the subbundle $L$,
we set
$$
\begin{aligned}
K_{\delta_1,\delta_2}^{(i,d_2)}(L) :=  \{ k \mid  k\in\{1,2,\ldots,n_i\} , \text{ and }
\epsilon_{i,k}(L) =\delta_1, \epsilon_{i,k}(\O(d_2)) =\delta_2 \},
\end{aligned}
$$
where $i\in I$ and $\delta_1,\delta_2 \in \{ +,-\}$. 
We set $k_{\delta_1,\delta_2}^{(i,d_2)}(L)=\#
K_{\delta_1,\delta_2}^{(i,d_2)}(L)$ for any $i \in I$, and
$$
\begin{aligned}
 k_{\delta_1,\delta_2}^{(i,d_2)}(L)'=\#
(K_{\delta_1,\delta_2}^{(i,d_2)}(L)  \cap K_{i,{\rm max}} (\mathcal{O} (d_2)))
&&\text{and}&&k_{\delta_1,\delta_2}^{(i,d_2)}(L)''=\#
(K_{\delta_1,\delta_2}^{(i,d_2)}(L)  \cap K^c_{i,{\rm max}} (\mathcal{O} (d_2)))
\end{aligned}
$$
for any $ i \in I^+_L$.
We have the following inequality:
$$
\begin{aligned}
&\sum_{i \in  I^+_{\mathcal{O} (d_2)}}
 \sum_{k=1}^{n_i} \epsilon_{i,k} (L) \cdot w_{i,k} 
+  \sum_{i \in I \setminus I^+_{\mathcal{O} (d_2)}}
 \sum_{k =1}^{n_i}\epsilon_{i,k} (L)\cdot  w' \\
 &= \sum_{i \in  I^+_{\mathcal{O} (d_2)}}
 ( k_{+,+}^{(i,d_2)}(L)' - k_{-,+}^{(i,d_2)}(L)' + k_{+,-}^{(i,d_2)}(L)'-k_{-,-}^{(i,d_2)}(L)' ) \cdot w \\
 &\qquad+ 
 \sum_{i \in  I^+_{\mathcal{O} (d_2)}}
  ( k_{+,+}^{(i,d_2)}(L)'' - k_{-,+}^{(i,d_2)}(L)'' + k_{+,-}^{(i,d_2)}(L)''-k_{-,-}^{(i,d_2)}(L)'' ) \cdot w' \\
&\qquad \quad +  \sum_{i \in I \setminus I^+_{\mathcal{O} (d_2)}}
 ( k_{+,+}^{(i,d_2)}(L) - k_{-,+}^{(i,d_2)}(L)+ k_{+,-}^{(i,d_2)}(L)-k_{-,-}^{(i,d_2)}(L) )\cdot  w' \\
 &\geq \sum_{i \in  I^+_{\mathcal{O} (d_2)}}
 ( -N_i(\O(d_2)) -2 k_{-,-}^{(i,d_2)}(L)' ) \cdot w
+  \sum_{i \in  I^+_{\mathcal{O} (d_2)}}
  ( -N^c_i(\O(d_2)) -2 k_{-,-}^{(i,d_2)}(L)'' ) \cdot w' \\
&\qquad +  \sum_{i \in I \setminus I^+_{\mathcal{O} (d_2)}}
 ( -N^c_i(\O(d_2)) -2k_{-,-}^{(i,d_2)}(L) )\cdot  w'
\end{aligned}
$$
By this inequality and $w' <w$, we have the following inequality:
\begin{equation}\label{eq:2022.1.8.21.07}
\begin{aligned}
\mathrm{Stab}_{\boldsymbol{w}}(L ) 
&\geq \deg(E) -2 \deg(L) \\
&\qquad -  \sum_{i \in I} N^c_i(\O(d_2)) w'  
-  \sum_{i \in I^+_{\O(d_2)}} N_i(\O(d_2)) w 
 - 2 \left(  \sum_{i \in I} k_{-,-}^{(i,d_2)}(L) \right)  w  
\end{aligned}
\end{equation}
We describe the subbundle $L \hookrightarrow \O(d_1) \oplus \O(d_2)$
by a vector
$(f_1^L(x) ,f_2^L(x))$,
where $f_1^L(x)$ and $f_2^L(x)$ are polynomials in $x$
such that $\deg(f_1^L(x))\leq d_1-\deg(L)$ and $\deg(f_2^L(x))\leq d_2-\deg(L)$.
We have
\begin{equation*}\label{eq:2021.12.26.15.32}
\left[ \text{the zero order of $f_1^L(x)$ at $t_i$} \right]  \geq \# 
\{  k\in\{1,2,\ldots,n_i\} \mid
\epsilon_{i,k}(L) =\epsilon_{i,k}(\O(d_2)) =-1 \}
\end{equation*}
for each $i\in I$.
By this inequality,
we have $d_1 - \deg(L) \geq   \sum_{i \in I} k_{-,-}^{(i,d_2)}(L)$.
When $\sum_{i \in I} k_{-,-}^{(i,d_2)}(L) \neq 0$
for the subbundle $L \subset E$, we have the following inequality:
$$
\begin{aligned}
\mathrm{Stab}_{\boldsymbol{w}}(L ) 
&\geq d_2 - d_1 
-  \sum_{i \in I} N^c_i(\O(d_2)) w'
-  \sum_{i \in I^+_{\O(d_2)}} N_i(\O(d_2)) w 
 +\left(  \sum_{i \in I} k_{-,-}^{(i,d_2)}(L) \right)  \cdot 2 (1-w)   \\
&\geq d_2 - d_1 
- \left( \sum_{i \in I} N^c_i(\O(d_2))\right) w'
- \left( \sum_{i \in I^+_{\O(d_2)}} N_i(\O(d_2))\right) w+2 (1-w) \\
&> 2(1- w -w')>0.
\end{aligned}
$$
Here in the first inequality, we used the inequality $d_1 - \deg(L) \geq
\sum_{i \in I} k_{-,-}^{(i,d_2)}(L)$ and the inequality \eqref{eq:2022.1.8.21.07}.
In the second inequality, we used the inequalities 
$\sum_{i \in I} k_{-,-}^{(i,d_2)}(L)\geq 1$ and $w<1$.
In the third inequality and in the last inequality, we used 
the inequalities in \eqref{2021.11.29.20.44}.
When $\sum_{i \in I} k_{-,-}^{(i,d_2)}(L)= 0$,
for the subbundle $L \subset E$, we have the following inequality:
$$
\begin{aligned}
\mathrm{Stab}_{\boldsymbol{w}}(L ) 
&\geq d_1 +d_2 -2 \deg(L)
- \left( \sum_{i \in I} N^c_i(\O(d_2))\right) w'
- \left( \sum_{i \in I^+_{\O(d_2)}} N_i(\O(d_2))\right) w  \\
&\geq d_2 - d_1 
- \left( \sum_{i \in I} N^c_i(\O(d_2))\right) w'
- \left( \sum_{i \in I^+_{\O(d_2)}} N_i(\O(d_2))\right) w+1 \\
& > 1-2w'>0.
\end{aligned}
$$
Here we used the inequality $\deg(L) \leq d_1-1$ in the second inequality.
Then we obtain that $(E, \{ l_{ i , \bullet} \}_{i \in I  } )$ is $\boldsymbol{w}$-stable
for the weights $\boldsymbol{w}$.
\end{proof}

\begin{prop}\label{prop:2021.12.3.12.50_3}
Assume that $E=\O(d)\oplus \O(d)$.
If a refined parabolic bundle $(E, \{ l_{ i , \bullet} \}_{i \in I  } )$ 
is tame and undecomposable,
then it is stable for a convenient choice of weights $\boldsymbol{w}$.
\end{prop}

\begin{proof}
We take an element $i_1 \in I$.
Since $E =\O(d) \oplus \O(d)$ and $\length(l_{i_1,1}) =1$,
we can take a line subbundle $L_1 \subset E$ 
with $\deg(L_1)=d$
such that $\epsilon_{i_1,1}(L_1) =-1$.
Since $(E, \{ l_{ i , \bullet} \}_{i \in I  } )$ is tame, 
there exists an element $i_2 \in I$
such that 
$$
N_{i_2}(L_1) =\sum_{k \in K_{i_2,\text{max}}(L_1)} \epsilon_{i_2,k} \geq 1.
$$
We divide into the following two cases:
\begin{itemize}
\item[(A)] the set $\{ k \in K_{i_2,\text{max}}(L_1) \mid
\epsilon_{i_2,k}( L_1) =1 , \epsilon_{i_2,k}( L) =-1 \}$
is empty for any line subbundles $L \subset E$ with $\deg (L) =d$, 
\item[(B)] the set $\{ k \in K_{i_2,\text{max}}(L_1) \mid
\epsilon_{i_2,k}( L_1) =1 , \epsilon_{i_2,k}( L) =-1 \}$
is not empty for some line subbundle $L \subset E$ with $\deg (L) =d$.
\end{itemize}
For the case (A),
we determine weights $\boldsymbol{w}$ as follows:
\begin{itemize}
\item[(A-i)] When $i_1\neq i_2$, 
\begin{equation*}
w_{i,k}:=
\begin{cases}
w_1 & \text{when $i=i_1$ and $k=1$}\\
w_2 & \text{when $i=i_2$ and $k\in K_{i_2,\text{max}}(L_1)$}\\
 0 & \text{otherwise} 
\end{cases}.
\end{equation*}
\item[(A-ii)] When $i_1= i_2$, 
\begin{equation*}
w_{i,k}:=
\begin{cases}
w_1 & \text{when $i=i_1$ and $k=1$}\\
w_2 & \text{when $i=i_1$ and $k\in K_{i_1,\text{max}}(L_1) \setminus\{1\}$}\\
 0 & \text{otherwise} 
\end{cases}
\end{equation*}
with $0<w_2 < w_1<1$.
\end{itemize}
For the case (B),
we determine weights $\boldsymbol{w}$ as follows.
We take a line subbundle $L_2$ of $E$ 
with $\deg(L_2)=d$
so that 
\begin{equation}\label{2021.11.29.10.45}
\# \{ k \in K_{i_2,\text{max}}(L_1) \mid
\epsilon_{i_2,k}( L_1) =1 , \epsilon_{i_2,k}( L_2) =-1 \}
\end{equation}
is maximized. 
In particular $L_1 \neq L_2$ in $E=\mathcal{O}(d) \oplus \mathcal{O}(d)$.
So it is impossible that $\epsilon_{i,k}(L_1) = \epsilon_{i,k}(L_2) =-1$
for any $(i,k) \in \tilde I$.
Since $(E, \{ l_{ i , \bullet} \}_{i \in I  } )$ is undecomposable, 
the set
\begin{equation}\label{2021.11.28.13.16}
\{ (i,k) \in \tilde I \mid \epsilon_{i,k}(L_1) = \epsilon_{i,k}(L_2) =1
\}
\end{equation}
is not empty.
We take an element $(i_3,k_3)$ of the set \eqref{2021.11.28.13.16}
such that the pair satisfies the equality 
 $k_3=\min \{ k \in \{1,2,\ldots,n_{i_3} \} \mid 
\epsilon_{i_3,k}(L_1) = \epsilon_{i_3,k}(L_2) =1 \}$.
If the subset 
\begin{equation}\label{2021.11.28.13.16_2}
\{ k \in \{1,2,\ldots, n_i \}  \mid 
\epsilon_{i_2,k}(L_1) = \epsilon_{i_2,k}(L_2) =1 \} \, \cap \, 
K_{i_2,\text{max}}(L_1)
\end{equation}
is not empty, we take the element $(i_3,k_3)$
from this subset \eqref{2021.11.28.13.16_2}.
So,
if $i_2\neq i_3$, we have that 
\begin{equation}\label{2022.1.8.22.49}
\epsilon_{i_2,k} (L_1)+
\epsilon_{i_2,k} (L_2)=0 \qquad (k \in K_{i_2,\text{max}}(L_1)).
\end{equation}
So, when $i_2\neq i_3$, or
when $i_2=i_3$ and $ k_3>k_{i_2,\text{max}}(L_1)$,
we have that 
\begin{equation}\label{2021.11.29.12.08}
\epsilon_{i_2,k} (L)=1 \qquad (k \in K_{i_2,\text{max}}(L_1))
\end{equation}
for any line subbundles $L$ such that $\deg(L) = d$, $L\neq L_1$, and $L\neq L_2$.

\begin{itemize}
\item[(B-i)] When $i_1,i_2,i_3$
are distinct from each other, we determine weights $\boldsymbol{w}$ as follows:
\begin{equation*}
w_{i,k}:=
\begin{cases}
w_1 & \text{when $i=i_1$ and $k=1$}\\
w_2 & \text{when $i=i_2$ and $k \in K_{i_2,\text{max}}(L_1)$}\\
w_3 & \text{when $i=i_3$ and $k=1,2,\ldots,k_3 -1$}\\
w_4 & \text{when $i=i_3$ and $k=k_3$}\\
 0 & \text{otherwise} 
\end{cases}
\end{equation*}
with $0<w_1,w_2<1$ and $0<w_4<w_3 <1$.

\item[(B-ii)] 
When $i_1\neq i_2=i_3$ and $k_{i_2,\text{max}}(L_1) \geq k_3$, 
we determine weights $\boldsymbol{w}$ as follows:
\begin{equation*}
w_{i,k}:=
\begin{cases}
w_1 & \text{when $i=i_1$ and $k=1$}\\
w_2 & \text{when $i=i_2$ and $k \in K_{i_2,\text{max}}(L_1)$}\\
 0 & \text{otherwise} 
\end{cases}
\end{equation*}
with $0<w_1,w_2<1$.

\item[(B-iii)] When $i_1\neq i_2=i_3$ and $ k_3>k_{i_1,\text{max}}(L_1)$, 
we determine weights $\boldsymbol{w}$ as follows:
\begin{equation*}
w_{i,k}:=
\begin{cases}
w_1 & \text{when $i=i_1$ and $k=1$}\\
w_2 & \text{when $i=i_2$ and $k\in K_{i_2,\text{max}}(L_1)$}\\
w_3 & \text{when $i=i_2$ and $k=k_{i_2,\text{max}}(L_1)+1,\ldots,k_3-1$}\\
w_4 & \text{when $i=i_2$ and $k=k_3$}\\
 0 & \text{otherwise} 
\end{cases}
\end{equation*}
with $0<w_4<w_3<w_2<1$ and $0<w_1<1$.

\item[(B-iv)]
When $i_1=i_3\neq i_2$, we determine weights $\boldsymbol{w}$ as follows:
\begin{equation*}
w_{i,k}:=
\begin{cases}
w_1 & \text{when $i=i_1$ and $k=1$}\\
w_2 & \text{when $i=i_2$ and $k\in K_{i_2,\text{max}}(L_1)$}\\
w_3 & \text{when $i=i_1$ and $k=2,3,\ldots,k_3 -1$}\\
w_4 & \text{when $i=i_1$ and $k=k_3$}\\
 0 & \text{otherwise} 
\end{cases}
\end{equation*}
with $0<w_2<1$ and $0<w_4<w_3<w_1 <1$.

\item[(B-v)] 
When $i_1= i_2\neq i_3$, 
we determine weights $\boldsymbol{w}$ as follows:
\begin{equation*}
w_{i,k}:=
\begin{cases}
w_1 & \text{when $i=i_1$ and $k=1$}\\
w_2 & \text{when $i=i_1$ and $k\in K_{i_1,\text{max}}(L_1) \setminus\{1\}$}\\
w_3 & \text{when $i=i_3$ and $k=1,\ldots,k_3-1$}\\
w_4 & \text{when $i=i_3$ and $k=k_3$}\\
 0 & \text{otherwise} 
\end{cases}
\end{equation*}
with $0<w_2<w_1<1$ and $0<w_4<w_3<1$.

\item[(B-vi)] 
When $i_1= i_2=i_3$ and $k_{i_1,\text{max}}(L_1) \geq k_3$, 
we determine weights $\boldsymbol{w}$ as follows:
\begin{equation*}
w_{i,k}:=
\begin{cases}
w_1 & \text{when $i=i_1$ and $k=1$}\\
w_2 & \text{when $i=i_1$ and $k\in K_{i_1,\text{max}}(L_1) \setminus\{1\}$}\\
 0 & \text{otherwise} 
\end{cases}
\end{equation*}
with $0<w_2<w_1<1$.

\item[(B-vii)]
When $i_1= i_2=i_3$ and $ k_3>k_{i_1,\text{max}}(L_1)$, 
we determine weights $\boldsymbol{w}$ as follows:
\begin{equation*}
w_{i,k}:=
\begin{cases}
w_1 & \text{when $i=i_1$ and $k=1$}\\
w_2 & \text{when $i=i_1$ and $k\in K_{i_1,\text{max}}(L_1) \setminus\{1\}$}\\
w_3 & \text{when $i=i_1$ and $k=k_{i_1,\text{max}}(L_1)+1,\ldots,k_3-1$}\\
w_4 & \text{when $i=i_1$ and $k=k_3$}\\
 0 & \text{otherwise} 
\end{cases}
\end{equation*}
with $0<w_4<w_3<w_2<w_1<1$.
\end{itemize}
If we take a convenient choice of $w_1,w_2,w_3,w_4$ (or $w_1,w_2$),
we may check that $(E, \{ l_{ i , \bullet} \}_{i \in I  } )$ is $\boldsymbol{w}$-stable
for the weights $\boldsymbol{w}$.
Now we check this claim for the 3 cases: (A-i); (B-i); and (B-vi).
For other cases, we will check $(E, \{ l_{ i , \bullet} \}_{i \in I  } )$ is $\boldsymbol{w}$-stable
for some weights $\boldsymbol{w}$ by the same argument as in the three cases.

We consider the case (A-i).
We assume that 
$$
 -w_1 + N_{i_2} (L_1) \cdot w_2  >0 .
$$
Since $N_{i_2} (L_1) \geq 1$, there exists a pair $w_1,w_2$ satisfies 
this inequality and $0<w_2,w_1 <1$.
We have the following inequality for $\mathrm{Stab}_{\boldsymbol{w}}(L_1)$:
$$
\begin{aligned}
\mathrm{Stab}_{\boldsymbol{w}}(L_1)
&= -w_1 + N_{i_2} (L_1)\cdot w_2  >0 .
\end{aligned}
$$
For any line subbundles $L \subset E$ with $L_1\neq L$,  
we have the following inequality:
$$
\begin{aligned}
\mathrm{Stab}_{\boldsymbol{w}}(L)
&= \deg(E) -2 \deg(L) + w_1  +  k_{i_2,\text{max}} (L_1) \cdot w_2 \\ 
&\geq  w_1  +  k_{i_2,\text{max}} (L_1) \cdot w_2 >0 .
\end{aligned}
$$
Then we obtain that $(E, \{ l_{ i , \bullet} \}_{i \in I  } )$ is $\boldsymbol{w}$-stable
for the weights $\boldsymbol{w}$.

We consider the case (B-i).
We assume that 
$$
\begin{cases}
-w_1 + N_{i_2} (L_1) w_2 + \left( \sum_{k'=1}^{k_3-1} \epsilon_{i_3,k'} (L_1) \right) w_3 +w_4 >0\\
-w_1 + N_{i_2} (L_1) w_2 + \left( \sum_{k'=1}^{k_3-1} \epsilon_{i_3,k'} (L_1) \right) w_3 - w_4 <0\\
w_1 + k_{i_2,\text{max}} (L_1) \cdot  w_2 + (k_3-1) w_3 -w_4  >0\\
w_1 + k_{i_2,\text{max}} (L_1) \cdot  w_2 + (k_3-1) w_3 + w_4 <2
\end{cases}.
$$
We may check that 
there exist such real numbers $w_1,w_2,w_3,w_4$.
We have the following inequalities for $\mathrm{Stab}_{\boldsymbol{w}}(L_1)$ 
and $\mathrm{Stab}_{\boldsymbol{w}}(L_2)$:
$$
\begin{aligned}
\mathrm{Stab}_{\boldsymbol{w}}(L_1)
&= -w_1 + N_{i_2} (L_1) w_2 + \left( \sum_{k'=1}^{k_3-1} \epsilon_{i_3,k'} (L_1) \right) w_3 +w_4 >0 \\
\mathrm{Stab}_{\boldsymbol{w}}(L_2)
&= w_1 - N_{i_2} (L_1) w_2 - \left( \sum_{k'=1}^{k_3-1} \epsilon_{i_3,k'} (L_1) \right) w_3 +w_4>0.
\end{aligned}
$$
Here for the equality in the computation of $\mathrm{Stab}_{\boldsymbol{w}}(L_2)$,
we used the equality
\eqref{2022.1.8.22.49} and the equality $\epsilon_{i_3,k}(L_1)+\epsilon_{i_3,k}(L_2)=0$
for $k=1,2,\ldots,k_3-1$.
This equality is given by the equality  $k_3=\min \{ k \in \{1,2,\ldots,n_{i_3} \} \mid 
\epsilon_{i_3,k}(L_1) = \epsilon_{i_3,k}(L_2) =1 \}$.
For any line subbundles $L$ such that $\deg(L) = d$, $L\neq L_1$, and $L\neq L_2$,
we have the following inequalities:
$$
\begin{aligned}
\mathrm{Stab}_{\boldsymbol{w}}(L)
&\geq
w_1 + k_{i_2,\text{max}} (L_1) \cdot  w_2 + (k_3-1) w_3 -w_4  >0.
\end{aligned}
$$
by the equality \eqref{2021.11.29.12.08}.
For any line subbundles $L$ with $\deg(L) \leq d-1$,
we have the following inequalities 
$$
\begin{aligned}
\mathrm{Stab}_{\boldsymbol{w}}(L)
&\geq \deg(E) -2 \deg(L) 
-w_1 -k_{i_2,\text{max}} (L_1) \cdot  w_2 - (k_3-1) w_3 -w_4  \\
&\geq 2 
-w_1 - k_{i_2,\text{max}} (L_1) \cdot  w_2 - (k_3-1) w_3 -w_4 >0.
\end{aligned}
$$
Then we obtain that $(E, \{ l_{ i , \bullet} \}_{i \in I  } )$ is $\boldsymbol{w}$-stable
for the weights $\boldsymbol{w}$.

We consider the case (B-vi).
We assume that 
$$
\begin{cases}
 -w_1 + (N_{i_2}(L_1) +1)  \cdot w_2  >0 \\
 w_1  -  (N_{i_2}(L_1) -1)\cdot  w_2 >0 \\
 w_1  + (k_{i_2,\text{max}} (L_1) -1) \cdot  w_2 <2
 \end{cases}.
$$
Since $N_{i_2} (L_1) \geq 1$, there exists a pair $w_1,w_2$ satisfies 
this inequality and $0<w_2,w_1 <1$.
We have the following inequality for $\mathrm{Stab}_{\boldsymbol{w}}(L_1)$:
$$
\begin{aligned}
\mathrm{Stab}_{\boldsymbol{w}}(L_1)
&= -w_1 + (N_{i_2}(L_1) +1)\cdot w_2  >0 ,
\end{aligned}
$$
since $\epsilon_{i_1,1} (L_1) =-1$ and 
$N_{i_2}(L_1)= \sum_{k \in K_{i_2 , \text{max}} (L_1)} \epsilon_{i_2,k} (L_1)$.
For any line subbundles $L \subset E$ with $\deg(L) =d$ and $L\neq L_1$,  
we have the following inequality:
$$
\begin{aligned}
\mathrm{Stab}_{\boldsymbol{w}}(L) \geq \mathrm{Stab}_{\boldsymbol{w}}(L_2)
&\geq \deg(E) -2 \deg(L_2) +w_1  +  (-N_{i_2}(L_1) +1)\cdot  w_2  \\ 
&=  w_1  -  (N_{i_2}(L_1) -1)\cdot  w_2 >0 .
\end{aligned}
$$
Here for the first inequality, we used the fact that $L_2$ maximizes
$\# \{ k \in K_{i_2,\text{max}}(L_1) \mid
\epsilon_{i_2,k}( L_1) =1 , \epsilon_{i_2,k}( L_2) =-1 \}$ and 
for the second inequality we used $k_{i_1,\text{max}}(L_1) \geq k_3$.
For any line subbundles $L \subset E$ with $\deg(L) \leq d-1$,  
we have the following inequality:
$$
\begin{aligned}
\mathrm{Stab}_{\boldsymbol{w}}(L) 
&\geq \deg(E) -2 \deg(L) -w_1  - (k_{i_2,\text{max}} (L_1) -1) \cdot  w_2  \\ 
&\geq 2 -w_1  - (k_{i_2,\text{max}} (L_1) -1) \cdot  w_2 >0 .
\end{aligned}
$$
Then we obtain that $(E, \{ l_{ i , \bullet} \}_{i \in I  } )$ is $\boldsymbol{w}$-stable
for the weights $\boldsymbol{w}$.
\end{proof}

By Proposition \ref{prop:2021.12.3.12.50_1}, 
Proposition \ref{prop:2021.12.3.12.50_2},
and Proposition \ref{prop:2021.12.3.12.50_3}, 
we obtain the following theorem:

\begin{theorem}\label{thm:2021.10.27.13.44}
A refined parabolic bundle $(E, \{ l_{ i , \bullet} \}_{i \in I  } )$ is tame and undecomposable if, 
and only if, it is stable for a convenient choice of weights $\boldsymbol{w}$.
\end{theorem}

By Proposition \ref{prop:2022.2.23.20.13},
we have the following corollary:
\begin{corollary}\label{cor:2022.2.24.0.20}
Let $(E, \{ l_{ i} \}_{i \in I  } )$ be a parabolic bundle 
(Definition \ref{2022.2.23.22.18})
and $(E, \{ l_{ i , \bullet} \}_{i \in I  } )$ be the corresponding 
refined parabolic bundle.
The parabolic bundle $(E, \{ l_{ i} \}_{i \in I  } )$ is simple if, and only if, 
$(E, \{ l_{ i , \bullet} \}_{i \in I  } )$ is 
stable for a convenient choice of weights $\boldsymbol{w}$
\end{corollary}

\section{Elementary transformation of refined parabolic bundles}

For studying parabolic connections or parabolic bundles (where $D$ is a reduced effective divisor),
the elementary transformations play an important role.
In this section, we will define elementary transformations for refined parabolic bundles
and will show some properties of the elementary transformations.

\subsection{Elementary transformations for $\Lambda$-connections}
First we recall the elementary transformations for $\Lambda$-connections.
Let $(E,\nabla)$ be a $\Lambda$-connection.
For $(E,\nabla)$ we may define transformations as follows.
Let $\bold{l} = \{ l_i \}_{i \in I}$ be the corresponding parabolic structure of $(E,\nabla)$.
That is, $l_i$ is a free submodule of $E|_{n_i[t_i]}$.
We take an integer $k$ where $1\leq k \leq n_i$.
Let 
$\varphi_{k} \colon E|_{n_i[t_i]} \rightarrow E|_{k[t_i]}$
be the natural morphism.
For the free submodule $\varphi_{k}(l_i)$ of $E|_{k[t_i]}$, 
the vector bundle $E'_{\varphi_{k}(l_i)}$ is defined by the exact sequence of sheaves
$$
0 \longrightarrow  E'_{\varphi_{k}(l_i)} \longrightarrow E \longrightarrow 
E|_{k [t_i]}/\varphi_{k}(l_i) \longrightarrow 0.
$$
Note that $E|_{k [t_i]}/\varphi_{k}(l_i)$ is a skyscraper sheaf supported on $t_i$.
The degree of $E'_{\varphi_{k}(l_i)}$ is $\deg(E) - k$.
Let $\nabla_{\varphi_{k}(l_i)}$ be the induced connection by
the morphism $E'_{\varphi_{k}(l_i)} \rightarrow E$.
So we have a transformation $(E'_{\varphi_{k}(l_i)},\nabla_{\varphi_{k}(l_i)})$.
Let $\Lambda'_{i,k}$ be the formal data 
of $(E'_{\varphi_{k}(l_i)},\nabla_{\varphi_{k}(l_i)})$ (see Definition \ref{2022_9_29_12_06}
for the definition of the formal data).
By this transformation, we have a map 
$\CON_{\Lambda}(D) \rightarrow \CON_{\Lambda'_{i,k}}(D)$.

\subsection{Elementary transformations for parabolic bundles}
Second we recall the elementary transformations for parabolic bundles.
Let $(E,\{ l_i \}_{i \in I})$ be a parabolic bundle.
For $(E,\{ l_i \}_{i \in I})$ we may define transformations as follows.
For the free submodule $l_i$ of $E|_{n_i[t_i]}$, 
the vector bundle $E'_i$ is defined by the exact sequence of sheaves
$$
0 \longrightarrow  E'_{i} \longrightarrow E \longrightarrow 
E|_{n_i [t_i]}/l_i \longrightarrow 0.
$$
The degree of $E'_i$ is $\deg(E) - n_i$.
We have a filtration of sheaves
$$
\xymatrix{
E'_i \otimes \mathcal{O}(-n_i[t_i]) \ar[r]^-{\subset} 
&E\otimes \mathcal{O}(-n_i[t_i]) \ar[r]^-{\subset} & E'_i \ar[r]^-{\subset} & E. 
}
$$
We define a parabolic structure $l_{i}'$ on $E'_i$ by   
$$
l_{i}' = (E\otimes \mathcal{O}(-n_i[t_i]))/(E'_i \otimes \mathcal{O}(-n_i[t_i])) .
$$
So we have a transformation 
$(E'_{i},\{ l_{j}\}_{j \in I\setminus \{ i\}} \cup \{ l_{i}'\} )$.
We call this transformation the {\it elementary transformation} 
of a parabolic bundles $(E,\{ l_i \}_{i \in I})$ at $n_i[t_i]$.
We have the elementary transformations for $\Lambda$-connections and parabolic bundles: 
$$ 
(E,\nabla) \mapsto (E'_{\varphi_{k}(l_i)},\nabla_{\varphi_{k}(l_i)}) \quad
\text{and}
\quad  (E,\{ l_i \}_{i \in I}) \mapsto (E'_{i},\{ l_{j}\}_{j \in I\setminus \{ i\}} \cup \{ l_{i}'\} ),
$$
respectively.
On the other hand, we have a correspondence 
$(E,\nabla) \mapsto (E,\{ l_i \}_{i \in I})$.
If $k=n_i$, then these transformations are compatible with this correspondence.

\subsection{Definition of elementary transformations 
for refined parabolic bundles}\label{section:2021.12.31.15.32}
Now we consider the extension of the elementary transformation 
of parabolic bundles.
That is, we will define the elementary transformation 
for refined parabolic bundles.
For this purpose, we rephrase refined parabolic structures 
as filtrations of sheaves.
By this rephrasing, we may define the elementary transformation 
for refined parabolic bundles simply. 
Let $(E , \bold{l})$ be a refined parabolic bundle,
where $\bold{l} =\{ l_{i,\bullet } \}_{i \in I}$.
We take $i_0 \in I$
and consider the filtration $l_{i_0 ,\bullet}\colon 
E|_{n_{i_0} [t_{i_0}]} \supset l_{i_0,n_{i_0}} \supset l_{i_0,n_{i_0}-1} \supset \cdots \supset 
l_{i_0,1} \supset 0 $.
We set 
$$
E_{i_0}^{(k)} := \ker (E \longrightarrow E|_{n_{i_0}[t_{i_0}]} / l_{i_0,k})
$$
for $k=0,1,2, \ldots,n_{i_0}$.
Then we have a filtration of sheaves 
\begin{equation}\label{eq:FiltSheavesATt_i0}
E \supset E^{(n_{i_0})}_{i_0} \supset
 E^{(n_{i_0}-1)}_{i_0} \supset \cdots \supset 
E_{i_0}^{(1)} \supset E_{i_0}^{(0)} =  E(-n_{i_0}[t_{i_0}])
\end{equation}
where $\mathrm{length} (E / E_{i_0}^{(n_{i_0})})= n_{i_0}$ and 
$\mathrm{length} (E_{i_0}^{(k)} / E_{i_0}^{(k-1)})= 1$ for $k=1,2,\ldots, n_{i_0}$.
Conversely, 
if we set $l_{i_0,k} := E_{i_0}^{(k)} /  E_{i_0}^{(0)}$ for $k=0,1,2,\ldots,n_{i_0}$, 
then we have 
a filtration 
$$
E / E_{i_0}^{(0)} = E|_{n_i [t_i]} \supset l_{i_0,n_{i_0}} \supset 
l_{i_0,n_{i_0}-1} \supset \cdots \supset 
l_{i_0,1} \supset 0 
$$
of $\mathcal{O}_{n_{i_0} [t_{i_0}]}$-modules where the length of $l_{i_0,k}$ is $k$.

Now we will transform the filtration \eqref{eq:FiltSheavesATt_i0}
into a new filtration of sheaves.
We set $E_{i_0}'= E^{(n_{i_0})}_{i_0}$.
Since $\mathrm{length} (E_{i_0}^{(k)} / E_{i_0}^{(k-1)})= 1$ for $k=1,2,\ldots, n_{i_0}$, 
we may check that $E_{i_0}^{(k-1)} \supset E_{i_0}^{(k)} (-[t_{i_0}])$ for $k=1,2,\ldots, n_{i_0}$.
So we can define a filtration of sheaves as 
$$
\begin{aligned}
E'_{i_0}&=E^{(n_{i_0})}_{i_0} 
\supset E_{i_0}^{(0)} \supset E_{i_0}^{(1)}(-[t_{i_0}])  
\supset E_{i_0}^{(2)}(-2[t_{i_0}])  \supset 
\cdots \\
&\quad \cdots   \supset 
E_{i_0}^{(n_{i_0}-1)} (-(n_{i_0}-1)[t_{i_0}]) 
\supset E_{i_0}^{(n_{i_0})}(-n_{i_0}[t_{i_0}]).
\end{aligned}
$$
This is a new filtration of sheaves.
For $k=0,1,\ldots, n_{i_0}$,
we set 
$$
(E'_{i_0})^{(n_{i_0}-k)} := E^{(k)} (-k [t_{i_0}]) .
$$
We will check that $\mathrm{length} (E'_{i_0} / (E_{i_0}')^{(n_{i_0})})= n_{i_0}$ and 
$\mathrm{length} ((E_{i_0}')^{(k)} / (E_{i_0}')^{(k-1)})= 1$ for $k=1,2,\ldots, n_{i_0}$
as follows.
For $k=1,2,\ldots, n_{i_0}$, 
we have the following equalities:
$$
\begin{aligned}
&\mathrm{length} (E_{i_0}^{(k)}/E_{i_0}^{(k-1)})+
\mathrm{length} ((E_{i_0}')^{(n_{i_0}-k+1)}/(E_{i_0}')^{(n_{i_0}-k)} ) \\
&=\mathrm{length} (E_{i_0}^{(k)}/E_{i_0}^{(k-1)})+
\mathrm{length} (E_{i_0}^{(k-1)}(-(k-1) [t_{i_0}])/E_{i_0}^{(k)}(-k [t_{i_0}]) ) \\
& =
\mathrm{length} (E_{i_0}^{(k)}/E_{i_0}^{(k)}(-[t_{i_0}]))=2.
\end{aligned}
$$
Since $\mathrm{length} (E_{i_0}^{(k)}/E_{i_0}^{(k-1)})=1$ for $k=1,2,\ldots, n_{i_0}$,
we have 
$$
\mathrm{length} ((E_{i_0}')^{(n_{i_0}-k+1)}/(E_{i_0}')^{(n_{i_0}-k)} ) =1
$$
for $k=1,2,\ldots, n_{i_0}$.
Moreover we may check that 
$\mathrm{length} (E_{i_0}'/(E_{i_0}')^{(n_{i_0})} ) =n_{i_0}$.
Finally, 
if we set $l_{i_0,k}' := (E_{i_0}')^{(k)} /  (E_{i_0}')^{(0)}$ for $k=0,1,2,\ldots,n_{i_0}$, 
then we have 
a filtration 
\begin{equation}\label{eq:ElmOfRefinedStr}
 E|_{n_i [t_i]} \supset l'_{i_0,n_{i_0}} \supset 
l'_{i_0,n_{i_0}-1} \supset \cdots \supset 
l'_{i_0,1} \supset 0 
\end{equation}
of $\mathcal{O}_{n_{i_0} [t_{i_0}]}$-modules where the length of $l'_{i_0,k}$ is $k$.

\begin{definition}
We fix $i_0$ where $ i_0 \in I$.
We define the {\rm elementary transformation} of 
$(E, \{ l_{ i,\bullet} \}_{ i \in I  })$ at $n_{i_0}[t_{i_0}]$ by
$$
\elm^-_{i_0}(E, \{ l_{i,\bullet} \}_{ i \in I} )
=(E'_{i_0},\{ l_{i,\bullet}\}_{i \in I\setminus \{ i_0\}} \cup \{ l_{i_0,\bullet}'\} ).
$$
The degree of $E'_{i_0}$ is $d-n_{i_0}$.
\end{definition}

\subsection{Some properties of elementary transformations 
for refined parabolic bundles}\label{section:2021.12.31.15.34}
We take $i \in I$.
Let $l_{i,\bullet} \colon l_{i,n_i} \supset \cdots \supset l_{i,1} \supset 0$ be a refined parabolic structure 
at $n_i[t_i]$
and $l_{i,\bullet}'$ be its transformation. 
For $l_{i,\bullet}$ and $l_{i,\bullet}'$, we can define 
standard tableaus $T$ and $T'$, respectively 
(see Definition \ref{def:StandTab}).
We will consider the relation between these standard tableaus $T$ and $T'$. 
Let $\lambda_{i}^{(k)}$ be the Young diagram
which is the type of $l_{i,k}$.
We define integers $a^k_1$ and $a^k_2$ so that
$\lambda_{i}^{(k)}=(1^{a^k_1},2^{a^k_2})$
for $k=1,2,\ldots,n_i$.
We set 
$$
X_1(k):=  a^{n_{i}}_1+a^{n_{i}}_2 -a^{k}_1-a^{k}_2 
\quad \text{and}
\quad X_2(k):=  a^{n_{i}}_2  -a^{k}_2 .
$$

\begin{definition}
We say a refined parabolic structure $l_{i,\bullet}$ is {\rm generic} if 
\begin{itemize}
\item $a_2^{n_i} =0$, or 

\item $a_2^{n_i} \neq 0$ and we can take generators $\boldsymbol{v}_1, \boldsymbol{v}_2 \in l_{i,n_i}$
such that
$\length (\langle \boldsymbol{v}_1 \rangle)=a^{n_{i}}_1+a^{n_{i}}_2$,
$\length (\langle \boldsymbol{v}_2 \rangle)=a^{n_{i}}_2$, 
and 
$l_{i,\bullet}$ satisfies
$$
l_{i,k}/ \left\langle f_i^{X_1(k)}\boldsymbol{v}_1,\ 
f_i^{X_2(k)} \boldsymbol{v}_2 \right\rangle  \neq 0
$$
for any $k$ with $a_2^{k+1} -a_2^{k} > 0$.
\end{itemize}
\end{definition}

Now we consider examples of generic $l_{i,\bullet}$.
Set $\lambda_i^{(n_i)}= (1^{n_i-2},2^1)$.
So $a_1^{n_i}=n_i-2$ and $a_2^{n_i}=1$.
Let $T_{k}$ be the standard tableau defined in \eqref{eq:2021.8.15.18.46}
for $k=1,2,\ldots,n-1$.
We fix $k_0$ ($1\leq k_0 \leq n-1$).
If the standard tableau of a refined parabolic bundle is $T_{k_0}$,
then $a_2^{n-k_0+1}-a_2^{n-k_0}=1>0$ and $a_2^{k+1}-a_2^{k}=0$ when $k\neq n- k_0$.
We take elements $\boldsymbol{v}_1, \boldsymbol{v}_2 \in E|_{n[t]}$
such that
$\length (\langle \boldsymbol{v}_1 \rangle)=n_i-1$, 
and $\length (\langle \boldsymbol{v}_2 \rangle)=1$,
and $l_{i,n_i}=\langle \boldsymbol{v}_1, \boldsymbol{v}_2 \rangle_{\mathcal{O}_{n [t]}}$.
In particular, $f_i^{n_i-1} \boldsymbol{v}_1=f_i \boldsymbol{v}_2=0$.
We can describe 
a refined parabolic structure as follows:
$$
\begin{aligned}
l_{i,n_i}=&\langle \boldsymbol{v}_1 , \boldsymbol{v}_2 \rangle 
\supset \langle f_i \boldsymbol{v}_1 , \boldsymbol{v}_2 \rangle
\supset \langle f_i^2 \boldsymbol{v}_1 , \boldsymbol{v}_2 \rangle \supset \cdots 
\supset \langle f_i^{k_0-1} \boldsymbol{v}_1 , \boldsymbol{v}_2 \rangle  \\
&\qquad
\supset \langle  f_i^{k_0-1} \boldsymbol{v}_1 +\alpha \boldsymbol{v}_2 \rangle 
\supset  \langle  f_i^{k_0} \boldsymbol{v}_1  \rangle \supset \cdots
\supset \langle  f_i^{n_i-2} \boldsymbol{v}_1 \rangle,
\end{aligned}
$$
where $\alpha \in \mathbb{C}$.
This refined parabolic structure is generic if, and only if, 
$\alpha \neq 0$.

To describe the standard tableau of 
the transformation $l_{i,\bullet}'$,
first, we define integers $b_{n_i},b_{n_i-1},\ldots,b_1$ by
$b_{n_i}= a_1^{n_i}$
and $b_k = \min \{ b_{k+1}, a_{1}^k \}$
for $k=n_i-1,\ldots,2,1$.
Second, we modify $X_1(k)$ and $X_2(k)$ by using $b_{k}$
as follows:
$$
X^b_1(k):=  X_1(k) +a^{k}_1-b_{k}  \quad
\text{and}
\quad X^b_2(k):=  X_2(k) -a^{k}_1+b_{k}  .
$$
We claim that $X^b_2(k)\geq 0$.
Indeed, $X^b_2(k)=a_2^{n_i}-a_1^k -a_2^k + b_{k}$.
When $a_1^k +a_2^k>a_2^{n_i}$, 
we have $\min\{ a^{n-1}_1,a_1^{n-2},\ldots,a_1^k\} \geq a_1^k +a_2^k - a_2^{n_i}$
since we have the inclusions 
$l_{i,n_{i}}\supset l_{i,n_{i}-1} \supset l_{i,n_{i}-2} \supset \cdots \supset l_{i,k} $.
Then we obtain $X^b_2(k)\geq 0$.

\begin{prop}\label{prop:2021.9.9.9.53}
Let $l'_{i,\bullet}$ be the transformation \eqref{eq:ElmOfRefinedStr} of $l_{i,\bullet}$.
Assume that $l_{i,\bullet}$ is generic.
Then the type of $l'_{i,n_i-k}$ is the following Young diagram:
$$
\begin{cases}
(1^{X_2^b(k)-X_1^b(k)},2^{X_1^b(k)}) & \text{when $X^b_2(k)-X^b_1(k)\geq 0$}\\
(1^{X_1^b(k)-X_2^b(k)},2^{X_2^b(k)}) & \text{when $X^b_2(k)-X^b_1(k)<0$}.
\end{cases}
$$
\end{prop}

\begin{proof}
When $a_2^{n_i}=0$, it is clear.
So we consider the case where $a_2^{n_i}\neq 0$.
We take generators $\boldsymbol{v}_1, \boldsymbol{v}_2$ of $l_{i,n_i}$
as in the definition of the genericness of $l_{i,\bullet}$.
We can take $F(f_i) \in \O_{n_i[t_i]}$ (where $F(0)\neq0$)
so that 
\begin{equation}\label{eq:2021.9.8.17.54}
l_{i,k} = \left\langle 
f_i^{X_1(k)} \boldsymbol{v}_1 + f_i^{X_2^b(k)}  F(f_i) \boldsymbol{v}_2, \ 
f_i^{X_2(k)}  \boldsymbol{v}_2
\right\rangle.
\end{equation}
Indeed, if $k=n_i$, then
$$
\left\langle 
f_i^{X_1(k)} \boldsymbol{v}_1 + f_i^{X_2^b(k)}  F(f_i) \boldsymbol{v}_2, \ 
f_i^{X_2(k)}  \boldsymbol{v}_2
\right\rangle = \langle \boldsymbol{v}_1, \ \boldsymbol{v}_2  \rangle
$$
since $X_1(n_i) =X_2(n_i)=0$ and $b_{n_i} =a_1^{n_i}$.
We fix $k$ ($1\leq k \leq n_i$).
We assume that $l_{i,k}$ has the description \eqref{eq:2021.9.8.17.54}
for the fixed $k$.
Since $X_2^b(k) \leq X_2(k)$ and $F(0)\neq0$, we have an inequality
$$
\length
\left(f_i^{X_1(k)} \boldsymbol{v}_1 + f_i^{X_2^b(k)}  F(f_i) \boldsymbol{v}_2\right)
\geq \length\left(f_i^{X_2(k)}  \boldsymbol{v}_2\right).
$$
First, we consider the case $a_1^{k}=a_1^{k-1}+1$ and $a_2^{k}=a_2^{k-1}$
for the fixed $k$.
In this case,
the submodule $l_{i,k-1}$ has the following descriptions:
$$
l_{i,k-1} =
\left\langle 
f_i \cdot 
\left (f_i^{X_1(k)} \boldsymbol{v}_1 + f_i^{X_2^b(k)} F(f_i) \boldsymbol{v}_2 \right), \ 
f_i^{X_2(k)}  \boldsymbol{v}_2
\right\rangle .
$$
If $b_k > b_{k-1}$, then $b_k = b_{k-1}+1$ and $b_{k-1}=a_1^{k-1}$.
So $X_2^b(k) = X_2^b(k-1) =X_2(k-1)$.
Since $X_1(k)+1=X_{1}(k-1)$, $X_2(k) =X_2(k-1)$ and $X_2^b(k)+1 > X_2(k-1)$, 
there exists 
$$
\begin{aligned}
l_{i,k-1} &=
\left\langle 
f_i^{X_1(k)+1} \boldsymbol{v}_1 + f_i^{X_2^b(k)+1} F(f_i) \boldsymbol{v}_2, \ 
f_i^{X_2(k)}  \boldsymbol{v}_2
\right\rangle \\
&=
\left\langle 
f_i^{X_1(k-1)} \boldsymbol{v}_1 , \ 
f_i^{X_2(k-1)}  \boldsymbol{v}_2
\right\rangle .
\end{aligned}
$$
Here remark that $X_2^b(k-1) =X_2(k-1)$.
If $b_k = b_{k-1}$, then $X_2^b(k)+1  = X_2^b(k-1)$.
Since $X_2(k)=X_2(k-1)$, we have
$$
\begin{aligned}
l_{i,k-1} &=
\left\langle 
f_i^{X_1(k)+1} \boldsymbol{v}_1+ f_i^{X_2^b(k)+1} F(f_i) \boldsymbol{v}_2, \ 
f_i^{X_2(k)}  \boldsymbol{v}_2
\right\rangle \\
&=
\left\langle 
f_i^{X_1(k-1)} \boldsymbol{v}_1 + f_i^{X_2^b(k-1)} F(f_i) \boldsymbol{v}_2, \ 
f_i^{X_2(k-1)}  \boldsymbol{v}_2
\right\rangle .
\end{aligned}
$$
Second, we consider the case  $a_1^{k}=a_1^{k-1}-1$ and $a_2^{k}=a_2^{k-1}+1$
for the fixed $k$.
In this case,
the submodule $l_{i,k-1}$ has the following descriptions:
$$
l_{i,k-1} =
\left\langle 
f_i^{X_1(k)} \boldsymbol{v}_1 + f_i^{X_2^b(k)}  \tilde{F}(f_i) \boldsymbol{v}_2, \ 
f_i^{X_2(k)+1}  \boldsymbol{v}_2
\right\rangle,
$$
since $\boldsymbol{v}_1$
and $\boldsymbol{v}_2$ satisfy the condition of the genericness of $l_{i,\bullet}$.
Here remark that $X_2(k)\geq X^b_2(k)$.
Since $a_1^{k}=a_1^{k-1}-1$, $a_1^k \geq b_k$, and $b_{k-1} = \min \{b_k,a_1^{k-1}\}$, 
we have $b_k = b_{k-1}$. 
We may check that $X_1(k)  = X_1(k-1)$,
$X_2^b(k)  = X_2^b(k-1)$ and $X_2(k)=X_2(k-1)-1$.
So
$$
\begin{aligned}
l_{i,k-1} &=
\left\langle 
f_i^{X_1(k)} \boldsymbol{v}_1 + f_i^{X_2^b(k)}  \tilde{F}(f_i) \boldsymbol{v}_2, \ 
f_i^{X_2(k)+1}  \boldsymbol{v}_2
\right\rangle \\
&=
\left\langle 
f_i^{X_1(k-1)} \boldsymbol{v}_1 + f_i^{X_2^b(k-1)}  \tilde{F}(f_i) \boldsymbol{v}_2, \ 
f_i^{X_2(k-1)}  \boldsymbol{v}_2
\right\rangle.
\end{aligned}
$$
Finally, we have 
the claim that we 
may take
 $F(f_i) \in \O_{n_i[t_i]}$ (where $F(0)\neq0$)
so that \eqref{eq:2021.9.8.17.54}.

We take $\tilde{\boldsymbol{v}}_1, \tilde{\boldsymbol{v}}_2 \in E \otimes \O_{\mathbb{P}^1,t_{i}}$
such that
$\tilde{\boldsymbol{v}}_1|_{n_i[t_i]} =\boldsymbol{v}_1 $
and $\tilde{\boldsymbol{v}}_2|_{n_i[t_i]} =\boldsymbol{v}_2$.
We define $\tilde{\boldsymbol{v}}'_1,
\tilde{\boldsymbol{v}}'_2,
\tilde{\boldsymbol{v}}''_1,
\tilde{\boldsymbol{v}}''_2 \in E\otimes \O_{\mathbb{P}^1,t_i}$ as  
$$
\begin{aligned}
\tilde{\boldsymbol{v}}'_1 =
f_i^{a_1^{k}+2 a_2^{k}} (f_i^{X_1(k)} \tilde{\boldsymbol{v}}_1 + 
f_i^{X_2^b(k)}  F(f_i) \tilde{\boldsymbol{v}}_2), & \quad&
\tilde{\boldsymbol{v}}'_2 =f_i^{a_1^{k}+2 a_2^{k}}f_i^{X_2(k)}  \tilde{\boldsymbol{v}}_2, \\
\tilde{\boldsymbol{v}}''_1 =
f_i^{a_1^{n_{i}}+2 a_2^{n_{i}}} \tilde{\boldsymbol{v}}_1, & \quad&
\tilde{\boldsymbol{v}}''_2 =f_i^{a_1^{n_{i}}+2 a_2^{n_{i}}} \tilde{\boldsymbol{v}}_2  ,
\end{aligned}
$$
respectively. 
Then we have 
$(E_i')^{(k)} \otimes \O_{\mathbb{P}^1,t_i}\cong
\langle  \tilde{\boldsymbol{v}}'_1  \rangle
\oplus 
\langle \tilde{\boldsymbol{v}}'_2  \rangle$ 
and 
$(E_i')^{(n_i)} \otimes \O_{\mathbb{P}^1,t_i} \cong
 \langle \tilde{\boldsymbol{v}}''_1 \rangle
\oplus 
\langle \tilde{\boldsymbol{v}}''_2 \rangle$.
Now we compute the quotient $l_{i,k}'=(E_i')^{(k)}/(E_i')^{(n_i)}$.
We may check the following equalities
$$
\tilde{\boldsymbol{v}}'_1 = f_i^{-X_2(k)} \tilde{\boldsymbol{v}}''_1+
f_i^{-X_1^b(k)} F(f_i) \tilde{\boldsymbol{v}}''_2, \qquad 
\tilde{\boldsymbol{v}}'_2
= f_i^{-X_1(k)}   \tilde{\boldsymbol{v}}_2'' ,
$$
and
$$
f_i^{X_1^b(k)-X_1(k)}\tilde{\boldsymbol{v}}'_1 - F(f_i) \tilde{\boldsymbol{v}}'_2
=f_i^{-X^b_2(k)}\tilde{\boldsymbol{v}}_1''.
$$
We set $\tilde{\boldsymbol{v}}'_3:=f_i^{X_1^b(k)-X_1(k)}\tilde{\boldsymbol{v}}'_1 - F(f_i) \tilde{\boldsymbol{v}}'_2$,
which is an element of $l_{i,k}'$. Here remark that $X_1^b(k) - X_1(k)=a_1^k -b_k\geq 0$.
We have 
$$
f_i^{X_1^b(k)} \tilde{\boldsymbol{v}}'_1 = 
f_i^{X_1(k)} \tilde{\boldsymbol{v}}'_2=
f_i^{X_2^b(k)} \tilde{\boldsymbol{v}}'_3= 0
$$
in $l_{i,k}'$.
Here remark that $X_1^b(k) - X_2(k)=a_1^{n_i} -b_k\geq 0$.
So, when $X^b_2(k)-X^b_1(k)\geq 0$, the type of $l_{i,k}'$ is $(1^{X_2^b(k)-X_1^b(k)},2^{X_1^b(k)})$.
When $X^b_2(k)-X^b_1(k)<0$
the type of $l_{i,k}'$ is
$(1^{X_1^b(k)-X_2^b(k)},2^{X_2^b(k)})$.
\end{proof}

If $a_2^{n_i}=0$ or $a_2^{n_i}=1$, then we have a simple description of the standard tableau of 
the transformation $l'_{i,\bullet}$.
If $a_2^{n_i}=0$, the standard tableau of $l_{i,\bullet}$ is unique. 
We denote by $T_0$ this standard tableau.
Let $T'$ be the standard tableau of the transformation $l'_{i,\bullet}$ of $l_{i,\bullet}$.
Then we have $T'= T_0$ by the proposition above. 
That is, if $(E, \{ l_{i,\bullet} \}_{ i \in I} )$
is an ordinary parabolic bundle, then
$\elm^-_{i_0}(E, \{ l_{i,\bullet} \}_{ i \in I} )$
is also an ordinary parabolic bundle.
When $a_2^{n_i}=1$, 
description of the standard tableau of 
the transformation $l'_{i,\bullet}$ is as follows:

\begin{cor}\label{cor:2021.10.22.14.33}
We assume that $\lambda^{(n)}=(1^{n-2},2^1)$.
Let $T_{k}$ be the standard tableau defined in \eqref{eq:2021.8.15.18.46}
for $k=1,2,\ldots,n-1$.
Let $l_{i,\bullet}$ be a generic refined parabolic structure with 
the standard tableau $T_k$.
Let $l'_{i,\bullet}$ be the transformation \eqref{eq:ElmOfRefinedStr} of $l_{i,\bullet}$
and $T'$ be the standard tableau of $l'_{i,\bullet}$.
Then $T'= T_{n-k}.$
\end{cor}

\begin{proof}
By direct computation, we may check that
$$
a_1^{k'} =
\begin{cases}
k'-2 & \text{when $n-k+1 \leq k'\leq n$} \\
k' & \text{when $1 \leq k'\leq n-k$}
\end{cases}
\quad \text{and} \quad
a_2^{k'} =
\begin{cases}
1 & \text{when $n-k+1 \leq k'\leq n$} \\
0 & \text{when $1 \leq k'\leq n-k$}
\end{cases}
$$
and
$$
b_{k'} =
\begin{cases}
a_1^{k'} & \text{when $k'\neq n-k$} \\
a_1^{k'} -1 & \text{when $k'= n-k$}.
\end{cases}
$$
Then we have the following equalities
$$
X_1^b(k') =
\begin{cases}
n-k' & \text{when $n-k \leq k'\leq n-1$} \\
n-k'-1 & \text{when $1 \leq k'\leq n-k-1$}
\end{cases}
\quad \text{and} \quad
X_2^b(k') =
\begin{cases}
0 & \text{when $n-k \leq k'\leq n-1$} \\
1 & \text{when $1 \leq k'\leq n-k-1$}.
\end{cases}
$$
By Proposition \ref{prop:2021.9.9.9.53},
we have the claim.
\end{proof}

Now,
we discuss a relation between 
the elementary transformations and 
the stability condition.

\begin{prop}\label{prop:2021.10.21.15.22}
We take $i_0 \in I$.
Let $(E'_{i_0}, \bold{l}'_{i_0} )=\elm^-_{n_{i_0}[t_{i_0}]}(E, \{ l_{i,\bullet} \}_{ i \in I} )$.
Let $L$ be a line subbundle of $E$ 
and $L'$ be a induced line subbundle of $E'_{i_0}$ via 
$E'_{i_0} \subset  E$.
We set 
$$
w'_{i, k} = 
\begin{cases}
1- w_{i,n_{i_0} -k+1} & \text{for $i=i_0$ and $k=1,\ldots,n_{i_0}$} \\
w_{i,k} & \text{for $i\neq i_0$ and $k=1,\ldots,n_{i}$}.
\end{cases}
$$
Let $\mathrm{Stab}_{\boldsymbol{w}'} (L')$ be the stability index of $L'$
with respect to $(E'_{i_0}, \bold{l}'_{i_0} )$ and the weights $\boldsymbol{w}'=\{ w_{i,k}'\}$.
Then
$$
\mathrm{Stab}_{\boldsymbol{w}} (L)
=\mathrm{Stab}_{\boldsymbol{w}'} (L').
$$
\end{prop}

\begin{proof}
Since we have a short exact sequence 
$$
0 \longrightarrow E^{(n_{i_0})} 
\longrightarrow E
\longrightarrow E/E^{(n_{i_0})}
\longrightarrow 0
$$
and $\mathrm{length} (E/E^{(n_{i_0})}) =n_{i_0}$, we have
$\deg(E) = \deg (E_{i_0}')  + n_{i_0}$.
We have a short exact sequence 
$$
0 \longrightarrow L' 
\longrightarrow L
\longrightarrow E|_{L}/E^{(n_{i_0})} |_{L}
\longrightarrow 0.
$$
Since
$$
\begin{aligned}
\mathrm{length} (E|_{L}/E^{(n_{i_0})} |_{L})
&= \mathrm{length} (E|_{L}/E(-n_{i_0}[t_{i_0}])|_{L}) -
  \sum_{k=1}^{n_{i_0}} 
\mathrm{length} (E^{(k)}|_{L}/E^{(k-1)}|_{L}) \\
&=n_{i_0} -   \sum_{k=1}^{n_{i_0}} 
\mathrm{length} (E^{(k)}|_{L}/E^{(k-1)}|_{L}),
\end{aligned}
$$
we have an equality:
$$
\deg(L) = \deg (L')  + n_{i_0}-   \sum_{k=1}^{n_{i_0}} 
\mathrm{length} (E^{(k)}|_{L}/E^{(k-1)}|_{L}).
$$
Moreover, we have the following equalities
$$
\begin{aligned}
&\mathrm{length} (E^{(k)}|_{L}/E^{(k-1)}|_{L})+
\mathrm{length} ((E_{i_0}')^{(n_{i_0}-k+1)}|_{L'}/(E_{i_0}')^{(n_{i_0}-k)} |_{L'}) \\
&=\mathrm{length} (E^{(k)}|_{L}/E^{(k-1)}|_{L})+
\mathrm{length} (E^{(k-1)}(-k [t_{i_0}])|_{L}/E^{(k)}(-(k+1) [t_{i_0}]) |_{L}) \\
& =
\mathrm{length} (E^{(k)}|_{L}/E^{(k)}(-[0])|_{L})=1.
\end{aligned}
$$
Then we obtain that 
$$
\begin{aligned}
\mathrm{Stab}_{\boldsymbol{w}} (L) &=
\deg (E) - 2 \deg(L) 
+  \sum_{k=1}^{n_{i_0}}  w_{i_0,k} (1 -2  \cdot 
\mathrm{length} (E^{(k)}|_{L}/E^{(k-1)}|_{L}))
) +\sum_{i \in I \setminus\{i_0\}} \sum_k \epsilon_{i,k}(L) w_{i,k} \\
&=\deg (E_{i_0}') +n_{i_0} - 2 \left(\deg(L')+ n_{i_0}-  \sum_{k=1}^{n_{i_0}} 
\mathrm{length} (E^{(k)}|_{L}/E^{(k-1)}|_{L}))   \right)\\
&\qquad +  \sum_{k=1}^{n_{i_0}}  w_{i_0,k}  \left(2  \cdot 
\mathrm{length} ((E_{i_0}')^{(n_{i_0}-k+1)}|_{L'}/(E_{i_0}')^{(n_{i_0}-k)} |_{L'}) -1 \right) 
+\sum_{i \in I \setminus\{i_0\}} \sum_k \epsilon_{i,k}(L) w_{i,k} \\
&=\deg (E_{i_0}') - 2 \deg(L') +  \sum_{k'=1}^{n_{i_0}} 
(1- 2  \cdot 
\mathrm{length} ((E_{i_0}')^{(k')}|_{L'}/(E_{i_0}')^{(k'-1)} |_{L'}) ) \\
&\qquad +  \sum_{k'=1}^{n_{i_0}} ( -  w_{i_0,n_{i_0}-k'+1}) (1- 2  \cdot 
\mathrm{length} ((E_{i_0}')^{(k')}|_{L'}/(E_{i_0}')^{(k'-1)} |_{L'}) ) 
+\sum_{i \in I \setminus\{i_0\}} \sum_k \epsilon_{i,k}(L) w_{i,k} \\
&=\mathrm{Stab}_{\boldsymbol{w}'} (L').
\end{aligned}
$$

\end{proof}

\begin{cor}
If $(E, \{ l_{i,\bullet} \}_{ i \in I} )$ is tame and undecomposable, 
then 
$\elm^-_{n_{i_0}[t_{i_0}]}(E, \{ l_{i,\bullet} \}_{ i \in I} )$
is also tame and undecomposable.
\end{cor}

\begin{proof}
Since $(E, \{ l_{i,\bullet} \}_{ i \in I} )$ is tame and undecomposable,
$(E, \{ l_{i,\bullet} \}_{ i \in I} )$ is stable for some weights $\boldsymbol{w}$
(Theorem \ref{thm:2021.10.27.13.44}).
By Proposition \ref{prop:2021.10.21.15.22}, 
the elementary transformation
$\elm^-_{n_{i_0}[t_{i_0}]}(E, \{ l_{i,\bullet} \}_{ i \in I} )$ 
stable for some weights $\boldsymbol{w}'$.
By Theorem \ref{thm:2021.10.27.13.44}, 
$\elm^-_{n_{i_0}[t_{i_0}]}(E, \{ l_{i,\bullet} \}_{ i \in I} )$ 
is also tame and undecomposable.
\end{proof}

\section{Geometry of moduli spaces and weak del Pezzo surfaces}

Let
$\overline{\BUN}(D,1)$
be the moduli space of tame and undecomposable refined parabolic bundles on $(\mathbb{P}^1,D)$ whose degree are one. 
To construct a good moduli space, we need the stability condition with 
the $\boldsymbol{w}$-stability index (or, equivalently, with 
the $\boldsymbol{\alpha}$-parabolic degree). 
Then for weights $\boldsymbol{w}$, we may construct  
the coarse moduli space $\overline{\BUN}_{\boldsymbol{w}} (D,1)$
 of $\boldsymbol{w}$-semistable refined parabolic bundles on
$(\mathbb{P}^1,D)$ whose degrees are one as a scheme.
Moreover, this coarse moduli scheme is projective \cite[Corollary 5.13]{Yoko1}.
Remark that the open subset of $\boldsymbol{w}$-stable parabolic bundles is smooth \cite{Yoko2}.
By Theorem \ref{2022_9_29_12_23}, 
for a tame and undecomposable refined parabolic bundle,
there exist weights $\boldsymbol{w}$ such that 
this refined parabolic bundle is $\boldsymbol{w}$-stable.
So $\overline{\BUN}(D,1)$ is covered by the projective schemes 
$\overline{\BUN}_{\boldsymbol{w}} (D,1)$ when $\boldsymbol{w}$ runs.

The purpose of this section is to give a proof of Theorem \ref{2022.2.2.23.11}.
That is, we will give detail description of the moduli spaces 
$\overline{\BUN}_{\boldsymbol{w}} (D,1)$ with democratic weights for $n=5$. 
If we take generic weights $\boldsymbol{w}$, 
then $\boldsymbol{w}$-semistable = $\boldsymbol{w}$-stable. 
So $\overline{\BUN}_{\boldsymbol{w}} (D,1)$ is a smooth projective scheme in this case.
We will consider $\overline{\BUN}_{\boldsymbol{w}} (D,1)$ only for 
generic democratic weights $\boldsymbol{w}$.
If $n=5$, then the dimension of $\overline{\BUN}_{\boldsymbol{w}} (D,1)$ is two.
We will give description of these two-dimensional smooth projective schemes,
and see that these schemes are smooth projective surfaces.
We investigate the change of the smooth projective surfaces
when the democratic weights vary. 
Moreover we will investigate the automorphism groups of 
the smooth projective surfaces 
when all the democratic weights are $1/2$.
Remark that, in this case, 
the elementary transformations for refined parabolic bundles
give automorphisms on $\overline{\BUN}_{\frac{1}{2}} (D,1)$
by Proposition \ref{prop:2021.10.21.15.22}.

A free $\mathcal{O}_{n_i [t_i]}$-submodule $l_{ i }$ of $E|_{n_i [t_i]}$
induces a one dimensional subspace $l_i^{\mathrm{red}}$ of $E|_{t_i}$, that is the restriction
of $l_i$ to $t_i$ (without multiplicity).

\begin{definition}
Assume that $E=\O \oplus \O(1)$.
Let $l_{ i }$ be a free $\mathcal{O}_{n_i [t_i]}$-submodule of $E|_{n_i [t_i]}$.
We say $l_{ i }$ is {\rm in general position} if 
$l_i^{\mathrm{red}} \notin \O(1) \subset  \mathbb{P} (E) $ 
for each for $i\in I$.
\end{definition}

We set $\tilde I=\{ (i,k) \mid i \in I, \, k \in \{1,2,\ldots,n_i\} \} $.

\begin{prop}\label{prop:2021.10.21.14.37}
We set $w_{i,k}=w$ 
(for any $(i,k) \in \tilde I$)
with 
$\frac{1}{n} < w < \frac{1}{n-2}$.
A refined parabolic bundle $(E,\bold{l})$ of degree $1$ is $\boldsymbol{w}$-stable if,
and only if, $(E,\bold{l})$ satisfies the following conditions:
\begin{itemize}
\item $E\cong \O\oplus \O(1)$;
\item $(E,\bold{l})$ is an undecomposable parabolic bundle 
(in particular, $l_i$ is free for any $i \in I$);
\item  $\bold{l}$ is in general position.
\end{itemize}
\end{prop}

\begin{proof}
We assume that $E \cong \mathcal{O}(-d_0+1)\oplus \mathcal{O}(d_0)$ where $d_0 \geq 2$.
We have the following inequality:
$$
\mathrm{Stab}_{\boldsymbol{w}}(\mathcal{O}(d_0))\leq 1-2d_0  + n w < -3 + \frac{n}{n-2} <0.
$$
So this refined parabolic bundle is not stable.
We assume that $E=\mathcal{O}\oplus \mathcal{O}(1)$ and
there is a refined parabolic structure which intersects 
$\mathcal{O}(1)$.
Then 
$$
\mathrm{Stab}_{\boldsymbol{w}}(\mathcal{O}(1))\leq -1 +(n-1) w - w =-1 + (n-2) w <0.
$$
So this refined parabolic bundle is not stable.
This implies that $\bold{l}$ is a (ordinary) parabolic structure and 
$\bold{l}$ is in general position.
We assume that $E=\mathcal{O}\oplus \mathcal{O}(1)$,
 the parabolic bundle is decomposable,
and there is no parabolic structure which intersects 
$\mathcal{O}(1)$.
Then there is a subbundle $\mathcal{O} \subset E$ such that 
any parabolic structures are contained in $\mathcal{O}$.
Then 
$$
\mathrm{Stab}_{\boldsymbol{w}}(\mathcal{O})\leq 1 - n w<0.
$$
Finally we have the following claim:
If $(E,\bold{l})$ is $\boldsymbol{w}$-stable, then
$E \cong \mathcal{O}\oplus \mathcal{O}(1)$, 
$(E,\bold{l})$ is an undecomposable parabolic bundle, and
$\bold{l}$ is in general position.

We assume that $(\mathcal{O}\oplus \mathcal{O}(1),\bold{l})$ is undecomposable 
and $\bold{l}$ is in generic position.
Since there is no refined parabolic structure which intersects 
$\mathcal{O}(1)$, we have 
$$
\mathrm{Stab}_{\boldsymbol{w}}(\mathcal{O}(1))= -1 + n w > 0 .
$$
Since $(\mathcal{O}\oplus \mathcal{O}(1),\boldsymbol{l})$ is undecomposable,
$$
\mathrm{Stab}_{\boldsymbol{w}}(\mathcal{O})\geq  1 + w (n-1) -w =1+w(n-2) >0 .
$$
For any subbundle $\mathcal{O}(d_0)$ of $\mathcal{O}\oplus \mathcal{O}(1)$ where 
$d_0 \leqq -1$,
we have 
$$
\mathrm{Stab}_{\boldsymbol{w}}(\mathcal{O}(d_0))
> 1 - 2d_0 -  n w \geq 3- n w > 3-  \frac{n}{n-2}>0 .
$$
Then $(\mathcal{O}\oplus \mathcal{O}(1),\bold{l})$ is $\boldsymbol{w}$-stable.
\end{proof}

\begin{prop}\label{prop:2021.7.28.14.36}
We set $w_{i,k}=w$ 
(for any $(i,k) \in \tilde I$)
with 
$\frac{1}{n} < w < \frac{1}{n-2}$.
We have an isomorphism 
$\overline{\BUN}_{\boldsymbol{w}} (D,1)\cong \mathbb{P}^{n-3}$.
\end{prop}

\begin{proof}
By Proposition \ref{prop:2021.10.21.14.37}, 
a $\boldsymbol{w}$-stable refined parabolic bundle $(E,\bold{l})$ of degree $1$ is 
a non trivial extension
$$
0\to (\O(1),\emptyset) \to (E,\bold{l}) \to (\O,\bold{l})\to 0.
$$
The obstruction to split the extension is measured by an element of
$$
H^1(\mathbb{P}^1,\mathcal{H}om (\O(D) ,\O(1))) \cong 
H^1(\mathbb{P}^1,\O(1)\otimes \O(-D)) \cong 
H^0(\mathbb{P}^1,\O(-1)\otimes \Omega_{\mathbb{P}^1}^1(D))^*.
$$
Any two non trivial extensions define isomorphic parabolic bundles if, and only if,
the corresponding cocycles are proportional.
So the moduli space of extension is parametrized by 
$\mathbb{P}H^0(\mathbb{P}^1,\O(-1)\otimes \Omega_{\mathbb{P}^1}^1(D))^*
\cong \mathbb{P}^{n-3}$.
\end{proof}

\subsection{Moduli of stable refined parabolic bundles 
with democratic weights for $n=5$
(Proof of Theorem \ref{2022.2.2.23.11})}\label{section:2021.12.31.15.25}

We assume that $n=5$.
Let $D$ be one of the effective divisors in the following list: 
$$
\begin{aligned}
D_{2111}:=2[t_1] +[t_2]+[t_3] +[t_4], & &D_{221}:=2[t_1] +2[t_2]+[t_3], &&
D_{311}:=3[t_1] +[t_2]+[t_3],\\
D_{32}:=3[t_1] +2[t_2],&&
D_{41}:=4[t_1] +[t_2], && 
D_5:=5[t_1].
\end{aligned}
$$
Our purpose here is to describe moduli spaces of 
stable refined parabolic bundles with democratic weights 
$\boldsymbol{w}=(\boldsymbol{w}_1,\ldots, \boldsymbol{w}_{\nu})$ for $n=\deg(D) =5$.
Here the democratic weights means that 
$$
\boldsymbol{w}_i = (w_{i,n_i}, \ldots,w_{i,2},w_{i,1}) =(w,\ldots,w,w)
$$
for $0<w<1$.
If $\frac{1}{5}< w< \frac{1}{3}$, then 
this moduli space is $\mathbb{P}^2$ 
(see Proposition \ref{prop:2021.7.28.14.36}).
We would like to describe the moduli spaces for another democratic weights $w$.
If $w$ is moved and $w$ crosses walls, then 
some ``special bundles'' (which are stable for $\frac{1}{5}< w< \frac{1}{3}$) become unstable
and some ``special bundles'' (which are unstable for $\frac{1}{5}< w< \frac{1}{3}$)
become stable.
Here, the ``special bundles'' are defined in the appendix below.
To describe the moduli spaces, 
we consider such a wall-crossing.

\begin{definition}
We define {\rm special bundles} and the {\rm types of special bundles} as in the appendix below.
There are 6 types: type A, type B,$\ldots$, type F.
We say a tame and undecomposable refined parabolic bundle is {\rm generic}
if this refined parabolic bundle is not a special bundle.
\end{definition}

We may check that the following claims:
\begin{itemize}

\item If a tame and undecomposable refined parabolic bundle is generic, 
then this refined parabolic bundle is stable for any $w$ ($\frac{1}{5}<w<1$);

\item 
If $\frac{1}{5} < w<\frac{1}{3}$, 
then special bundles with type A, type B, and type C are stable.
On the other hand, special bundles with type D and type F are unstable;

\item 
If $\frac{1}{3} < w <\frac{3}{5}$, 
then special bundles with type C become unstable.
Instead of it, special bundles with type D become stable;

\item 
If $\frac{3}{5} < w <1$, 
then special bundles with type A become unstable.
Instead of it, special bundles with type F become stable;

\item
Special bundles with type E are unstable for any democratic weights $\boldsymbol{w}$.

\end{itemize}
These can be summarized as Table \ref{table:2021.10.21.23.12}.

\begin{table}[hbtp]
  \caption{Table of stability of special bundles}
  \label{table:2021.10.21.23.12}
\begin{center}
  \begin{tabular}{|c||c|c|c|c|c|c|}
  \hline
  & type A& type B&type C&type D&type E&type F\\
    \hline\hline
$0 < w <\frac{1}{5}$&  unstable  & unstable  & unstable  & unstable  & unstable  & unstable     \\
    \hline
$\frac{1}{5} < w <\frac{1}{3}$&  stable  & stable  & stable  & unstable  & unstable  & unstable     \\
    \hline
    $\frac{1}{3} < w <\frac{3}{5}$&    stable  & stable  & unstable  & stable  & unstable  & unstable     \\
    \hline
    $\frac{3}{5} < w <1$&     unstable  & stable  & unstable  & stable  & unstable  & stable    \\
    \hline
  \end{tabular}
\end{center}
\end{table}

When $\frac{1}{5} <w <\frac{1}{3}$,
$\overline{\BUN}_{\boldsymbol{w}} (D,1)\cong \mathbb{P}^{2}$.
So there are loci on $\mathbb{P}^2$ corresponding to 
special bundles with type A, type B, and type C.
The locus corresponding special bundles with type A is a conic on $\mathbb{P}^2$.
The loci corresponding special bundles with type B are lines on $\mathbb{P}^2$.
The loci corresponding special bundles with type C are intersections 
the conic and the lines.
We may define an embedding 
\begin{equation}\label{2022.3.16.10.23}
\Pi \colon \mathbb{P}^1 
\longrightarrow \overline{\BUN}_{\boldsymbol{w}} (D,1) \cong \mathbb{P}^2
\end{equation}
such that the image is the conic corresponding to special bundles with type A
and the images of the points $t_1,t_2,\ldots,t_{\nu}$ 
are points corresponding to special bundles with type C.
We may describe the configuration of the loci of special bundles of type A, type B, and type C
as follows.
\begin{itemize}
\item When $D=D_{2111}$,
the configuration of the loci of special bundles is as follows: $P_{[t_i]} = \Pi(t_i)$ ($i=1,2,3,4$), 
the line $L_{2[t_1]}$ is tangent to the conic $C$ at $P_{[t_1]}$, and
the line $L_{[t_i],[t_j]}$ is intersect to the conic $C$ at $P_{[t_i]}$ and $P_{[t_j]}$
for $i,j \in \{ 1,2,3,4\}$ with $i\neq j$.

\item When $D=D_{221}$,
the configuration of the loci of special bundles is as follows: $P_{[t_i]} = \Pi(t_i)$ ($i=1,2,3$), 
the line $L_{2[t_{i'}]}$ is tangent to the conic $C$ at $P_{[t_{i'}]}$ ($i'=1,2$), and
the line $L_{[t_i],[t_j]}$ is intersect to the conic $C$ at $P_{[t_i]}$ and $P_{[t_j]}$
for $i,j \in \{ 1,2,3\}$ with $i\neq j$.

\item When $D=D_{311}$,
the configuration of the loci of special bundles is as follows: $P_{[t_i]} = \Pi(t_i)$ ($i=1,2,3$), 
the line $L_{2[t_{1}]}$ is tangent to the conic $C$ at $P_{[t_{1}]}$, and
the line $L_{[t_i],[t_j]}$ is intersect to the conic $C$ at $P_{[t_i]}$ and $P_{[t_j]}$
for $i,j \in \{ 1,2,3\}$ with $i\neq j$.

\item When $D=D_{32}$,
the configuration of the loci of special bundles is as follows: $P_{[t_i]} = \Pi(t_i)$ ($i=1,2$), 
the line $L_{2[t_{i}]}$ is tangent to the conic $C$ at $P_{[t_{i}]}$ ($i=1,2$), and
the line $L_{[t_1],[t_2]}$ is intersect to the conic $C$ at $P_{[t_1]}$ and $P_{[t_2]}$.

\item When $D=D_{41}$,
the configuration of the loci of special bundles is as follows: $P_{[t_i]} = \Pi(t_i)$ ($i=1,2$), 
the line $L_{2[t_{1}]}$ is tangent to the conic $C$ at $P_{[t_{1}]}$, and
the line $L_{[t_1],[t_2]}$ is intersect to the conic $C$ at $P_{[t_1]}$ and $P_{[t_2]}$.

\item When $D=D_{5}$,
the configuration of the loci of special bundles is as follows: $P_{[t_1]} = \Pi(t_1)$, 
the line $L_{2[t_{1}]}$ is tangent to the conic $C$ at $P_{[t_{1}]}$.
\end{itemize}

Now,
we will describe the moduli space $\overline{\BUN}_{\boldsymbol{w}} (D,1)$
with $\frac{1}{3}<w<\frac{3}{5}$.
Let $\boldsymbol{w}'$ be democratic weights with $\frac{1}{5} < w' <\frac{1}{3}$.
We identify $\overline{\BUN}_{\boldsymbol{w}'} (D,1)$ with $\mathbb{P}^2$.
We consider the sequence of blowing ups
$$
X_5(D) \xrightarrow{\pi_5}
X_4(D) \xrightarrow{\pi_4}
X_3(D) \xrightarrow{\pi_3}
X_2(D) \xrightarrow{\pi_2}
X_1(D) \xrightarrow{\pi_1} \mathbb{P}^2 .
$$
of $\overline{\BUN}_{\boldsymbol{w}'} (D,1)$ at points which have the following position:
\begin{itemize}

\item $P_{[t_1]}^{(0)} \succ_C P_{[t_1]}$, $P_{[t_2]}$, $P_{[t_3]}$, and $P_{[t_4]}$
when $D=D_{2111}$;

\item $P_{[t_1]}^{(0)} \succ_C P_{[t_1]}$, $P_{[t_2]}^{(0)} \succ_C P_{[t_2]}$, and $P_{[t_3]}$
when $D=D_{221}$;

\item $P_{[t_1]}^{(0)} \succ_C P_{[t_1]}^{(1)} \succ_C P_{[t_1]}$, 
$P_{[t_2]}$, and $P_{[t_3]}$
when $D=D_{311}$;

\item $P_{[t_1]}^{(0)} \succ_C P_{[t_1]}^{(1)} \succ_C P_{[t_1]}$, 
and $P^{(0)}_{[t_2]} \succ_C P_{[t_2]}$
when $D=D_{32}$;

\item $P_{[t_1]}^{(0)} \succ_C P_{[t_1]}^{(1)} \succ_C P_{[t_1]}^{(2)} \succ_C P_{[t_1]}$, 
and $P_{[t_2]}$
when $D=D_{41}$;

\item $P_{[t_1]}^{(0)} \succ_C P_{[t_1]}^{(1)} \succ_C P_{[t_1]}^{(2)} \succ_C
P_{[t_1]}^{(3)} \succ_C P_{[t_1]}$
when $D=D_{5}$.

\end{itemize}
Here $P' \succ_C P$ means 
$P'$ is the intersection of the exceptional divisor of the blowing up at $P$ 
and the strict transformation of the conic $C$.
Here $P$ is a point on $C$ or a strict transformation of $C$.
We denote by ${\rm wdP}^{(n)}_{X}$ 
the weak del Pezzo surface of degree $n$ whose configuration of $(-2)$-curves is $X$.
We have that 
$$
\begin{aligned}
X_5(D_{2111})={\rm wdP}^{(4)}_{A_1},
 &&X_5(D_{221})={\rm wdP}^{(4)}_{2A_1},
 &&X_5(D_{311})={\rm wdP}^{(4)}_{A_2}, \\ 
X_5(D_{32})={\rm wdP}^{(4)}_{A_2+A_1}, 
&&X_5(D_{41})={\rm wdP}^{(4)}_{A_3},
&&X_5(D_{5})={\rm wdP}^{(4)}_{A_4},
\end{aligned}
$$
and
$$
\overline{\BUN}_{\boldsymbol{w}} (D,1)\cong X_5(D), \qquad 
\frac{1}{3}<w<\frac{3}{5}.
$$
Let $E^{(k)}_{[t_i]}$ ($k=0,1,2,\ldots, n_i-1$)
be the locus of special bundles with type D on
the moduli space $\overline{\BUN}_{\boldsymbol{w}} (D,1)$.
By using Corollary \ref{2022.1.15.10.27}, 
$E^{(k)}_{[t_i]}$ ($k=1,2,\ldots, n_i-1$) compose a chain of $(n_i-1)$ projective lines.
We may check that
$E^{(k)}_{[t_i]}$ is the exceptional divisor of the blowing up at $P^{(k)}_{[t_i]}$ 
($k=0,1,2,\ldots, n_i-1$).
Here we set $P^{(n_i-1)}_{[t_i]}=P_{[t_i]}$.

Now,
we will describe the moduli space $\overline{\BUN}_{\boldsymbol{w}} (D,1)$
with $\frac{3}{5}<w<1$.
We may check that the strict transformation of the conic $C$ on $X_5(D)$
is a $(-1)$-curve.
We contract this $(-1)$-curve on $X_5(D)$:
$$
\pi_5' \colon X_5(D) \longrightarrow X_5'(D).
$$
Then we have weak del Pezzo surfaces of degree $5$:
$$
\begin{aligned}
X_5'(D_{2111})={\rm wdP}^{(5)}_{A_1},
 &&X_5'(D_{221})={\rm wdP}^{(5)}_{2A_1},
 &&X_5'(D_{311})={\rm wdP}^{(5)}_{A_2}, \\ 
X_5'(D_{32})={\rm wdP}^{(5)}_{A_2+A_1}, 
&&X_5'(D_{41})={\rm wdP}^{(5)}_{A_3},
&&X_5'(D_{5})={\rm wdP}^{(5)}_{A_4}.
\end{aligned}
$$
We have that
$$
\overline{\BUN}_{\boldsymbol{w}} (D,1)\cong X'_5(D), \qquad 
\frac{3}{5}<w<1.
$$
Let $P_C$
be the locus of special bundles with type F on
the moduli space $\overline{\BUN}_{\boldsymbol{w}} (D,1)$.
We may check that $P_C = \pi_5' (C)$.
These can be summarized as Table \ref{table:2021.10.21.23.13}.

\begin{table}[hbtp]
  \caption{Table of moduli spaces $ \overline{\BUN}_{\boldsymbol{w}} (D,1) $}
  \label{table:2021.10.21.23.13}
\begin{center}
  \begin{tabular}{|c||c|c|c|c|c|c|}
  \hline
  & $D=D_{2111}$& $D=D_{221}$&$D=D_{311}$&$D=D_{32}$&$D=D_{41}$&$D=D_{5}$\\
    \hline\hline
$0 < w <\frac{1}{5}$&  $\emptyset$  & $\emptyset$  & $\emptyset$  & $\emptyset$  & $\emptyset$  & $\emptyset$     \\
    \hline
$\frac{1}{5} < w <\frac{1}{3}$&  $\mathbb{P}^2$  & $\mathbb{P}^2$ & $\mathbb{P}^2$  & $\mathbb{P}^2$  & $\mathbb{P}^2$
  & $\mathbb{P}^2$     \\
    \hline
    $\frac{1}{3} < w <\frac{3}{5}$&   ${\rm wdP}^{(4)}_{A_1}$  & ${\rm wdP}^{(4)}_{2A_1}$  
    & ${\rm wdP}^{(4)}_{A_2}$  & ${\rm wdP}^{(4)}_{A_1+A_2}$  & ${\rm wdP}^{(4)}_{A_3}$  & ${\rm wdP}^{(4)}_{A_4}$     \\
    \hline
    $\frac{3}{5} < w <1$&   ${\rm wdP}^{(5)}_{A_1}$  & ${\rm wdP}^{(5)}_{2A_1}$  
    & ${\rm wdP}^{(5)}_{A_2}$  & ${\rm wdP}^{(5)}_{A_1+A_2}$  & ${\rm wdP}^{(5)}_{A_3}$  & ${\rm wdP}^{(5)}_{A_4}$     \\
    \hline
  \end{tabular}
\end{center}
\end{table}

Now we check the assertion (ii) in Theorem B.
Let $\Pi_{-2} (D)$ be the effective divisors consisting of all $(-2)$-curves on
$\overline{\BUN}_{\boldsymbol{w}} (D,1)$ with $\frac{1}{3} < w <\frac{3}{5}$.
We have that
$$
\Pi_{-2} (D) = \sum_{\{ i \in I \mid n_i>1 \} } (E^{(1)}_{[t_i]} + E^{(2)}_{[t_i]} +\cdots + 
E^{(n_i-1)}_{[t_i]}).
$$
By the list in the appendix (Section \ref{section:2022.1.3.0.11}),
we may check the following claim:
If $(E,\{ l_{i,\bullet }\}_{i} )$ be a refined parabolic bundle
corresponding to a point on $\Pi_{-2}(D)$,
then there exists $i$ $(1\leq i\leq \nu)$ such that $l_{i,n_i}$ is not free,
that is, $l_{i,\bullet }$ is not a parabolic structure for this $i$.

Next we check the assertion (iii) in Theorem B.
Let $(E,\{ l_{i,\bullet }\}_{i} )$ be a refined parabolic bundle
corresponding 
to a point on $\overline{\mathfrak{Bun}}_{\boldsymbol{w}}(D,1)$ 
with $\frac{1}{5}<w <\frac{1}{3}$.
By Proposition \ref{prop:2021.10.21.14.37}, $(E,\{ l_{i,\bullet }\}_{i} )$ is an ordinary parabolic bundle.
By Lemma \ref{2022.2.2.22.32} (or Proposition \ref{prop:2022.2.23.20.13}), 
we have that 
$(E,\{ l_{i,n_i}\}_i )$ is a simple parabolic bundle.
Let $\Pi_{-2} (D)'$ be the effective divisor on $\overline{\BUN}_{\boldsymbol{w}} (D,1)$ with $\frac{3}{5} < w <1$
induced by $\Pi_{-2} (D)$.
Let $(E,\{ l_{i,\bullet }\}_{i} )$ be a refined parabolic bundle
corresponding 
to a point on
$\overline{\mathfrak{Bun}}_{\boldsymbol{w}}(D,1) \setminus \Pi_{-2}(D)$ 
with $\frac{1}{3}<w <\frac{3}{5}$, or 
$\overline{\mathfrak{Bun}}_{\boldsymbol{w}}(D,1)\setminus \Pi_{-2}'(D)$ 
with $\frac{3}{5}<w <1$.
By the argument of the effect on the moduli space of varying the weights and
the list in the appendix (Section \ref{section:2022.1.3.0.11}),
we may check that $(E,\{ l_{i,\bullet }\}_{i} )$ is an ordinary parabolic bundle.
Moreover, by Lemma \ref{2022.2.2.22.32} (or Proposition \ref{prop:2022.2.23.20.13}), 
we have that 
$(E,\{ l_{i,n_i}\}_i )$ is a simple parabolic bundle.
Finally, we obtain the statements in Theorem B.

\subsection{Automorphisms of the moduli spaces
for the democratic weights with $w=\frac{1}{2}$}\label{section:2021.12.31.15.42}

Now we introduce some automorphisms of the moduli space 
$\overline{\BUN}_{\boldsymbol{w}} (D,1)$
for the democratic weights with $w=\frac{1}{2}$.
By Proposition \ref{prop:2021.10.21.15.22}, we have the following claim:
If a refined parabolic bundle $(E,\bold{l})$ is stable for democratic weights 
for $w$ ($0<w<1$), then 
the elementary transformation of $(E,\bold{l})$ at $t_{i_0}$ is stable for 
the weights $\boldsymbol{w}'=(\boldsymbol{w}'_i)_{i\in I}$, where
$$
\boldsymbol{w}'_i = 
\begin{cases}
(1-w,\ldots,1-w,1-w) & \text{when $i= i_0$} \\
(w,\ldots,w,w) & \text{when $i\neq i_0$}.
\end{cases}
$$
So we have an isomorphism between moduli spaces:
$$
\elm^-_{n_{i_0}[t_{i_0}]} \colon 
\overline{\BUN}_{\boldsymbol{w}} (D,1) \longrightarrow \overline{\BUN}_{\boldsymbol{w}'} (D,1-n_{i_0}).
$$
Note that $\boldsymbol{w}'_i=\boldsymbol{w}_i$ when $w=\frac{1}{2}$.
If $w=\frac{1}{2}$ and $n_{i_0}$ is even, 
we set 
$$
\elm_{n_{i_0}[t_{i_0}]}:=
\O((n_{i_0}/2) \cdot t_{i_0})
\otimes \elm^-_{n_{i_0}[t_{i_0}]}.
$$
Then $\elm_{n_{i_0}[t_{i_0}]}$ induces an automorphism of
$\overline{\BUN}_{\frac{1}{2}} (D,1)$.
By Corollary \ref{cor:2021.10.22.14.33},
we have the claim that
the image of $(-2)$-curves under the automorphism 
$\O((n_{i_0}/2) \cdot t_{i_0})
\otimes \elm^-_{n_{i_0}[t_{i_0}]}$ 
(where $n_{i_0} \geq 2$ and $n_{i_0}$ is even)
is the following:
$$
\O((n_{i_0}/2) \cdot t_{i_0})
\otimes \elm^-_{n_{i_0}[t_{i_0}]} \colon E^{(k)}_{[t_{i_0}]} \longrightarrow  
E^{(n_{i_0}-k)}_{[t_{i_0}]}
$$
for $k=1,2,\ldots,n_{i_0}-1$.

We take $i_1 \in I \setminus \{ i_0\}$ 
so that $n_{i_0}+ n_{i_1}$ is even.
We set
$$
\elm_{n_{i_0}[t_{i_0}]+n_{i_1}[t_{i_1}]} 
:= 
\mathcal{O}(\frac{n_{i_0}+n_{i_1}}{2}) \otimes 
\left( \elm^-_{n_{i_0}[t_{i_0}]} \circ \elm^-_{n_{i_1}[t_{i_1}]} \right),
$$
which is an automorphism of  
$\overline{\BUN}_{\frac{1}{2}} (D,1)$.
Since the automorphisms 
 $\elm_{n_{i_0}[t_{i_0}]}$ and $\elm_{n_{i_1}[t_{i_1}]}$ 
are commutative, 
the automorphism $\elm_{n_{i_0}[t_{i_0}]+n_{i_1}[t_{i_1}]}$ 
is independent of the order of $n_{i_0}[t_{i_0}]$ and $n_{i_1}[t_{i_1}]$.

\subsection{Geometry of the weak del Pezzo surfaces from the modular point of view}\label{section:2021.12.31.15.25_1}
Now we impose that $n=5$ and $w_{i,1}=w_{i,2}=\cdots =w_{i,n_i}=\frac{1}{2}$ for any $i \in I$.
Then the moduli space $\overline{\mathfrak{Bun}}_{\frac{1}{2}}(D,1)$
is a weak del Pezzo surface ${\rm wdP}^{(4)}_X$.
Here $X$ means 
the configuration of $(-2)$-curves on the weak del Pezzo surface.
The purposes of this section is to recover the geometry of the weak del Pezzo surfaces
of degree 4 from the modular point of view.
First we will describe 
the graph of negative curves on 
the surface ${\rm wdP}^{(4)}_X$ for each $D$.
Second we will 
reconstruct all automorphisms of the surface ${\rm wdP}^{(4)}_X$
via elementary transformations of refined parabolic bundles
and automorphisms of the base curve with $D$.

Now we set $D=D_{2111}$.
The left-hand-side of Figure \ref{figure1}
is the projective plane $\mathbb P^2$ with the locus of special bundles.
Here the projective plane $\mathbb P^2$ is identified
with the moduli space $\overline{\BUN}_{\boldsymbol{w}} (D_{2111},1)$ with $\frac{1}{5}<w<\frac{1}{3}$.
The right-hand-side of Figure \ref{figure1} is
the graph of negative curves on
 the weak del Pezzo surface obtained by blowing-up of the left-hand-side.
This weak del Pezzo surface 
$\mathrm{wdP}^{(4)}_{A_1}$
is just the moduli space $\overline{\BUN}_{\frac{1}{2}} (D_{2111},1)$.
Remark that some lines are repeated
in order to see all intersections. 
Below, the curve $C$ is identified with the base curve of the parabolic bundle for simplicity
by \eqref{2022.3.16.10.23}.
Note that the one thick line is $(-2)$-curves.

\begin{figure}[h]
\begin{pspicture}(0,0)(16,6)


\psellipse[linecolor=green,linewidth=1pt](3.5,2.5)(1.9,1.3)
\psline[linecolor=brown,showpoints=false,linewidth=1pt](1.6,0)(1.6,5.4)
\psline[linecolor=orange,showpoints=false,linewidth=1pt](0,3.4)(6,0)
\psline[linecolor=orange,showpoints=false,linewidth=1pt](0,1.6)(6,5)
\psline[linecolor=orange,showpoints=false,linewidth=1pt](0,2.25)(6,3.15)
\psline[linecolor=cyan,showpoints=false,linewidth=1pt](1.2,5.2)(6,2.6)
\psline[linecolor=cyan,showpoints=false,linewidth=1pt](2.8,0)(6,4)
\psline[linecolor=cyan,showpoints=false,linewidth=1pt](3.8,0)(3.8,5)
\psline[linecolor=magenta,showpoints=false,linewidth=2pt]{->}(1.6,2.5)(1.6,3)
\psdot[linecolor=blue,linewidth=2pt](3.8,1.2)
\psdot[linecolor=blue,linewidth=2pt](3.8,3.8)
\psdot[linecolor=blue,linewidth=2pt](5.2,3)
\psdot[linecolor=red,linewidth=2pt](1.6,2.5)
\uput[l](1.7,0.8){\color{brown}$L_{2[t_1]}$}
\uput[r](-0.2,3.6){\color{orange}$L_{[t_1],[t_4]}$}
\uput[r](-0.2,2.6){\color{orange}$L_{[t_1],[t_3]}$}
\uput[r](-0.2,1.4){\color{orange}$L_{[t_1],[t_2]}$}
\uput[r](1.6,4.2){\color{cyan}$L_{[t_2],[t_3]}$}
\uput[r](3.7,4.8){\color{cyan}$L_{[t_2],[t_4]}$}
\uput[r](5.6,4.2){\color{cyan}$L_{[t_3],[t_4]}$}
\uput[l](1.7,3.1){\color{red}$P_{[t_1]}$}
\uput[r](4.1,3.8){\color{blue}$P_{[t_2]}$}
\uput[r](5.3,2.4){\color{blue}$P_{[t_3]}$}
\uput[r](4.1,1.1){\color{blue}$P_{[t_4]}$}
\uput[r](5,1.6){\color{green}$C$}

\psline[linecolor=red,showpoints=false,linewidth=2pt](7.6,0)(7.6,5)
\psline[linecolor=magenta,showpoints=false,linewidth=1pt](7.2,4)(12,4)
\psline[linecolor=orange,showpoints=false,linewidth=1pt](7.2,3)(12,3)
\psline[linecolor=orange,showpoints=false,linewidth=1pt](7.2,2)(12,2)
\psline[linecolor=orange,showpoints=false,linewidth=1pt](7.2,1)(12,1)

\psline[linecolor=brown,showpoints=false,linewidth=1pt](9,3.8)(12.8,5)
\psline[linecolor=brown,linestyle=dashed,showpoints=false,linewidth=0.5pt](12.8,5)(14,5.4)

\psline[linecolor=green,showpoints=false,linewidth=1pt](11,4.2)(15,2.6)

\psline[linecolor=blue,showpoints=false,linewidth=1pt](9,2.8)(15,4.7)
\psline[linecolor=cyan,showpoints=false,linewidth=1pt](11,3.2)(15,1.6)
\psline[linecolor=blue,showpoints=false,linewidth=1pt](9,1.8)(15,3.7)
\psline[linecolor=cyan,showpoints=false,linewidth=1pt](11,2.2)(14,1)
\psline[linecolor=cyan,linestyle=dashed,showpoints=false,linewidth=0.5pt](14,1)(15,0.6)

\psline[linecolor=blue,showpoints=false,linewidth=1pt](9,0.8)(15,2.7)
\psline[linecolor=cyan,showpoints=false,linewidth=1pt](11,1.2)(11.8,0.9)
\psline[linecolor=cyan,linestyle=dashed,showpoints=false,linewidth=0.5pt](11.8,0.9)(13,0.4)

\psline[linecolor=brown,linestyle=dashed,showpoints=false,linewidth=0.5pt](11,0.4)(12.8,1)
\psline[linecolor=brown,showpoints=false,linewidth=1pt](12.8,1)(15,1.7)

\psline[linecolor=cyan,linestyle=dashed,showpoints=false,linewidth=0.5pt](11,5.2)(11.8,4.9)
\psline[linecolor=cyan,showpoints=false,linewidth=1pt](11.8,4.9)(15,3.6)

\psline[linecolor=cyan,linestyle=dashed,showpoints=false,linewidth=0.5pt](13,5.4)(14,5)
\psline[linecolor=cyan,showpoints=false,linewidth=1pt](14,5)(15,4.6)

\uput[r](7.6,0.3){\red$E_{[t_1]}^{(1)}$}
\uput[r](7.6,4.3){\color{magenta}$E_{[t_1]}^{(0)}$}
\uput[r](7.6,3.2){\color{orange}$L_{[t_1],[t_2]}$}
\uput[r](7.6,2.2){\color{orange}$L_{[t_1],[t_3]}$}
\uput[r](7.6,1.2){\color{orange}$L_{[t_1],[t_4]}$}
\uput[r](9.6,4.46){\color{brown}$L_{2[t_1]}$}
\uput[r](9.6,3.46){\color{blue}$L_{[t_2]}$}
\uput[r](9.6,2.46){\color{blue}$L_{[t_3]}$}
\uput[r](9.6,1.46){\color{blue}$L_{[t_4]}$}
\uput[r](14.2,3){\color{green}$C$}
\uput[r](14.2,2){\color{cyan}$L_{[t_3],[t_4]}$}
\uput[r](14.2,1){\color{cyan}$L_{[t_2],[t_4]}$}
\uput[r](12.8,0.6){\color{cyan}$L_{[t_2],[t_3]}$}
\uput[r](14.2,5){\color{cyan}$L_{[t_2],[t_4]}$}
\uput[r](14.2,4){\color{cyan}$L_{[t_2],[t_3]}$}

\end{pspicture}

\caption{$\mathbb{P}^2$ and $\mathrm{wdP}^{(4)}_{A_1}$}
\label{figure1}
\end{figure}

In the previous section, we had seen that
some elementary transformations of refined parabolic bundles induce
automorphisms 
on the moduli space $\overline{\BUN}_{\boldsymbol{w}} (D_{2111},1)\cong \mathrm{wdP}^{(4)}_{A_1}$.
Now we reconstruct all automorphisms of
the weak del Pezzo surface ${\rm wdP}^{(4)}_{A_1}$
via these elementary transformations.
Let $\mathrm{Aut}(C,D_{2111})$ be the automorphism of $C$ sending $D_{2111}$ to itself.
Remark that any automorphisms in $\mathrm{Aut}(C,D_{2111})$
fix the point $t_1$.
If the points $t_1,t_2,t_3$, and $t_4$ have a generic cross-ratio,
then the non trivial automorphisms in $\mathrm{Aut} (C,[t_1]+[t_2]+[t_3]+[t_4])$
act on the 4 points by double-transposition:
That is,
$$
(t_1\leftrightarrow t_2,t_3\leftrightarrow t_4 ), \quad
(t_1\leftrightarrow t_3,t_2\leftrightarrow t_4 ), \quad \text{and} \quad
(t_1\leftrightarrow t_4,t_2\leftrightarrow t_3 ).
$$
So $\mathrm{Aut}(C,D_{2111})= \langle \mathrm{id} \rangle$.

\begin{prop}If 
$\mathrm{Aut}(C,D_{2111})= \langle \mathrm{id} \rangle$,
then the automorphism group 
of the weak del Pezzo surface 
is $\mathbb Z/2\times\mathbb Z/2\times\mathbb Z/2$, generated by elementary transformations:
$$\mathrm{Aut}\left(\mathrm{wdP}^{(4)}_{A_1}\right)=\langle \elm_{[t_2]+[t_3]},
\elm_{[t_3]+[t_4]},\elm_{2[t_1]}  \rangle.$$
\end{prop}

\begin{proof} These three elementary transformations act as follows on the negative curves:
\begin{itemize}
\item all of them stabilize the unique $(-2)$-curve $E^{(1)}_{[t_1]}$ (and so does any automorphism),
and therefore permute 
the $4$ intersecting $(-1)$-curves $E^{(0)}_{[t_1]}$, $L_{[t_1],[t_2]}$, $L_{[t_1],[t_3]}$, $L_{[t_1],[t_4]}$;
\item the action on these $4$ curves is by double-transposition, since the crossratio of their interection point
with the $(-2)$-curve $E^{(1)}_{[t_1]}$ should be preserved and is generic 
(coinciding with the cross-ratio of $(C,[t_1]+[t_2]+[t_3]+[t_4])$);
\item $\elm_{[t_2]+[t_3]}$ permutes $L_{[t_1],[t_2]}$ with $L_{[t_1],[t_3]}$;
\item $\elm_{[t_3]+[t_4]}$ permutes $L_{[t_1],[t_3]}$ with $L_{[t_1],[t_4]}$;
\item $\elm_{2[t_1]}$ stabilizes all $4$ curves $E^{(0)}_{[t_1]}$, $L_{[t_1],[t_2]}$, $L_{[t_1],[t_3]}$, $L_{[t_1],[t_4]}$,
but permutes $C$ and $L_{2[t_1]}$.
\end{itemize}
One easily check these properties by studying the action of elementary transformations on parabolic structure
and special line bundles. One also deduce that the $3$ elementary transformations (which commute)
indeed generate a order $8$ group of automorphisms isomorphic to $\mathbb Z/2\times\mathbb Z/2\times\mathbb Z/2$.

It remains to check that any automorphism $\phi$ belongs to that $8$ order group. 
One can assume, after composition by an element of the $4$ order group
$\langle \elm_{[t_2]+[t_3]},\elm_{[t_3]+[t_4]}\rangle$, that $\phi$ stabilizes
all $4$ curves $E^{(0)}_{[t_1]}$, $L_{[t_1],[t_2]}$, $L_{[t_1],[t_3]}$, $L_{[t_1],[t_4]}$.
Since $E^{(0)}_{[t_1]}$ intersects only $3$ negative curves, namely $C$ and $L_{2[t_1]}$ 
and the $(-2)$-curve $E^{(1)}_{[t_1]}$,
it has to permute 
the two first ones.
By composing with $\elm_{2[t_1]}$,
we can assume that $\phi$ also stabilizes each of the two curves $C$ and $L_{2[t_1]}$. 
Then, the intersection combinatorics shows that $\phi$ eventually stabilizes all negative curves.
Therefore, we can blow-down back to $\mathbb P^2$ and get a linear automorphism $\bar\phi$ 
which fixes the $4$ points supporting the divisor on the conic. 
There exists an one-to-one correspondence between
linear automorphisms on $\mathbb{P}^2$ which preserve the conic 
and automorphisms on the conic.
Therefore $\bar\phi$ is the identity,
$\phi$ as well, obviously belonging to the $8$ order group.
\end{proof}

If $D$ has special configuration, 
then we might have non trivial automorphism $\varphi\in\mathrm{Aut}(C,D_{2111})$.
For the non trivial automorphism $\varphi$, we can
pull-back a refined parabolic bundle by $\varphi$.
Since we impose that the weights are democratic,
the stability condition of this pull-back is the same as the  
stability condition of the refined parabolic bundle.
So, in that case, the automorphism $\varphi$ induces an extra automorphism for the weak del Pezzo surface.


Now we set $D=D_{221}$.
The left-hand-side of Figure \ref{figure2}
is the projective plane $\mathbb P^2$ with the locus of special bundles.
The right-hand-side of Figure \ref{figure2} is
the graph of negative curves on
 the weak del Pezzo surface obtained by blowing-up of the left-hand-side.
This weak del Pezzo surface $\mathrm{wdP}^{(4)}_{2A_1}$
 is just the moduli space $\overline{\BUN}_{\frac{1}{2}} (D_{221},1)$.
Note that $C$ and $L_{[t_3]}$ intersect each other,
although we don't see on the picture.

\begin{figure}[h]
\begin{pspicture}(0,0)(15,5)


\psellipse[linecolor=green,linewidth=1pt](3.37,1.87)(1.87,1.25)
\psline[linecolor=brown,showpoints=false,linewidth=1pt](0,0)(3.5,5)
\psline[linecolor=brown,showpoints=false,linewidth=1pt](6.75,0)(3.25,5)
\psline[linecolor=cyan,showpoints=false,linewidth=1pt](0,2.31)(6.75,2.31)
\psline[linecolor=orange,showpoints=false,linewidth=1pt](0.25,3.75)(4,0)%
\psline[linecolor=orange,showpoints=false,linewidth=1pt](6.5,3.75)(2.75,0)
\psline[linecolor=magenta,showpoints=false,linewidth=2pt]{->}(1.65,2.31)(2,2.87)
\psline[linecolor=magenta,showpoints=false,linewidth=2pt]{->}(5.1,2.31)(4.71,2.87)
\psdot[linecolor=blue,linewidth=2pt](3.37,0.62)
\psdot[linecolor=red,linewidth=2pt](1.65,2.31)
\psdot[linecolor=red,linewidth=2pt](5.06,2.31)
\uput[u](1.62,2.62){\red$P_{[t_1]}$}
\uput[u](5.12,2.62){\red$P_{[t_2]}$}
\uput[d](3.37,0.44){\blue$P_{[t_3]}$}
\uput[u](3.37,3.12){\green$C$}
\uput[u](3.5,2.3){\cyan$L_{[t_1],[t_2]}$}
\uput[u](1,3.5){\color{orange}$L_{[t_1],[t_3]}$}%
\uput[u](5.75,3.5){\color{orange}$L_{[t_2],[t_3]}$}
\uput[u](0.56,1.25){\color{brown} $L_{2[t_1]}$}
\uput[u](6.19,1.25){\color{brown} $L_{2[t_2]}$}



\psline[linecolor=red,showpoints=false,linewidth=2pt](8,2)(10.5,5)
\psline[linecolor=red,showpoints=false,linewidth=2pt](15,2)(12.5,5)
\psline[linecolor=cyan,showpoints=false,linewidth=1pt](9.5,4.5)(13.5,4.5)
\psline[linecolor=magenta,showpoints=false,linewidth=1pt](9,4)(10.5,1.5)
\psline[linecolor=magenta,showpoints=false,linewidth=1pt](14,4)(12.5,1.5)
\psline[linecolor=orange,showpoints=false,linewidth=1pt](8,3)(10,0)
\psline[linecolor=orange,showpoints=false,linewidth=1pt](15,3)(13,0)
\psline[linecolor=blue,showpoints=false,linewidth=1pt](9,0.5)(14,0.5)
\psline[linecolor=green,showpoints=false,linewidth=1pt](9.5,2.75)(13.5,2.75)
\psline[linecolor=brown,showpoints=false,linewidth=1pt](10,2)(14,0.75)
\psline[linecolor=brown,showpoints=false,linewidth=1pt](13,2)(9,0.75)
\uput[r](9.75,4){\red$E_{[t_1]}^{(1)}$}
\uput[l](13.25,4){\red$E_{[t_2]}^{(1)}$}
\uput[d](11.5,0.5){\blue$L_{[t_3]}$}
\uput[u](11.5,2.75){\green$C$}
\uput[d](11.5,5.2){\cyan$L_{[t_1],[t_2]}$}
\uput[u](8.4,1){\color{orange}$L_{[t_1],[t_3]}$}
\uput[u](14.6,1){\color{orange}$L_{[t_2],[t_3]}$}
\uput[d](10.5,1.2){\color{brown} $L_{2[t_1]}$}
\uput[d](12.5,1.2){\color{brown} $L_{2[t_2]}$}
\uput[ur](9.5,2.8){\color{magenta} $E_{[t_1]}^{(0)}$}
\uput[ur](12.5,2.8){\color{magenta} $E_{[t_2]}^{(0)}$}
\end{pspicture}

\caption{$\mathbb{P}^2$ and $\mathrm{wdP}^{(4)}_{2A_1}$}
\label{figure2}
\end{figure}

We take the automorphism $\varphi \in \mathrm{Aut}(C,D_{221})$ 
permuting $t_1$ and $t_2$, fixing $t_3$.
Let $\varphi^*$ be the automorphism of $\mathrm{wdP}^{(4)}_{2A_1}$
induced by the pull-backs of refined parabolic bundles by $\varphi$.

\begin{prop}The automorphism group of the weak del Pezzo
surface is $\mathbb Z/2\ltimes(\mathbb Z/2\times\mathbb Z/2)$, generated by:
$$\mathrm{Aut}\left(\mathrm{wdP}^{(4)}_{2A_1}\right)=\langle \varphi^*, \elm_{2[t_1]},\elm_{2[t_2]}  \rangle.$$
\end{prop}

\begin{proof} It is similar to the previous one. One can check:
\begin{itemize}
\item any automorphism permutes the two $(-2)$-curves $E^{(1)}_{[t_1]}$ and $E^{(1)}_{[t_2]}$,
and therefore stabilize the unique $(-1)$-curve intersecting them, namely $L_{[t_1],[t_2]}$;
\item $\varphi$ permutes non trivially $E^{(1)}_{[t_1]}$ and $E^{(1)}_{[t_2]}$;
\item $\elm_{2[t_1]}$ stabilizes $E^{(1)}_{[t_1]}$, $E^{(0)}_{[t_1]}$, $L_{[t_1],[t_3]}$, 
$E^{(1)}_{[t_2]}$ but permutes $E^{(0)}_{[t_2]}$ and $L_{[t_2],[t_3]}$ non trivially;
\item in a similar way, $\elm_{2[t_2]}$ stabilizes $E^{(1)}_{[t_1]}$, $E^{(1)}_{[t_2]}$, 
$E^{(0)}_{[t_2]}$ and $L_{[t_2],[t_3]}$, but permutes $E^{(0)}_{[t_1]}$ and  $L_{[t_1],[t_3]}$
non trivially.
\end{itemize}
Given an automorphism $\phi$, one can assume, after composition by an element of the $8$ order group
of the statement, that is stabilizes all $7$ curves $E^{(i)}_{[t_j]}$, $L_{[t_i],[t_j]}$; but combinatorics imply
that $\phi$ stabilizes all negative curves: its blow-down $\bar\phi$ on $\mathbb P^2$ has to preserve 
the conic and fix support of $D_{221}$. It is the identity.
\end{proof}

Now we set $D=D_{311}$.
The left-hand-side of Figure \ref{figure3}
is the projective plane $\mathbb P^2$ with the locus of special bundles.
The right-hand-side of Figure \ref{figure3} is
the graph of negative curves on
 the weak del Pezzo surface obtained by blowing-up of the left-hand-side.
This weak del Pezzo surface $\mathrm{wdP}^{(4)}_{A_2}$
 is just the moduli space $\overline{\BUN}_{\frac{1}{2}} (D_{311},1)$.
Note that the two thick lines are $(-2)$-curves.

\begin{figure}[h]

\begin{pspicture}(0,0)(15,5)


\psellipse[linecolor=green,linewidth=1pt](3.5,2.5)(1.9,1.3)
\psline[linecolor=brown,showpoints=false,linewidth=1pt](1,3.8)(6,3.8)
\psline[linecolor=orange,showpoints=false,linewidth=1pt](1,1.2)(4.1,4.4)
\psline[linecolor=orange,showpoints=false,linewidth=1pt](6,1.2)(2.9,4.4)
\psline[linecolor=cyan,showpoints=false,linewidth=1pt](1,2)(6,2)
\psline[linecolor=magenta,showpoints=false,linewidth=2pt]{->}(3.8,3.8)(4.4,3.63)
\psline[linecolor=purple,showpoints=false,linewidth=2pt]{->}(3.5,3.8)(4.1,3.8)
\psdot[linecolor=blue,linewidth=2pt](1.8,2)
\psdot[linecolor=red,linewidth=2pt](3.5,3.8)
\psdot[linecolor=blue,linewidth=2pt](5.2,2)
\uput[u](3.5,4){\red$P_{[t_1]}$}
\uput[r](5.2,2.3){\blue$P_{[t_2]}$}
\uput[l](1.8,2.3){\blue$P_{[t_3]}$}
\uput[u](3.5,1.1){\green$C$}
\uput[u](3.5,1.9){\cyan$L_{[t_2],[t_3]}$}
\uput[u](0.7,1.3){\color{orange} $L_{[t_1],[t_3]}$}
\uput[u](6.2,1.3){\color{orange} $L_{[t_1],[t_2]}$}
\uput[u](5.4,3.7){\color{brown}$L_{2[t_1]}$}



\psline[linecolor=red,showpoints=false,linewidth=2pt](8,3)(12,5)
\psline[linecolor=purple,showpoints=false,linewidth=2pt](15,3)(11,5)
\psline[linecolor=orange,showpoints=false,linewidth=1pt](8.6,3.6)(8.6,1)
\psline[linecolor=orange,showpoints=false,linewidth=1pt](10,4.6)(10,2)
\psline[linecolor=magenta,showpoints=false,linewidth=1pt](14.4,3.6)(14.4,1)
\psline[linecolor=brown,showpoints=false,linewidth=1pt](13,4.6)(13,2)

\psline[linecolor=cyan,showpoints=false,linewidth=1pt](14,3)(9,1)
\psline[linecolor=blue,showpoints=false,linewidth=1pt](8,2)(13,0)
\psline[linecolor=blue,showpoints=false,linewidth=1pt](9,3)(14,1)
\psline[linecolor=green,showpoints=false,linewidth=1pt](15,2)(10,0)

\uput[r](8.6,4.1){\red$E_{[t_1]}^{(2)}$}
\uput[l](14.25,4){\color{purple}$E_{[t_1]}^{(1)}$}
\uput[d](12.2,1.8){\blue$L_{[t_2]}$}
\uput[d](10.5,1.1){\blue$L_{[t_3]}$}
\uput[u](12.5,0.5){\green$C$}
\uput[d](12,3){\cyan$L_{[t_2],[t_3]}$}
\uput[u](10.6,3){\color{orange}$L_{[t_1],[t_3]}$}
\uput[u](9.2,2){\color{orange}$L_{[t_1],[t_2]}$}
\uput[d](12.6,3.8){\color{brown} $L_{2[t_1]}$}
\uput[ur](13.5,2){\color{magenta} $E_{[t_1]}^{(0)}$}
\end{pspicture}

\caption{$\mathbb{P}^2$ and $\mathrm{wdP}^{(4)}_{A_2}$}
\label{figure3}
\end{figure}

We take the automorphism $\varphi \in \mathrm{Aut}(C,D_{311})$ 
permuting $t_2$ and $t_3$, fixing $t_1$.
Let $\varphi^*$ be the automorphism of $\mathrm{wdP}^{(4)}_{A_2}$
induced by the pull-backs of refined parabolic bundles by $\varphi$.

\begin{prop}The automorphism group of the weak del Pezzo
surface is $\mathbb Z/2\ltimes(\mathbb Z/2\times\mathbb Z/2)$, generated by:
$$\mathrm{Aut}\left(\mathrm{wdP}^{(4)}_{A_2}\right)=\langle \varphi^*, \elm_{3[t_1]+[t_2]},\elm_{[t_2]+[t_3]}  \rangle.$$
\end{prop}


\begin{proof} It is similar to the previous ones. One can check:
\begin{itemize}
\item any automorphism permutes the two $(-2)$-curves $E^{(2)}_{[t_1]}$ and $E^{(1)}_{[t_1]}$;
\item $\elm_{3[t_1]+[t_2]}$ permutes non trivially $E^{(2)}_{[t_1]}$ and $E^{(1)}_{[t_1]}$;
\item $\varphi$ stabilizes $E^{(2)}_{[t_1]}$, $E^{(1)}_{[t_1]}$, $E^{(0)}_{[t_1]}$ and $L_{2[t_1]}$,
but permutes non trivially the two lines $L_{[t_1],[t_2]}$ and $L_{[t_1],[t_3]}$;
\item $\elm_{[t_2]+[t_3]}$ stabilizes $E^{(2)}_{[t_1]}$ and $E^{(1)}_{[t_1]}$,
but permutes non trivially the two lines $E^{(0)}_{[t_1]}$ and $L_{2[t_1]}$.
\end{itemize}
Given an automorphism $\phi$, one can assume, after composition by an element of the $8$ order group
of the statement, that is stabilizes all $6$ curves $E^{(i)}_{[t_1]}$, $L_{[t_1],[t_j]}$ and $L_{2[t_1]}$; but combinatorics imply
that $\phi$ stabilizes all negative curves: its blow-down $\bar\phi$ on $\mathbb P^2$ has to preserve 
the conic and fix support of $D_{311}$. It is the identity.
\end{proof}

Now we set $D=D_{32}$.
The left-hand-side of Figure \ref{figure4}
is the projective plane $\mathbb P^2$ with the locus of special bundles.
The right-hand-side of Figure \ref{figure4} is
the graph of negative curves on
 the weak del Pezzo surface obtained by blowing-up of the left-hand-side.
This weak del Pezzo surface $\mathrm{wdP}^{(4)}_{A_2+A_1}$
 is just the moduli space $\overline{\BUN}_{\frac{1}{2}} (D_{32},1)$.
 Note that the three thick lines are $(-2)$-curves.

\begin{figure}[h]
\begin{pspicture}(0,0)(15,5)


\psellipse[linecolor=green,linewidth=1pt](3.37,1.87)(1.87,1.25)
\psline[linecolor=brown,showpoints=false,linewidth=1pt](0,0)(3.5,5)
\psline[linecolor=brown,showpoints=false,linewidth=1pt](6.75,0)(3.25,5)
\psline[linecolor=orange,showpoints=false,linewidth=1pt](0,2.31)(6.75,2.31)
\psline[linecolor=magenta,showpoints=false,linewidth=2pt]{->}(1.76,2.5)(2.3,2.9)
\psline[linecolor=purple,showpoints=false,linewidth=2pt]{->}(1.65,2.31)(2,2.87)

\psline[linecolor=magenta,showpoints=false,linewidth=2pt]{->}(5.1,2.31)(4.71,2.87)
\psdot[linecolor=red,linewidth=2pt](1.65,2.31)
\psdot[linecolor=red,linewidth=2pt](5.06,2.31)
\uput[u](1.62,2.62){\red$P_{[t_1]}$}
\uput[u](5.12,2.62){\red$P_{[t_2]}$}
\uput[u](3.37,3.12){\green$C$}
\uput[u](3.5,2.2){\color{orange}$L_{[t_1],[t_2]}$}
\uput[u](0.56,1.25){\color{brown} $L_{2[t_1]}$}
\uput[u](6.19,1.25){\color{brown} $L_{2[t_2]}$}



\psline[linecolor=orange,showpoints=false,linewidth=1pt](8.5,0.5)(14.5,0.5)
\psline[linecolor=brown,showpoints=false,linewidth=1pt](8,3)(12,5)
\psline[linecolor=brown,showpoints=false,linewidth=1pt](15,3)(11,5)
\psline[linecolor=green,showpoints=false,linewidth=1pt](15,2)(11,4)
\psline[linecolor=magenta,showpoints=false,linewidth=1pt](8,2)(12,4)
\psline[linecolor=red,showpoints=false,linewidth=2pt](9,0)(11,2)
\psline[linecolor=red,showpoints=false,linewidth=2pt](12,2)(14,0)
\psline[linecolor=magenta,showpoints=false,linewidth=1pt](12,1)(15,4)
\psline[linecolor=purple,showpoints=false,linewidth=2pt](11,1)(8,4)
\uput[u](11.5,0.4){\color{orange}$L_{[t_1],[t_2]}$}
\uput[u](9.3,0.6){\red$E_{[t_1]}^{(2)}$}
\uput[u](9.5,1.5){\color{purple}$E_{[t_1]}^{(1)}$}
\uput[u](10.5,2.4){\color{magenta}$E_{[t_1]}^{(0)}$}
\uput[u](12.5,2.7){\color{green}$C$}
\uput[u](13.5,0.6){\red$E_{[t_2]}^{(1)}$}
\uput[u](13.5,1.5){\color{magenta}$E_{[t_2]}^{(0)}$}
\uput[u](9.5,3.8){\color{brown} $L_{2[t_1]}$}
\uput[u](13.5,3.8){\color{brown} $L_{2[t_2]}$}
\end{pspicture}

\caption{$\mathbb{P}^2$ and $\mathrm{wdP}^{(4)}_{A_2+A_1}$}
\label{figure4}
\end{figure}


Let $\mathrm{Aut}(C,D_{32})$ be the group of automorphisms  fixing $t_1$ and $t_2$.
We have an isomorphism $\mathrm{Aut}(C,D_{32})\cong \mathbb{C}^*$.

\begin{prop}The automorphism group of the weak del Pezzo
surface is $\mathbb Z/2\times\mathbb C^*$, generated by:
$$\mathrm{Aut}\left(\mathrm{wdP}^{(4)}_{A_1+A_2}\right)=\langle \elm_{2[t_1]}, \mathrm{Aut}(C,D_{32})  \rangle.$$
\end{prop}

\begin{proof} It is similar to the previous ones. One can check:
\begin{itemize}
\item any automorphism stabilizes $E^{(i)}_{[t_1]}$, $i=1,2$, $E^{(i)}_{[t_2]}$, $i=0,1$, and  $L_{[t_1],[t_2]}$
by intersection combinatorics between $(-2)$ and $(-1)$-curves;
\item $\elm_{2[t_2]}$ permutes $L_{2[t_1]}$ with $E^{(0)}_{[t_1]}$, and $L_{2[t_2]}$ with $C$.
\end{itemize}
Given an automorphism $\phi$, one can assume, after composition by $\elm_{2[t_2]}$
that is stabilizes all  negative curves and blow-down as a linear automorphism of $\mathbb P^2$ stabilizing the conic $C$
and fixing the two points $P_{[t_1]}$ and $P_{[t_2]}$. It is determined by its action on the conic;
there is an element $\varphi\in\mathrm{Aut}(C,D_{32})$ inducing the same action on the conic: 
$\phi$ coincides with the natural action of $\varphi$ on refined parabolic bundles.
So $\mathrm{Aut}\left(\mathrm{wdP}^{(4)}_{A_1+A_2}\right)$
is generated by $\elm_{2[t_1]}$ and $\mathrm{Aut}(C,D_{32})$.

Since $\varphi \in \mathrm{Aut}(C,D_{32})$ fixes $t_1$ and $t_2$, 
we may check that 
$$
\elm_{2[t_1]} (\varphi^* E, \varphi^* \bold{l}) = 
\varphi^*\left( \elm_{2[t_1]} (E,  \bold{l}) \right)
$$
for any $(E,  \bold{l}) \in \overline{\BUN}_{\frac{1}{2}} (D_{32},1)$.
That is, the automorphisms $\varphi^*$ and $\elm_{2[t_1]}$ commute.
So the automorphism group of the weak del Pezzo
surface is isomorphic to $\mathbb Z/2\times\mathbb C^*$.
\end{proof}

Now we set $D=D_{41}$.
The left-hand-side of Figure \ref{figure5}
is the projective plane $\mathbb P^2$ with the locus of special bundles.
The right-hand-side of Figure \ref{figure5} is
the graph of negative curves on
 the weak del Pezzo surface obtained by blowing-up of the left-hand-side.
This weak del Pezzo surface $\mathrm{wdP}^{(4)}_{A_3}$
 is just the moduli space $\overline{\BUN}_{\frac{1}{2}} (D_{41},1)$.
 Note that the three thick lines are $(-2)$-curves.

\begin{figure}[h]

\begin{pspicture}(0,0)(15,5)


\psellipse[linecolor=green,linewidth=1pt](3.37,1.87)(1.87,1.25)
\psline[linecolor=brown,showpoints=false,linewidth=1pt](0,0)(3.5,5)
\psline[linecolor=orange,showpoints=false,linewidth=1pt](0,2.31)(6.75,2.31)
\psline[linecolor=magenta,showpoints=false,linewidth=2pt]{->}(2,2.7)(2.6,3.1)
\psline[linecolor=purple,showpoints=false,linewidth=2pt]{->}(1.76,2.5)(2.3,3)
\psline[linecolor=violet,showpoints=false,linewidth=2pt]{->}(1.65,2.31)(2,2.87)
\psdot[linecolor=red,linewidth=2pt](1.65,2.31)
\psdot[linecolor=blue,linewidth=2pt](5.06,2.31)
\uput[u](1.42,2.42){\red$P_{[t_1]}$}
\uput[u](5.32,2.42){\blue$P_{[t_2]}$}
\uput[u](3.37,3.12){\green$C$}
\uput[u](3.5,2.2){\color{orange}$L_{[t_1],[t_2]}$}
\uput[u](0.56,1.25){\color{brown} $L_{2[t_1]}$}




\psline[linecolor=violet,showpoints=false,linewidth=2pt](9.5,4.5)(13.5,4.5)
\psline[linecolor=red,showpoints=false,linewidth=2pt](8,3)(11,5)
\psline[linecolor=purple,showpoints=false,linewidth=2pt](15,3)(12,5)
\psline[linecolor=brown,showpoints=false,linewidth=1pt](11.5,5)(11.5,2)
\psline[linecolor=orange,showpoints=false,linewidth=1pt](9,1)(9,4)
\psline[linecolor=magenta,showpoints=false,linewidth=1pt](14,1)(14,4)
\psline[linecolor=green,showpoints=false,linewidth=1pt](15,2)(11,0)
\psline[linecolor=blue,showpoints=false,linewidth=1pt](8,2)(12,0)
\uput[u](12,2.5){\color{brown} $L_{2[t_1]}$}
\uput[u](9.7,3.2){\color{red}$E_{[t_1]}^{(3)}$}
\uput[u](10.7,3.7){\color{violet}$E_{[t_1]}^{(2)}$}
\uput[u](13.1,3.2){\color{purple}$E_{[t_1]}^{(1)}$}
\uput[u](13.6,2){\color{magenta}$E_{[t_1]}^{(0)}$}
\uput[u](9.6,2){\color{orange}$L_{[t_1],[t_2]}$}
\uput[u](10.3,0.9){\color{blue}$E_{[t_2]}^{(0)}$}
\uput[u](13,1){\color{green}$C$}
\end{pspicture}

\caption{$\mathbb{P}^2$ and $\mathrm{wdP}^{(4)}_{A_3}$}
\label{figure5}
\end{figure}

Let $\mathrm{Aut}(C,D_{41})$ be the group of automorphisms fixing $t_1$ and $t_2$.

\begin{prop}The automorphism group of the weak del Pezzo
surface is $\mathbb Z/2\times\mathbb C^*$, generated by:
$$\mathrm{Aut}\left(\mathrm{wdP}^{(4)}_{A_3}\right)=\langle \elm_{4[t_1]}, \mathrm{Aut}(C,D_{41})  \rangle.$$
\end{prop}

\begin{proof} One can check:
\begin{itemize}
\item any automorphism stabilizes $E^{(2)}_{[t_1]}$ and $L_{2[t_1]}$;
\item $\elm_{4[t_1]}$ permutes $E^{(3)}_{[t_1]}$ with $E^{(1)}_{[t_1]}$.
\end{itemize}
Given an automorphism $\phi$, one can assume, after composition by $\elm_{4[t_1]}$
that is stabilizes all  negative curves, and therefore blow-down as a linear automorphism of $\mathbb P^2$ 
stabilizing the conic $C$ and fixing the two points $P_{[t_1]}$ and $P_{[t_2]}$. 
So $\mathrm{Aut}\left(\mathrm{wdP}^{(4)}_{A_3}\right)$
is generated by $\elm_{4[t_1]}$ and $\mathrm{Aut}(C,D_{41})$.
Since 
$\varphi \in \mathrm{Aut}(C,D_{41})$ fixes $t_1$ and $t_2$,
the automorphisms $\elm_{4[t_1]}$ and $\varphi^*$ commute.
\end{proof}

Now we set $D=D_{5}$.
The left-hand-side of Figure \ref{figure6}
is the projective plane $\mathbb P^2$ with the locus of special bundles.
The right-hand-side of Figure \ref{figure6} is
the graph of negative curves on
 the weak del Pezzo surface obtained by blowing-up of the left-hand-side.
This weak del Pezzo surface $\mathrm{wdP}^{(4)}_{A_4}$
 is just the moduli space $\overline{\BUN}_{\frac{1}{2}} (D_{5},1)$.
 Note that the three four lines are $(-2)$-curves.

\begin{figure}[h]
\begin{pspicture}(0,0)(15,5)


\psellipse[linecolor=green,linewidth=1pt](3.37,1.87)(1.87,1.25)
\psline[linecolor=brown,showpoints=false,linewidth=1pt](0,0)(3.5,5)
\psline[linecolor=pink,showpoints=false,linewidth=2pt]{->}(2.2,2.85)(2.9,3.1)
\psline[linecolor=magenta,showpoints=false,linewidth=2pt]{->}(2,2.7)(2.6,3.1)
\psline[linecolor=purple,showpoints=false,linewidth=2pt]{->}(1.76,2.5)(2.3,3)
\psline[linecolor=violet,showpoints=false,linewidth=2pt]{->}(1.65,2.31)(2,2.87)
\psdot[linecolor=red,linewidth=2pt](1.65,2.31)
\uput[u](1.42,2.42){\red$P_{[t_1]}$}
\uput[u](3.37,3.12){\green$C$}
\uput[u](0.56,1.25){\color{brown} $L_{2[t_1]}$}




\psline[linecolor=violet,showpoints=false,linewidth=2pt](9.5,4.5)(13.5,4.5)
\psline[linecolor=red,showpoints=false,linewidth=2pt](8,3)(11,5)
\psline[linecolor=purple,showpoints=false,linewidth=2pt](15,3)(12,5)
\psline[linecolor=brown,showpoints=false,linewidth=1pt](11.5,5)(11.5,2)
\psline[linecolor=magenta,showpoints=false,linewidth=2pt](14,1)(14,4)
\psline[linecolor=pink,showpoints=false,linewidth=1pt](15,2)(11,0)
\psline[linecolor=green,showpoints=false,linewidth=1pt](8,2)(12,0)
\uput[u](12,2.5){\color{brown} $L_{2[t_1]}$}
\uput[u](9.7,3.2){\color{red}$E_{[t_1]}^{(4)}$}
\uput[u](10.7,3.7){\color{violet}$E_{[t_1]}^{(3)}$}
\uput[u](13.1,3.2){\color{purple}$E_{[t_1]}^{(2)}$}
\uput[u](13.6,2){\color{magenta}$E_{[t_1]}^{(1)}$}
\uput[u](10.3,0.9){\color{green}$C$}
\uput[u](13,1){\color{pink}$E_{[t_1]}^{(0)}$}
\end{pspicture}

\caption{$\mathbb{P}^2$ and $\mathrm{wdP}^{(4)}_{A_4}$}
\label{figure6}
\end{figure}

Let $\mathrm{Aut}(C,D_5)$ be the group of automorphisms fixing $t_1$.

\begin{prop}The automorphism group of the weak del Pezzo
surface is:
$$\mathrm{Aut}\left(\mathrm{wdP}^{(4)}_{A_4}\right)=\mathrm{Aut}(C,D_5).$$
\end{prop}

\begin{proof} By intersection combinatorics, any automorphism $\phi$ must stabilize 
all negative curves, and therefore blow-down as a linear automorphism of $\mathbb P^2$ 
stabilizing the conic $C$ and fixing the point $P_{[t_1]}$. 
\end{proof}

\section{Appendix: List of special bundles}\label{section:2022.1.3.0.11}

In this appendix, we will give a list of special bundles for $n=5$.
First, we give a list of standard tableaus
of refined parabolic structures, which appear in the cases where $n=5$.
If the multiplicity of a point is 2, 
then the standard tableau of a refined parabolic structure
at this point has two possibility:
$$
T^{(2)}_{I} \colon  (1,1) \supset (1) 
\qquad T^{(2)}_{II} \colon  (2) \supset (1) .
$$
If the multiplicity of a point is 3, 
then the standard tableau of a refined parabolic structure
at this point has three possibility:
$$
\begin{aligned}
&T^{(3)}_{I} \colon (1,1,1) \supset T^{(2)}_{I} &
&T^{(3)}_{II} \colon (1,2) \supset T^{(2)}_{II} &
&T^{(3)}_{III} \colon (1,2) \supset T^{(2)}_{I} .
\end{aligned}
$$
If the multiplicity of a point is 4, 
then the standard tableau of a refined parabolic structure
at this point has six possibility:
$$
\begin{aligned}
&T^{(4)}_{I} \colon (1,1,1,1) \supset T^{(3)}_{I} &
&T^{(4)}_{II} \colon (1,1,2) \supset T^{(3)}_{II} &
&T^{(4)}_{III} \colon (1,1,2) \supset T^{(3)}_{III}\\
&T^{(4)}_{IV} \colon (1,1,2) \supset T^{(3)}_{I} &
&T^{(4)}_{V} \colon (2,2) \supset T^{(3)}_{II}&
&T^{(4)}_{VI} \colon (2,2) \supset T^{(3)}_{III}.
\end{aligned}
$$
If the multiplicity of a point is 5, 
then the standard tableau of a refined parabolic structure
at this point has ten possibility:
$$
\begin{aligned}
&T^{(5)}_{I} \colon (1,1,1,1,1) \supset T^{(4)}_{I} &
&T^{(5)}_{II} \colon (1,1,1,2) \supset T^{(4)}_{II} &
&T^{(5)}_{III} \colon (1,1,1,2) \supset T^{(4)}_{III} \\
&T^{(5)}_{IV} \colon (1,1,1,2) \supset T^{(4)}_{IV} &
&T^{(5)}_{V} \colon (1,1,1,2) \supset T^{(4)}_{I}&
&T^{(5)}_{VI} \colon (1,2,2) \supset T^{(4)}_{V} \\
&T^{(5)}_{VII} \colon (1,2,2) \supset   T^{(4)}_{VI} &
&T^{(5)}_{VIII} \colon (1,2,2) \supset   T^{(4)}_{II} &
&T^{(5)}_{IX} \colon (1,2,2) \supset   T^{(4)}_{III} \\
&T^{(5)}_{X} \colon (1,2,2) \supset   T^{(4)}_{IV}.
\end{aligned}
$$
For giving a list of special bundles for $n=5$, now, 
we discuss some condition for refined parabolic structures.
For a subbundle $L \subset E$, we define an integer 
$m_{t_i, L}$ by
$$
m_{t_i, L} =   \sum_{k=1}^{n_i} \mathrm{length}   
((l_{i,k}\cap L|_{n_i[t_i]})/(l_{i,k-1}\cap L|_{n_i[t_i]})).
$$
We assume that $(E, \bold{l})$ is tame and undecomposable and $\deg(E)=1$. 
In particular, $(E, \bold{l})$ is admissible.
So we have the inequality
$$
\sum_{i \in I} m_{t_i, L} \leq 4 -2 \deg(L)
$$
for any line subbundles $L$ such that $\deg(E) \leq 2\deg(L)$.
So we have the following facts:
\begin{itemize}
\item $(E, \bold{l})$ with $E \cong \O(-k) \oplus \O(1+k)$ (where $k\geq 2$)
disappears;
\item if $E=\O\oplus \O(1)$ and $L=\O(1)$,
then $\sum_{i \in I} m_{t_i, \O(1)} \leq 2$;
\item if $E=\O(-1)\oplus \O(2)$ and $L=\O(2)$,
then $ m_{t_1, \O(2)}=\cdots= m_{t_n, \O(2)}=0$.
\end{itemize}
If $l_{i,n_i}$ is not free,
then $m_{t_i, L} \geqq 1$
for any line subbundles $L$ such that $\deg(E) \leq 2\deg(L)$.
So if $E=\O(-1)\oplus \O(2)$, then 
$l_{i,n_i}$ are free for any $i$.

Now we consider the cases where $E=\mathcal{O}\oplus \mathcal{O}(1)$.
We impose $m_{t_i, \mathcal{O}(1)} \leq 2$ for $i\in I$.
We set $\epsilon_{i,k} := \epsilon_{i,k}(\mathcal{O}(1))$
for $i \in I$ and $k \in \{1,2,\ldots,n_i\}$.
When we will check the tameness of refined parabolic bundles,
we will calculate $N_i(\mathcal{O}(1))$, which is defined in \eqref{eq:2021.12.28.10.22}.
Now the tameness means that 
$$
2 \leq \sum_{i \in I^+_{\mathcal{O}(1)}}N_i(\mathcal{O}(1)).
$$
To calculate $N_i(\mathcal{O}(1))$, we consider the tuple 
$(\epsilon_{i,n_i},\ldots,\epsilon_{i,1})$.
The tuple $(\epsilon_{i,n_i},\ldots,\epsilon_{i,1})$
is generically determined by the pair $(T_{l_{i,\bullet}},m_{t_i,\mathcal{O}(1)})$.
Here $T_{l_{i,\bullet}}$ is the standard tableau of the refined parabolic structure $l_{i,\bullet}$ at $t_i$.
Now we will describe the correspondence between 
$(\epsilon_{i,n_i},\ldots,\epsilon_{i,1})$ 
and $(T_{l_{i,\bullet}},m_{t_i,\mathcal{O}(1)})$ in 
Table \ref{tab:2021.12.28_1},
Table \ref{tab:2021.12.28_2},
Table \ref{tab:2021.12.28_3}, and
Table \ref{tab:2021.12.28_4}.
\begin{table}[hbtp]
  \caption{The correspondence with $n_i=2$}
  \label{tab:2021.12.28_1}
  \centering
  \begin{tabular}{|c|c|c|}
    \hline
   $(T_{l_{i,\bullet}},m_{t_i,\mathcal{O}(1)})$
   & $(\epsilon_{i,2},\epsilon_{i,1})$ 
    &  $N_i(\mathcal{O}(1))$\\
    \hline\hline 
    $ (T_I^{(2)}, 1)$
     &$ (+,-)$
     & $0$  \\
    \hline
     $(T_{II}^{(2)}, 1)$
     & $(-,+)$
     & $1$  \\
         \hline
     $(T_{II}^{(2)}, 2)$
     & $(-,-)$
     & $-2$  \\
    \hline
  \end{tabular}
\end{table}
\begin{table}[hbtp]
  \caption{The correspondence with $n_i=3$}
  \label{tab:2021.12.28_2}
  \centering
  \begin{tabular}{|c|c|c||c|c|c|}
    \hline
   $(T_{l_{i,\bullet}},m_{t_i,\mathcal{O}(1)})$
   & $(\epsilon_{i,3},\epsilon_{i,2},\epsilon_{i,1})$ 
    &  $N_i(\mathcal{O}(1))$
    & $(T_{l_{i,\bullet}},m_{t_i,\mathcal{O}(1)})$
   & $(\epsilon_{i,3},\epsilon_{i,2},\epsilon_{i,1})$ 
    &  $N_i(\mathcal{O}(1))$ \\
    \hline\hline 
    $ (T_I^{(3)}, 1)$
     &$ (+,+,-)$
     & $1$ &     $(T_{I}^{(3)}, 2)$
     & $(+,-,-)$
     & $-1$  \\
    \hline
     $(T_{II}^{(3)}, 1)$
     & $(+,-,+)$
     & $1$ &$(T_{II}^{(3)}, 2)$
     & $(-,-,+)$
     & $1$  \\
         \hline
     $(T_{III}^{(3)}, 1)$
     & $(-,+,+)$
     & $2$ &         $(T_{III}^{(3)}, 2)$
     & $(-,+,-)$
     & $0$   \\
    \hline
      \end{tabular}
\end{table}
\begin{table}[hbtp]
  \caption{The correspondence with $n_i=4$}
  \label{tab:2021.12.28_3}
  \centering
    \begin{tabular}{|c|c|c||c|c|c|}
    \hline
   $(T_{l_{i,\bullet}},m_{t_i,\mathcal{O}(1)})$
   & $(\epsilon_{i,4},\epsilon_{i,3},\epsilon_{i,2},\epsilon_{i,1})$ 
    &  $N_i(\mathcal{O}(1))$
    & $(T_{l_{i,\bullet}},m_{t_i,\mathcal{O}(1)})$
   & $(\epsilon_{i,4},\epsilon_{i,3},\epsilon_{i,2},\epsilon_{i,1})$ 
    &  $N_i(\mathcal{O}(1))$ \\
    \hline\hline 
    $ (T_I^{(4)}, 1)$
     &$ (+,+,+,-)$
     & $2$ &     $(T_{I}^{(4)}, 2)$
     & $(+,+,-,-)$
     & $0$  \\
    \hline
     $(T_{II}^{(4)}, 1)$
     & $(+,+,-,+)$
     & $2$ &$(T_{II}^{(4)}, 2)$
     & $(+,-,-,+)$
     & $1$  \\
         \hline
     $(T_{III}^{(4)}, 1)$
     & $(+,-,+,+)$
     & $2$ &         $(T_{III}^{(4)}, 2)$
     & $(+,-,+,-)$
     & $0$   \\
    \hline
     $(T_{IV}^{(4)}, 1)$
     & $(-,+,+,+)$
     & $3$ &         $(T_{IV}^{(4)}, 2)$
     & $(-,+,+,-)$
     & $1$   \\
    \hline
         $(T_{V}^{(4)}, 2)$
     & $(-,+,-,+)$
     & $1$ &         $(T_{VI}^{(4)}, 2)$
     & $(-,-,+,+)$
     & $2$   \\
    \hline
     \end{tabular}
\end{table}
\begin{table}[hbtp]
  \caption{The correspondence with $n_i=5$}
  \label{tab:2021.12.28_4}
  \centering
    \begin{tabular}{|c|c|c||c|c|c|}
    \hline
   $(T_{l_{i,\bullet}},m_{t_i,\mathcal{O}(1)})$
   & $(\epsilon_{i,5},\epsilon_{i,4},\epsilon_{i,3},\epsilon_{i,2},\epsilon_{i,1})$ 
    &  $N_i(\mathcal{O}(1))$
    & $(T_{l_{i,\bullet}},m_{t_i,\mathcal{O}(1)})$
   & $(\epsilon_{i,5},\epsilon_{i,4},\epsilon_{i,3},\epsilon_{i,2},\epsilon_{i,1})$ 
    &  $N_i(\mathcal{O}(1))$ \\
    \hline\hline 
    $ (T_I^{(5)}, 1)$
     &$ (+,+,+,+,-)$
     & $3$ &     $(T_{I}^{(5)}, 2)$
     & $(+,+,+,-,-)$
     & $1$  \\
    \hline
     $(T_{II}^{(5)}, 1)$
     & $(+,+,+,-,+)$
     & $3$ &$(T_{II}^{(5)}, 2)$
     & $(+,+,-,-,+)$
     & $1$  \\
         \hline
     $(T_{III}^{(5)}, 1)$
     & $(+,+,-,+,+)$
     & $3$ &         $(T_{III}^{(5)}, 2)$
     & $(+,+,-,+,-)$
     & $1$   \\
    \hline
     $(T_{IV}^{(5)}, 1)$
     & $(+,-,+,+,+)$
     & $3$ &         $(T_{IV}^{(5)}, 2)$
     & $(+,-,+,+,-)$
     & $1$   \\
    \hline
          $(T_{V}^{(5)}, 1)$
     & $(-,+,+,+,+)$
     & $4$ &         $(T_{V}^{(5)}, 2)$
     & $(-,+,+,+,-)$
     & $2$   \\
    \hline
         $(T_{VI}^{(5)}, 2)$
     & $(+,-,+,-,+)$
     & $1$ &         $(T_{VII}^{(5)}, 2)$
     & $(+,-,-,+,+)$
     & $2$   \\
    \hline
         $(T_{VIII}^{(5)}, 2)$
     & $(-,+,+,-,+)$
     & $2$ &$(T_{IX}^{(5)}, 2)$
     & $(-,+,-,+,+)$
     & $2$   \\
    \hline
       $(T_{X}^{(5)}, 2)$
     & $(-,-,+,+,+)$
     & $3$   &&& \\
    \hline
     \end{tabular}
\end{table}

\pagebreak

\subsection{Case $D=D_{2111}$}

Now we describe special refined parabolic bundles which are undecomposable and tame
when $D=D_{2111}$.

\begin{itemize}
\item[] {\bf [Type A].} $E = \O \oplus \O(1)$ and $m_{t_1,\O(1)}=\cdots=m_{t_4,\O(1)}=0$.
\begin{center}
  \begin{tabular}{|c||c|c|}
  \hline
  & $T_{l_{1,\bullet}}$& 
  there exists a subundle $\O(-1) \subset E$ as below:\\
    \hline\hline
    $C$ &$T^{(2)}_{I}$
    & $\begin{array}{l}
\text{$m_{t_1, \O(-1)} =2,\ m_{t_2, \O(-1)}=m_{t_3, \O(-1)}=m_{t_4, \O(-1)}=1$}
\end{array}$  \\
    \hline
  \end{tabular}
\end{center}

\item[] {\bf [Type B].} $E = \O \oplus \O(1)$ and $m_{t_1,\O(1)}=\cdots=m_{t_4,\O(1)}=0$.
\begin{center}
  \begin{tabular}{|c||c|c|}
  \hline
  & $T_{l_{1,\bullet}}$& 
  there exists a subundle $\O \subset E$ as below:\\
    \hline\hline
    $L_{[t_i],[t_j]}$ &$T^{(2)}_{I}$
     & $\begin{array}{l}
\text{$m_{t_1, \O} =2$, $m_{t_k, \O} =1$ where $k \in \{2,3,4\} \setminus \{i,j\}$.}
\end{array}$ \\
    \hline
    $L_{[t_1],[t_i]}$  &$T^{(2)}_{I}$
    & $\begin{array}{l}
\text{$m_{t_1, \O} =m_{t_{j}, \O}=m_{t_{k}, \O} =1$ where $\{j,k\}  = \{ 2,3,4\} \setminus \{ i\}$.}
\end{array}$ \\
    \hline
    $L_{2[t_1]}$  &$T^{(2)}_{I}$
    & $\begin{array}{l}
\text{$m_{t_2, \O}=m_{t_3, \O} = m_{t_4, \O}=1$}
\end{array}$ \\
    \hline
  \end{tabular}
\end{center}
for $i,j \in \{ 2,3,4\}$ with $i \neq j$.

\item[] {\bf [Type C].} 
$$
P_{[t_1]} := C \cap L_{2[t_1]},\qquad 
P_{[t_i]} := C \cap \left( \bigcap_{j \in \{ 2,3,4\} \setminus \{i\}} L_{[t_i],[t_j]} \right)
\quad i \in \{ 2,3,4\}.
$$

\item[] {\bf [Type D].} $E = \O \oplus \O(1)$.
\begin{center}
  \begin{tabular}{|c||c|c|}
  \hline
  & $T_{l_{1,\bullet}}$& 
  there exists a subundle $\O(1) \subset E$ as below:\\
    \hline\hline
 $E^{(0)}_{[t_i]}$ &$T^{(2)}_{I} $
 &$\begin{array}{l} \text{$m_{t_i, \O(1)}=1$ }
 \end{array}$\\
\hline
  $E^{(1)}_{[t_1]}$ &$T^{(2)}_{II} $ &
 $\begin{array}{l} \text{$m_{t_1, \O(1)}=1$}
 \end{array}$\\
     \hline 
  \end{tabular}
\end{center}
for $i,j \in \{ 1,2,3,4\}$.

\item[] {\bf [Type E].} $E = \O \oplus \O(1)$.
\begin{center}
  \begin{tabular}{|c||c|c|}
  \hline
  & $T_{l_{1,\bullet}}$& 
  there exists a subundle $\O(1) \subset E$ as below:\\
    \hline\hline
   $Q^{(0)}_{[t_i],[t_j]}$ &$T^{(2)}_{I} $ &
 $\begin{array}{l} \text{$m_{t_i, \O(1)}=m_{t_{j}, \O(1)}=1$.}
 \end{array}$\\
 \hline
   $Q^{(0)}_{2[t_1]}$ &$T^{(2)}_{I} $ &
 $\begin{array}{l} \text{$m_{t_1, \O(1)}=2$}
 \end{array}$\\
 \hline
    $Q^{(1)}_{[t_1],[t_i']}$ &$T^{(2)}_{II} $ &
 $\begin{array}{l} \text{$m_{t_1, \O(1)}=m_{t_{i'}, \O(1)}=1$.}
 \end{array}$\\
 \hline
  \end{tabular}
\end{center}
for $i,j \in \{ 1,2,3,4\}$ with $i \neq j$ and for $i' \in \{2,3,4\}$.

\item[] {\bf [Type F].}
$E = \O(-1) \oplus \O(2)$ and $T_{l_{t_1,\bullet}}\colon (1,1) \supset (1)$.

\end{itemize}

\subsection{Case $D=D_{221}$}

Now we describe special refined parabolic bundles which are undecomposable and tame
when $D=D_{221}$.

\begin{itemize}
\item[] {\bf [Type A].} 
$E = \O \oplus \O(1)$ and $m_{t_1,\O(1)}=m_{t_2,\O(1)}=m_{t_3,\O(1)}=0$.
\begin{center}
  \begin{tabular}{|c||c|c|c|}
    \hline
     & $T_{l_{1,\bullet}}$ & $T_{l_{2,\bullet}}$ 
    &  there exists a subundle $\O(-1) \subset E$ as below:\\
    \hline\hline 
 $C$  &$T^{(2)}_{I} $ &$T^{(2)}_{I} $
    & $\begin{array}{l}
\text{$m_{t_1, \O(-1)} =m_{t_{2}, \O(-1)} =2,m_{t_{3}, \O(-1)} =1$.}
\end{array}$ \\
    \hline
  \end{tabular}
\end{center}

\item[] {\bf [Type B].} 
$E = \O \oplus \O(1)$ and $m_{t_1,\O(1)}=m_{t_2,\O(1)}=m_{t_3,\O(1)}=0$.
\begin{center}
  \begin{tabular}{|c||c|c|c|}
    \hline
    & $T_{l_{1,\bullet}}$ & $T_{l_{2,\bullet}}$ 
    &  there exists a subundle $\O \subset E$ as below:\\
    \hline\hline 
   $L_{[t_1],[t_2]}$  &$T^{(2)}_{I} $ &$T^{(2)}_{I} $
    & $\begin{array}{l}
\text{$m_{t_1, \O} =m_{t_{2}, \O}=m_{t_{3}, \O} =1$.}
\end{array}$ \\
    \hline
    $L_{[t_i],[t_3]}$  &$T^{(2)}_{I} $ &$T^{(2)}_{I} $
    & $\begin{array}{l}
\text{$m_{t_{i'}, \O} =2, m_{t_{3}, \O} =1$ where $i'  \in \{ 1,2\} \setminus \{ i\}$.}
\end{array}$ \\
    \hline
    $L_{2[t_i]}$  &$T^{(2)}_{I} $ &$T^{(2)}_{I} $
    & $\begin{array}{l}
\text{$m_{t_{i'}, \O}=2, m_{t_3, \O} = 1$ where $i' \in \{1,2 \} \setminus \{i\}$.}
\end{array}$ \\
    \hline
  \end{tabular}
\end{center}
for $i \in \{ 1,2\}$.

\item[] {\bf [Type C].}
$$
P_{[t_1]} := C \cap L_{2[t_1]},\qquad P_{[t_2]} := C \cap L_{2[t_2]},\qquad 
P_{[t_3]} := C \cap L_{[t_1],[t_3]} \cap L_{[t_2],[t_3]}.
$$

\item[] {\bf [Type D].} $E = \O \oplus \O(1)$.
\begin{center}
  \begin{tabular}{|c||c|c|c|}
    \hline
    & $T_{l_{1,\bullet}}$ & $T_{l_{2,\bullet}}$ 
    &  there exists a subundle $\O(1) \subset E$ as below:\\
    \hline\hline
 $E^{(0)}_{[t_i]}$ &$T^{(2)}_{I} $ &$T^{(2)}_{I} $ &
 $\begin{array}{l} \text{$m_{t_i, \O(1)}=1$ }
 \end{array}$\\
 \hline 
   $E^{(1)}_{[t_1]}$ &$T^{(2)}_{II} $ &$T^{(2)}_{I} $ &
 $\begin{array}{l} \text{$m_{t_1, \O(1)}=1$}
 \end{array}$\\
 \hline 
   $E^{(1)}_{[t_2]}$ &$T^{(2)}_{I} $ &$T^{(2)}_{II} $ &
 $\begin{array}{l} \text{$m_{t_2, \O(1)}=1$ }
 \end{array}$\\
 \hline  
  \end{tabular}
\end{center}
for $i \in \{ 1,2,3\}$.

\item[] {\bf [Type E].} $E = \O \oplus \O(1)$.
\begin{center}
  \begin{tabular}{|c||c|c|c|}
    \hline
    & $T_{l_{1,\bullet}}$ & $T_{l_{2,\bullet}}$ 
    &  there exists a subundle $\O(1) \subset E$ as below: \\
    \hline\hline
    $Q^{(0)}_{[t_i],[t_3]}$ &$T^{(2)}_{I} $ &$T^{(2)}_{I} $&
 $\begin{array}{l} \text{$m_{t_i, \O(1)}=m_{t_{3}, \O(1)}=1$.}
 \end{array}$   \\
 \hline
    $Q^{(1)}_{[t_{i'}],[t_2]}$ &$T^{(2)}_{I} $ &$T^{(2)}_{II} $&
 $\begin{array}{l} \text{$m_{t_{i'}, \O(1)}=m_{t_{2}, \O(1)}=1$.}
 \end{array}$ \\
 \hline
    $Q^{(2)}_{[t_1],[t_{j'}]}$ &$T^{(2)}_{II} $ &$T^{(2)}_{I} $&
 $\begin{array}{l} \text{$m_{t_1, \O(1)}=m_{t_{j'}, \O(1)}=1$.}
 \end{array}$\\
 \hline
  $Q^{(3)}_{[t_1],[t_2]}$ &$T^{(2)}_{II} $ &$T^{(2)}_{II} $ &
 $\begin{array}{l} \text{$m_{t_1, \O(1)}=m_{t_2, \O(1)}=1$}
 \end{array}$ \\
 \hline
    $Q^{(0)}_{2[t_i]}$ &$T^{(2)}_{I} $ &$T^{(2)}_{I} $&
 $\begin{array}{l} \text{$m_{t_i, \O(1)}=2$.}
 \end{array}$\\
  \hline
  \end{tabular}
\end{center}
for $i\in \{ 1,2\}$, and $i' \in \{ 1,3\}$, $j' \in \{ 2,3\}$.

\item[] {\bf [Type F].} 
$E = \O(-1) \oplus \O(2)$ and $T_{l_{t_1,\bullet}}\colon (1,1) \supset (1)$.

\end{itemize}

\subsection{Case $D=D_{311}$}

Now we describe special refined parabolic bundles which are undecomposable and tame
when $D=D_{311}$.

\begin{itemize}

\item[] {\bf [Type A].} 
$E = \O \oplus \O(1)$ and $m_{t_1,\O(1)}=m_{t_2,\O(1)}=m_{t_3,\O(1)}=0$.
\begin{center}
  \begin{tabular}{|c||c|c|}
    \hline
     & $T_{l_{1,\bullet}}$ 
    &  there exists a subundle $\O(-1) \subset E$ as below:\\
    \hline\hline 
    $C$ &$T^{(3)}_{I} $ 
    & $\begin{array}{l}
\text{$m_{t_1, \O(-1)} =3, m_{t_2, \O(-1)}=m_{t_3, \O(-1)}=1$}
\end{array}$  \\
    \hline
  \end{tabular}
\end{center}

\item[] {\bf [Type B].} 
$E = \O \oplus \O(1)$ and $m_{t_1,\O(1)}=m_{t_2,\O(1)}=m_{t_3,\O(1)}=0$.
\begin{center}
  \begin{tabular}{|c||c|c|}
    \hline
     & $T_{l_{1,\bullet}}$ 
    &  there exists a subundle $\O \subset E$ as below:\\
    \hline\hline 
    $L_{[t_2],[t_3]}$ &$T^{(3)}_{I} $
     & $\begin{array}{l}
\text{$m_{t_1, \O} =3$.}
\end{array}$ \\
    \hline
    $L_{[t_1],[t_i]}$  &$T^{(3)}_{I} $
    & $\begin{array}{l}
\text{$m_{t_1, \O} =2, m_{t_{i'}, \O} =1$ where $i' \in \{ 2,3\} \setminus \{ i\}$.}
\end{array}$ \\
    \hline
    $L_{2[t_1]}$  &$T^{(3)}_{I} $ 
    & $\begin{array}{l}
\text{$m_{t_1, \O}=m_{t_2, \O} = m_{t_3, \O}=1$}
\end{array}$ \\
    \hline
  \end{tabular}
\end{center}

\item[] {\bf [Type C].}
$$
P_{[t_1]} := C \cap L_{2[t_1]},\quad 
P_{[t_2]} := C \cap L_{[t_1],[t_2]}\cap L_{[t_2],[t_3]},\quad 
P_{[t_3]} := C \cap L_{[t_1],[t_3]}\cap L_{[t_2],[t_3]}.
$$

\item[] {\bf [Type D].} $E = \O \oplus \O(1)$.
\begin{center}
  \begin{tabular}{|c||c|c|}
    \hline
     & $T_{l_{1,\bullet}}$
    &  there exists a subundle $\O(1) \subset E$ as below:\\
    \hline\hline
 $E^{(0)}_{[t_i]}$ &$T^{(3)}_{I} $ &
 $\begin{array}{l} \text{$m_{t_i, \O(1)}=1$ }
 \end{array}$\\
 \hline 
   $E^{(1)}_{[t_1]}$ &$T^{(3)}_{II} $ &
 $\begin{array}{l} \text{$m_{t_1, \O(1)}=1$}
 \end{array}$\\
 \hline 
   $E^{(2)}_{[t_1]}$ &$T^{(3)}_{III} $ &
 $\begin{array}{l} \text{$m_{t_1, \O(1)}=1$ }
 \end{array}$\\
 \hline 
  \end{tabular}
\end{center}
for $i \in \{ 1,2,3\}$.

\item[] {\bf [Type E].} $E = \O \oplus \O(1)$.
\begin{center}
  \begin{tabular}{|c||c|c|}
    \hline
     & $T_{l_{1,\bullet}}$
    &  there exists a subundle $\O(1) \subset E$ as below:\\
    \hline\hline
   $Q^{(0)}_{[t_i],[t_j]}$ &$T^{(3)}_{I} $ &
 $\begin{array}{l} \text{$m_{t_i, \O(1)}=m_{t_{j}, \O(1)}=1$.}
 \end{array}$\\
 \hline
    $Q^{(0)}_{[t_1],[t_{j'}]}$ &$T^{(3)}_{II} $ &
 $\begin{array}{l} \text{$m_{t_1, \O(1)}=m_{t_{j'}, \O(1)}=1$.}
 \end{array}$\\
  \hline
    $Q^{(0)}_{[t_1],[t_{j'}]}$ &$T^{(3)}_{III} $ &
 $\begin{array}{l} \text{$m_{t_1, \O(1)}=m_{t_{j'}, \O(1)}=1$.}
 \end{array}$\\
 \hline
  $Q^{(0)}_{2[t_1]}$ &$T^{(3)}_{I} $ &
 $\begin{array}{l} \text{$m_{t_1, \O(1)}=2$}
 \end{array}$\\
 \hline
 $Q^{(1)}_{2[t_1]}$ &$T^{(3)}_{II} $ &
 $\begin{array}{l} \text{$m_{t_1, \O(1)}=2$}
 \end{array}$\\
 \hline
 $Q^{(2)}_{2[t_1]}$ &$T^{(3)}_{III} $ &
 $\begin{array}{l} \text{$m_{t_1, \O(1)}=2$ }
 \end{array}$\\
 \hline
  \end{tabular}
\end{center}
for $i ,j \in \{ 1,2,3\}$ with $i\neq j$ and $j' \in \{2,3\}$.

\item[] {\bf [Type F].}
 $E = \O(-1) \oplus \O(2)$ and $T_{l_{1,\bullet}}\colon (1,1,1) \supset (1,1) \supset (1)$.

\end{itemize}

\subsection{Case $D=D_{32}$}

Now we describe special refined parabolic bundles which are undecomposable and tame
when $D=D_{32}$.

\begin{itemize}

\item[] {\bf [Type A].} $E = \O \oplus \O(1)$ and $m_{t_1,\O(1)}=m_{t_2,\O(1)}=0$.
\begin{center}
  \begin{tabular}{|c||c|c|c|}
    \hline
     & $T_{l_{1,\bullet}}$     & $T_{l_{2,\bullet}}$
    &  there exists a subundle $\O(-1) \subset E$ as below:\\
    \hline\hline 
    $C$ &$T^{(3)}_{I} $ &$T^{(2)}_{I} $ 
    & $\begin{array}{l}
\text{$m_{t_1, \O(-1)} =3, m_{t_2, \O(-1)}=2$}
\end{array}$  \\
    \hline
  \end{tabular}
\end{center}

\item[] {\bf [Type B].} $E = \O \oplus \O(1)$ and $m_{t_1,\O(1)}=m_{t_2,\O(1)}=0$.
\begin{center}
  \begin{tabular}{|c||c|c|c|}
    \hline
     & $T_{l_{1,\bullet}}$ & $T_{l_{2,\bullet}}$ 
    &  there exists a subundle $\O \subset E$ as below:\\
    \hline\hline 
    $L_{2[t_2]}$ &$T^{(3)}_{I} $&$T^{(2)}_{I} $
     & $\begin{array}{l}
\text{$m_{t_1, \O} =3$.}
\end{array}$ \\
    \hline
    $L_{[t_1],[t_2]}$  &$T^{(3)}_{I} $&$T^{(2)}_{I} $
    & $\begin{array}{l}
\text{$m_{t_1, \O} =2, m_{t_{2}, \O} =1$.}
\end{array}$ \\
    \hline
    $L_{2[t_1]}$  &$T^{(3)}_{I} $ &$T^{(2)}_{I} $
    & $\begin{array}{l}
\text{$m_{t_1, \O}=1,m_{t_2, \O} =2$.}
\end{array}$ \\
    \hline
  \end{tabular}
\end{center}

\item[] {\bf [Type C].}
$$
P_{[t_1]} := C \cap L_{2[t_1]},\qquad P_{[t_2]} := C \cap L_{2[t_2]}.
$$

\item[] {\bf [Type D].} $E = \O \oplus \O(1)$.
\begin{center}
  \begin{tabular}{|c||c|c|c|}
    \hline
     & $T_{l_{1,\bullet}}$&$T_{l_{2,\bullet}} $ 
    &  there exists a subundle $\O(1) \subset E$ as below:\\
    \hline\hline
 $E^{(0)}_{[t_i]}$ &$T^{(3)}_{I} $ &$T^{(2)}_{I} $ &
 $\begin{array}{l} \text{$m_{t_i, \O(1)}=1$ }
 \end{array}$\\
 \hline 
   $E^{(1)}_{[t_1]}$ &$T^{(3)}_{II} $ &$T^{(2)}_{I} $ &
 $\begin{array}{l} \text{$m_{t_1, \O(1)}=1$}
 \end{array}$\\
 \hline 
    $E^{(1)}_{[t_2]}$ &$T^{(3)}_{I} $ &$T^{(2)}_{II} $ &
 $\begin{array}{l} \text{$m_{t_2, \O(1)}=1$}
 \end{array}$\\
 \hline
   $E^{(2)}_{[t_1]}$ &$T^{(3)}_{III} $ &$T^{(2)}_{I} $ &
 $\begin{array}{l} \text{$m_{t_1, \O(1)}=1$ }
 \end{array}$\\
 \hline 
  \end{tabular}
\end{center}
for $i \in \{1,2\}$.

\item[] {\bf [Type E].} $E = \O \oplus \O(1)$.
\begin{center}
  \begin{tabular}{|c||c|c|c|}
    \hline
     & $T_{l_{1,\bullet}}$&$T_{l_{2,\bullet}} $ 
    &  there exists a subundle $\O(1) \subset E$ as below:\\
 \hline  \hline 
   $Q^{(0)}_{[t_1],[t_2]}$ &$T^{(3)}_{II} $ &$T^{(2)}_{I} $ &
 $\begin{array}{l} \text{$m_{t_1, \O(1)}=m_{t_2, \O(1)}=1$}
 \end{array}$\\
 \hline 
  $Q^{(1)}_{[t_1],[t_2]}$ &$T^{(3)}_{III} $ &$T^{(2)}_{I} $ &
 $\begin{array}{l} \text{$m_{t_1, \O(1)}=m_{t_2, \O(1)}=1$ }
 \end{array}$\\
 \hline
 $Q^{(2)}_{[t_1],[t_2]}$ &$T^{(3)}_{I} $ &$T^{(2)}_{II} $ &
 $\begin{array}{l} \text{$m_{t_1, \O(1)}=m_{t_2, \O(1)}=1$ }
 \end{array}$\\
 \hline 
  $Q^{(3)}_{[t_1],[t_2]}$ &$T^{(3)}_{II} $ &$T^{(2)}_{II} $ &
 $\begin{array}{l} \text{$m_{t_1, \O(1)}=m_{t_2, \O(1)}=1$}
 \end{array}$\\
 \hline 
  $Q^{(4)}_{[t_1],[t_2]}$ &$T^{(3)}_{III} $ &$T^{(2)}_{II} $ &
 $\begin{array}{l} \text{$m_{t_1, \O(1)}=m_{t_2, \O(1)}=1$ }
 \end{array}$\\
 \hline 
  \end{tabular}
\end{center}
and 
\begin{center}
  \begin{tabular}{|c||c|c|c|}
    \hline
     & $T_{l_{1,\bullet}}$&$T_{l_{2,\bullet}} $ 
    &  there exists a subundle $\O(1) \subset E$ as below:\\
 \hline  \hline 
  $Q^{(0)}_{2[t_1]}$ &$T^{(3)}_{I} $ &$T^{(2)}_{I} $ &
 $\begin{array}{l} \text{$m_{t_1, \O(1)}=2$}
 \end{array}$\\
 \hline
 $Q^{(1)}_{2[t_1]}$ &$T^{(3)}_{II} $ &$T^{(2)}_{I} $ &
 $\begin{array}{l} \text{$m_{t_1, \O(1)}=2$}
 \end{array}$\\
 \hline
 $Q^{(2)}_{2[t_1]}$ &$T^{(3)}_{III} $ &$T^{(2)}_{I} $ &
 $\begin{array}{l} \text{$m_{t_1, \O(1)}=2$ }
 \end{array}$\\
    \hline
       $Q^{(0)}_{2[t_2]}$ &$T^{(3)}_{I} $ &$T^{(2)}_{I} $ &
 $\begin{array}{l} \text{$m_{t_2, \O(1)}=2$}
 \end{array}$\\
 \hline
     $Q^{(1)}_{2[t_2]}$ &$T^{(3)}_{I} $ &$T^{(2)}_{II} $ &
 $\begin{array}{l} \text{$m_{t_2, \O(1)}=2$}
 \end{array}$\\
 \hline 
  \end{tabular}
\end{center}

\item[] {\bf [Type F].} 
$E = \O(-1) \oplus \O(2)$ and $T_{l_{1,\bullet}}\colon (1,1,1) \supset (1,1) \supset (1)$.

\end{itemize}

\subsection{Case $D=D_{41}$}

Now we describe special refined parabolic bundles which are undecomposable and tame
when $D=D_{41}$.

\begin{itemize}

\item[] {\bf [Type A].} $E = \O \oplus \O(1)$ and $m_{t_1,\O(1)}=m_{t_2,\O(1)}=0$.
\begin{center}
  \begin{tabular}{|c||c|c|}
    \hline
     & $T_{l_{1,\bullet}}$ 
    &  there exists a subundle $\O(-1) \subset E$ as below:\\
    \hline\hline 
    $C$ &$T^{(4)}_{I} $ 
    & $\begin{array}{l}
\text{$m_{t_1, \O(-1)} =4, m_{t_2, \O(-1)}=1$}
\end{array}$  \\
    \hline
  \end{tabular}
\end{center}

\item[] {\bf [Type B].} $E = \O \oplus \O(1)$ and $m_{t_1,\O(1)}=m_{t_2,\O(1)}=0$.
\begin{center}
  \begin{tabular}{|c||c|c|}
    \hline
     & $T_{l_{1,\bullet}}$ 
    &  there exists a subundle $\O \subset E$ as below:\\
    \hline\hline 
    $L_{[t_1],[t_2]}$ &$T^{(4)}_{I} $
     & $\begin{array}{l}
\text{$m_{t_1, \O} =3$.}
\end{array}$ \\
    \hline
    $L_{2[t_1]}$  &$T^{(4)}_{I} $
    & $\begin{array}{l}
\text{$m_{t_1, \O} =2, m_{t_{2}, \O} =1$.}
\end{array}$ \\
    \hline
  \end{tabular}
\end{center}

\item[] {\bf [Type C].}
$$
P_{[t_1]} := C \cap L_{2[t_1]},\qquad P_{[t_2]} := (C \cap L_{[t_1],[t_2]} ) \setminus P_{[t_1]}.
$$

\item[] {\bf [Type D].} $E = \O \oplus \O(1)$.
\begin{center}
  \begin{tabular}{|c||c|c|}
    \hline
     & $T_{l_{1,\bullet}}$
    &  there exists a subundle $\O(1) \subset E$ as below:\\
    \hline\hline
 $E^{(0)}_{[t_i]}$ &$T^{(4)}_{I} $ &
 $\begin{array}{l} \text{$m_{t_i, \O(1)}=1$ }
 \end{array}$\\
 \hline 
   $E^{(1)}_{[t_1]}$ &$T^{(4)}_{II} $ &
 $\begin{array}{l} \text{$m_{t_1, \O(1)}=1$}
 \end{array}$\\
 \hline 
   $E^{(2)}_{[t_1]}$ &$T^{(4)}_{III} $ &
 $\begin{array}{l} \text{$m_{t_1, \O(1)}=1$ }
 \end{array}$\\
  \hline 
   $E^{(3)}_{[t_1]}$ &$T^{(4)}_{IV} $ &
 $\begin{array}{l} \text{$m_{t_1, \O(1)}=1$ }
 \end{array}$\\
 \hline 
  \end{tabular}
\end{center}

\item[] {\bf [Type E].} $E = \O \oplus \O(1)$.
\begin{center}
  \begin{tabular}{|c||c|c|}
    \hline
     & $T_{l_{1,\bullet}}$
    &  there exists a subundle $\O(1) \subset E$ as below:\\
 \hline  \hline 
   $Q^{(0)}_{[t_1],[t_2]}$ &$T^{(4)}_{I} $ &
 $\begin{array}{l} \text{$m_{t_1, \O(1)}=m_{t_{2}, \O(1)}=1$.}
 \end{array}$\\
 \hline
  $Q^{(1)}_{[t_1],[t_2]}$ &$T^{(4)}_{II} $ &
 $\begin{array}{l} \text{$m_{t_1, \O(1)}=m_{t_{2}, \O(1)}=1$.}
 \end{array}$\\
 \hline
   $Q^{(2)}_{[t_1],[t_2]}$ &$T^{(4)}_{III} $ &
 $\begin{array}{l} \text{$m_{t_1, \O(1)}=m_{t_{2}, \O(1)}=1$.}
 \end{array}$\\
 \hline
   $Q^{(3)}_{[t_1],[t_2]}$ &$T^{(4)}_{IV} $ &
 $\begin{array}{l} \text{$m_{t_1, \O(1)}=m_{t_{2}, \O(1)}=1$.}
 \end{array}$\\
 \hline
  \end{tabular}
\end{center}
and
\begin{center}
  \begin{tabular}{|c||c|c|}
    \hline
     & $T_{l_{1,\bullet}}$
    &  there exists a subundle $\O(1) \subset E$ as below:\\
 \hline  \hline 
 $Q^{(0)}_{2[t_1]}$ &$T^{(4)}_{II} $ &
 $\begin{array}{l} \text{$m_{t_1, \O(1)}=2$}
 \end{array}$\\
 \hline
 $Q^{(1)}_{2[t_1]}$ &$T^{(4)}_{IV} $ &
 $\begin{array}{l} \text{$m_{t_1, \O(1)}=2$}
 \end{array}$\\
 \hline
   $Q^{(2)}_{2[t_1]}$ &$T^{(4)}_{V} $ &
 $\begin{array}{l} \text{$m_{t_1, \O(1)}=2$ }
 \end{array}$\\
 \hline
  $Q^{(3)}_{2[t_1]}$ &$T^{(4)}_{VI} $ &
 $\begin{array}{l} \text{$m_{t_1, \O(1)}=2$}
 \end{array}$\\
 \hline
  \end{tabular}
\end{center}

\item[] {\bf [Type F].}
 $E = \O(-1) \oplus \O(2)$ and $T_{l_{1,\bullet}}\colon (1,1,1) \supset (1,1) \supset (1)$.

\end{itemize}

\subsection{Case $D=D_{5}$}

Now we describe special refined parabolic bundles which are undecomposable and tame
when $D=D_{5}$.

\begin{itemize}

\item[] {\bf [Type A].} $E = \O \oplus \O(1)$ and $m_{t_1,\O(1)}=0$.
\begin{center}
  \begin{tabular}{|c||c|c|}
    \hline
     & $T_{l_{1,\bullet}}$ 
    &  there exists a subundle $\O(-1) \subset E$ as below:\\
    \hline\hline 
    $C$ &$T^{(5)}_{I} $ 
    & $\begin{array}{l}
\text{$m_{t_1, \O(-1)} =5$}
\end{array}$  \\
    \hline
  \end{tabular}
\end{center}

\item[] {\bf [Type B].} $E = \O \oplus \O(1)$ and $m_{t_1,\O(1)}=0$.
\begin{center}
  \begin{tabular}{|c||c|c|}
    \hline
     & $T_{l_{1,\bullet}}$ 
    &  there exists a subundle $\O \subset E$ as below:\\
    \hline\hline 
    $L_{2[t_1]}$ &$T^{(5)}_{I} $
     & $\begin{array}{l}
\text{$m_{t_1, \O} =3$.}
\end{array}$ \\
    \hline
  \end{tabular}
\end{center}

\item[] {\bf [Type C].} 
$$
P_{[t_1]} := C \cap L_{2[t_1]}
$$

\item[] {\bf [Type D].} $E = \O \oplus \O(1)$.
\begin{center}
  \begin{tabular}{|c||c|c|}
    \hline
     & $T_{l_{1,\bullet}}$
    &  there exists a subundle $\O(1) \subset E$ as below:\\
    \hline\hline
 $E^{(0)}_{[t_1]}$ &$T^{(5)}_{I} $ &
 $\begin{array}{l} \text{$m_{t_1, \O(1)}=1$ }
 \end{array}$\\
 \hline 
   $E^{(1)}_{[t_1]}$ &$T^{(5)}_{II} $ &
 $\begin{array}{l} \text{$m_{t_1, \O(1)}=1$}
 \end{array}$\\
  \hline 
  $\vdots$ &$\vdots$ &
 $\vdots$\\
 \hline  
    $E^{(4)}_{[t_1]}$ &$T^{(5)}_{V} $ &
 $\begin{array}{l} \text{$m_{t_1, \O(1)}=1$ }
 \end{array}$\\
 \hline  
  \end{tabular}
\end{center}

\item[] {\bf [Type E].} $E = \O \oplus \O(1)$.
\begin{center}
  \begin{tabular}{|c||c|c|}
    \hline
     & $T_{l_{1,\bullet}}$
    &  there exists a subundle $\O(1) \subset E$ as below:\\
    \hline\hline
 $Q^{(0)}_{2[t_1]}$ &$T^{(5)}_{V} $ &
 $\begin{array}{l} \text{$m_{t_1, \O(1)}=2$}
 \end{array}$\\
 \hline
 $Q^{(1)}_{2[t_1]}$ &$T^{(5)}_{VII} $ &
 $\begin{array}{l} \text{$m_{t_1, \O(1)}=2$}
 \end{array}$\\
 \hline
  $Q^{(2)}_{2[t_1]}$ &$T^{(5)}_{VIII} $ &
 $\begin{array}{l} \text{$m_{t_1, \O(1)}=2$}
 \end{array}$\\
 \hline
  $Q^{(3)}_{2[t_1]}$ &$T^{(5)}_{IX} $ &
 $\begin{array}{l} \text{$m_{t_1, \O(1)}=2$}
 \end{array}$\\
 \hline
    $Q^{(4)}_{2[t_1]}$ &$T^{(5)}_{X} $ &
 $\begin{array}{l} \text{$m_{t_1, \O(1)}=2$ }
 \end{array}$\\
 \hline
  \end{tabular}
\end{center}

\item[] {\bf [Type F].}
$E = \O(-1) \oplus \O(2)$ and $T_{l_{1,\bullet}}\colon (1,1,1) \supset (1,1) \supset (1)$.

\end{itemize}

\end{document}